\documentclass{jams-l}

\usepackage{pifont}
\usepackage{amssymb}
\usepackage[makeroom]{cancel}
\usepackage{graphicx}
\usepackage[cmtip,all]{xy}
\usepackage[usenames,dvipsnames]{color}
\usepackage{upgreek}
\usepackage{multirow}
\usepackage{hyperref}
\definecolor{linkcolour}{rgb}{0,0.2,0.6}
\hypersetup{colorlinks,breaklinks,urlcolor=linkcolour, linkcolor=linkcolour}
\usepackage{multicol}
\usepackage{colortbl}
\definecolor{bobby}{gray}{0}
\definecolor{LightCyan}{rgb}{0.88,1,1}
\usepackage{arydshln}
\usepackage{soul}

\newtheorem{theorem}{Theorem}[section]
\newtheorem{lemma}[theorem]{Lemma}
\newtheorem{proposition}[theorem]{Proposition}

\theoremstyle{definition}
\newtheorem{definition}[theorem]{Definition}
\newtheorem{convention}[theorem]{Convention}

\newtheorem{example}[theorem]{Example}

\theoremstyle{remark}
\newtheorem{remark}[theorem]{Remark}

\numberwithin{equation}{section}

\newcommand\dcap{\mathrel{\ooalign{\rotatebox[origin=c]{-90}{$\longrightarrow$}\cr\kern0.4ex\hbox{$\not$}}}}

\usepackage{geometry}
\begin{document}

\large

\title{A category theoretical argument for causal inference} 

\author{R\'{e}my Tuy\'{e}ras}
\address{MIT Computer Science \& Artificial Intelligence Laboratory, 32 Vassar St, Cambridge, MA 02139}
\curraddr{}
\email{rtuyeras@mit.edu}

\thanks{This research was supported by the NIH grants R01-HG008155 and R01-AG058002. Email Manolis Kellis at manoli@mit.edu}

\date{}

\dedicatory{}

\begin{abstract}
The goal of this paper is to design a causal inference method accounting for complex interactions between causal factors. The proposed method relies on a category theoretical reformulation of the definitions of dependent variables, independent variables and latent variables in terms of products and arrows in the category of unlabeled partitions. Throughout the paper, we demonstrate how the proposed method accounts for possible hidden variables, such as environmental variables or noise, and how it can be interpreted statistically in terms of $p$-values. This interpretation, from category theory to statistics, is implemented through a collection of propositions highlighting the functorial properties of ANOVA. We use these properties in combination with our category theoretical framework to provide solutions to causal inference problems with both sound algebraic and statistical properties. As an application, we show how the proposed method can be used to design a combinatorial genome-wide association algorithm for the field of genetics.
\end{abstract}

\maketitle

\section{Introduction}\label{sec:introduction}

\subsection{Background}
Causal inference is the art of recognizing whether a given observation can be the cause of another given observation \cite{Pearl,MorganWinship}. Because observations are likely to be measured with noise, it is important to take into account possible measurement defects in the analysis process. Furthermore, these defects can hide non-observable factors (usually referred to as \emph{latent variables}) that also contribute to the observation. For instance, if an observation is the result of a large combination of factors, then the measured effects may not be fully explainable through the set of considered variables. All these aspects have been extensively studied through a wide range of statistical and hypothesis testing methods. Examples of such methods are: (multiple) analysis of variance (known as (M)ANOVA) \cite{Oehlert,Warne}; generalized linear models \cite{Gelman}, including least squares analyses \cite{Oehlert,Gelman}, linear mixed models \cite{Raudenbush,Pinheiro,Zhu} and hierarchical linear models \cite{Duckworth}; instrument variables analysis such as two-stage least squares methods \cite{Hill,AngristImbens}; and structural equation modeling \cite{Loehlin}.

Unfortunately, the aforementioned methods are often limited by linear restrictions that prevent them from detecting complex interactions between causal factors; see \cite[Ch. 1, page 6]{Loehlin} and \cite[section 2]{Moore}. Specifically, these linear assumptions would limit the method to returning hints of correlations rather than clear dependency patterns. Because correlation does not imply causation, these methods also need to be supported with further discussion, analysis or analytical methods (see \cite{Shi}).
In spite of these limitations, linear-model-based methods have remained more popular than combinatorial methods due to their superior computational speed.
On the other hand, answers to current and emerging research questions seem to lie in the discovery of complex causal relationships. Because detecting these complex relationships can be a challenge for linear methods, there is now a need for causal inference methods that can go beyond the linear realm.

In this paper, we chose to diverge from the usual path of statistical analysis and use the category of unlabeled partitions to formalize and design a theoretical method that detects multivariate and combinatorial causal relationships. Since partitioning is equivalent to reading data, we expect the implementations of such an algorithm to be reasonably fast. At the end of the paper, we explain how our methods can be used to design a combinatorial GWAS for the field of genetics.

\subsection{Motivation and presentation} \label{ssec:motivations}
One of the goals of the present article is to tackle the question of designing a general combinatorial genome-wide association algorithm for the field of genetics. Before explaining what this sentence means, we shall recall what genome-wide association studies (GWASs) are and explain why these studies usually fail to be `combinatorial.' 

The basic principle of a GWAS is to identify, for a given population, the genetic mutations (called \emph{single nucleotide polymorphisms}, or SNPs) that lead to a given phenotype (\emph{e.g} a disease) within that population. Hence, the set up of a GWAS would often involve a set of individuals $S$ whose genetic information is known and a function $y:S \to \mathbb{R}$ that associates every individual in $S$ with a phenotype measurement $\mathbb{R}$. The genetic information of every individual would be modeled by a function 
\[
g:S \times \{0,1\} \times \{1,\dots,N\} \to \{\mathtt{A},\mathtt{C},\mathtt{G},\mathtt{T}\}
\]
that takes, as inputs,

1) an individual $i \in S$, possessing the corresponding genetic information;

2) a binary index $c \in \{0,1\}$, indicating which of the two chromatids\footnote{either the mother's chromatid or the father's chromatid} is looked at ;

3) an index $p$, indicating the position of the nucleotide considered in the chromatid, 

\noindent and returns the nucleotide symbol observed at position $p$ on the chromatid $c$ of the individual $i$. For chemical and statistical reasons, fixing the position $p$ would usually make the images $g(i,c,p)$ run over a set of two nucleotide symbols. As a result, if we fix the position $p$, then the word $g(i,0,p)g(i,1,p)$, call it $g_p(i)$, is likely to run over a set of four elements, as shown below.
\[
\{\mathtt{AA},\mathtt{TA},\mathtt{AT},\mathtt{TT}\}.
\]
Now, in a GWAS, we would use this type of information to partition the individuals $i \in S$ in three groups depending on whether the image $g_p(i)$ is either $\mathtt{AA}$ (first group), $\mathtt{TT}$ (third group) or either one of the pairs $\mathtt{TA}$ or $\mathtt{AT}$ (the second group). Once the individuals of $S$ are partitioned\vspace{-6pt}

\hspace{-11.9pt}\begin{minipage}[b]{0.4\linewidth}
for every position $p$, the GWAS consists in finding those partitions of $S$ that best match a non-trivial linear regression of the phenotype measurements $y:S \to \mathbb{R}$ (see on the right). The slope of the linear curve is then taken as a measurement of causation: the further from zero, the more likely the SNP located at position $p$ is the cause of the observed phenotype -- at least in theory. In a way, the goal of the linear regression is to reduce the set of phenotype values to three repre-
\end{minipage}\hspace{15pt}
\begin{minipage}[b]{0.5\linewidth}
\[
\includegraphics[width = 8.1cm]{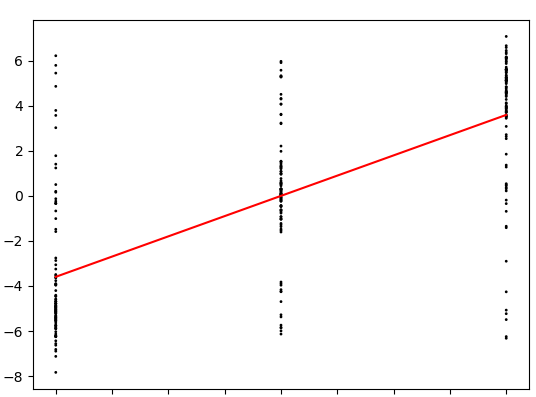}
\]
\end{minipage}
\vspace{4pt}

\noindent
sentative values, namely the values taken by the linear curve at the coordinates associated with the three groups of individuals (usually 0, 1 and 2). In essence, a GWAS compares these three representative phenotype values to the whole distribution of the phenotype values as a way to assess the quality of the partitioning of $S$. 

In this spirit, a GWAS could be seen as a study in which one tries to match each genotypic partition of $S$ to a theoretical partition of $S$ induced from the phenotype values. Although this phenotypic partition of $S$ is only theoretical, nothing prevents us from inferring this partition and actually compare it to the genotypic partitions. This line of thought can be useful if we want to go beyond additive genetic effects, which is to say those genetic effects that do not necessarily fit a linear equation. For instance, certain genetic effects, such as overdominant and underdominant effects, better fit quadratic curves, while other genetic effects, such as dominant and recessive effects, better fit exponential curves \cite{Tsepilov,Verona}. In addition to these non-additive effects, non-linearity may also occur when detecting combined effects of SNPs \cite{Corradin,Nelson,Culverhouse,Moore,Mackay}. For instance, it is not obvious to see how a partitioning induced by several positions $p_1,p_2,\dots,p_n$ can fit a linear equation. This is why GWASs and most statistical methods built around it (see \cite{Zhu,Paternoster,Shi,Loh,Yang}) cannot be combinatorial: they will usually struggle reconstructing the multivariate complex effects generated by an interaction of several mutations and their environment (see \cite{Moore,GWASChallenge}).

It is worth noting that previous works \cite{Llinarez2017,Llinarez-software} have managed to design combinatorial association methods for the study of genotype-phenotype associations. However, these methods only look at binary data and use preprocessing operations in order to reduce the complexity of the genome. Such limitations can be a problem if we want to analyze the type of data used in classical GWASs. 

Motivated by the aforementioned challenges, the present article tackles the question of making GWASs more combinatorial by proposing a translation of the usual formalism of linear regressions into a category theoretical setting  \cite{MacLane}. In order to receive validation from the biological community, we will design our method so that it is possible to interpret our results statistically in terms of $p$-values. Importantly, these $p$-values will not be generated by statistical linear regression methods, but through ANOVA \cite{Oehlert}, which is essentially a combinatorial relaxation of linear regression methods. 

While linear regressions for GWAS take the form of matrix equations 
\[
Y = \beta G + \varepsilon
\]
where $Y$ is the vector of phenotype values, $G$ is a vector of genotypes, $\beta$ is the slope of the regression and $\varepsilon$ is the error vector (see \cite{Zhu}), in our case, we translate such an analysis into the existence of an arrow 
\[
g \times \varepsilon \to y
\]
in the category of ``unlabeled partitions'' (see Definition \ref{def:unlabeled_partitions}), where the object $g$ corresponds to the genotypic partitioning of $S$, $y$ corresponds to the phenotypic partitioning of $S$ and $\varepsilon$ is a latent variable correcting for what the partition $g$ cannot explain in the partition of $y$. Hence, we are no longer solving an equation, but finding a relation of the form $g \times \varepsilon \to y$. As will be seen later, this will essentially amount to finding a partition $\varepsilon$ with good properties.

The advantage of our formalism over a usual linear regression is that the genotype vector $G$ would usually look at one SNP at a time, whereas our genotypic partition $g$ can be induced by any collection of SNPs. This difference is what makes our GWAS method combinatorial. Unfortunately, this advantage also brings its share of technicalities. Indeed, an obvious solution for the partition $\varepsilon$ is the partition $y$ itself. This means that random solutions $\varepsilon$ could well overfit the phenotypic information $y$, which is statistically not desirable. To avoid overfitting $y$, we will need to make sure that the latent variable $\varepsilon$ is minimal among the space of solutions. This minimality property will ensure that the latent variable $\varepsilon$ is only accounting for what the genotypic partition $g$ cannot explain in the partition $y$. 

The construction of latent variables $\varepsilon$ that are minimal and satisfy desirable statistical properties is the main technical achievement of this article. 
Note that even constructing minimal solutions is already a non-trivial problem in itself. For instance, we could think that if $\varepsilon_1$ and $\varepsilon_2$ are solutions for the pair $(g,y)$, then the coproduct partition $\varepsilon_1+\varepsilon_2$ (which is `smaller' than $\varepsilon_1$ and $\varepsilon_2$ in the category of unlabeled partitions) is also a solution. However, this may not be the case because the left arrow of the following canonical span is not necessarily an isomorphism.
\[
g \times (\varepsilon_1+\varepsilon_2) \leftarrow g \times \varepsilon_1+ g \times \varepsilon_2 \rightarrow y
\] 
As a result, the arrow $g \times (\varepsilon_1+\varepsilon_2) \to y$ may simply not exist. Hence, constructing a minimal solution as a copoduct of other solutions is likely to end up with no solution and cannot be used as an algorithmic principle.
Instead, we will tackle the problem by characterizing minimal solutions combinatorially (see Theorem \ref{theo:minimal-solutions}).

As an application, we will use our construction to design a general combinatorial GWAS algorithm that takes a phenotypic partition $y$ and returns a genotypic partition $g$ whose associated latent variable $\varepsilon$ is as insignificant as possible, meaning that the genotypic partition $g$ can almost (if not completely) explain the partitioning $y$ through an arrow $g \to y$. 

\subsection{Road map and results}
The goal of this article is to design a method that detects combinatorial causal relationships between a set of candidate causal variables and a set of observations. To do so, we use categorical concepts such as functoriality and universality. Our main contribution consists of two theorems, stated in Theorem \ref{theo:optimal-solutions-construction} and Theorem \ref{theo:minimal-solutions}, as well as a theoretical algorithm for combinatorial GWASs (see section \ref{ssec:combinatorial-GWAS}).

The article consists of four main sections (in addition to section \ref{sec:introduction}). First, section \ref{sec:partitions} (\emph{Partitions})  sets up the categorical formalism in which most of our results are expressed:
\begin{itemize}
\item[$\blacktriangleright$] in section \ref{ssec:labeled-partitions}, we define the category of labeled partitions $\mathbf{LP}(S)$ (Definition \ref{def:labeled-partitions}) and give a construction for its finite products (see Proposition \ref{prop:products_labeled_partitions}). In Proposition \ref{prop:product-idempotency:labeled}, we show that these products are idempotent;
\item[$\blacktriangleright$] in section \ref{ssec:unlabeled-partitions}, we construct an equivalent subcategory $\mathbf{UP}(S)$ of $\mathbf{LP}(S)$ in which all isomorphisms are identities (see Proposition \ref{lem:adjoint_equivalence} and Proposition \ref{prop:reflection_functor_partitions});
\item[$\blacktriangleright$] in section \ref{ssec:product-decomposition}, we show various combinatorial results on products in $\mathbf{UP}(S)$ and $\mathbf{LP}(S)$. These results will be substantial to the proof of our main theorems.
\end{itemize}
Then comes section \ref{sec:causal-inference} (\emph{Causal inference}), in which we define causal inference problems in terms of ``embedding problems'' and provide theorems regarding their solutions:
\begin{itemize}
\item[$\blacktriangleright$] in section \ref{ssec:embedding-problems}, we define embedding problems and their solutions. This essentially amounts to considering an arrow $a\times x \to b$ in $\mathbf{UP}(S)$ where the pair $(a,b)$ defines the embedding problem and $x$ is a solution (see Definition \ref{def:embedding-problems} and Definition \ref{def:embedding-problems:solutions}). We characterize a type of solutions said to be ``optimal'' (see Definition \ref{def:optimal-solutions}) and show that these solutions are minimal (see Proposition \ref{prop:minimality}). We finish the section with our first main result Theorem \ref{theo:optimal-solutions-construction}, which gives a recipe for constructing optimal (and hence minimal) solutions.
\item[$\blacktriangleright$] in section \ref{ssec:embedding-theorems}, we use Theorem \ref{theo:optimal-solutions-construction} to construct optimal solutions of embedding problems. This is possible thanks to Proposition \ref{prop:embedding-decomposition-III}, which allows us to satisfy all the conditions required by Theorem \ref{theo:optimal-solutions-construction} (see Proposition \ref{prop:embedding-IV}). We conclude the section with our major result, Theorem \ref{theo:minimal-solutions}, which gives us a recipe for constructing minimal solutions of embedding problems.
\end{itemize}
Then follows section \ref{sec:functorial-properties-of-ANOVA} (\emph{Functorial properties of ANOVA}), in which we unravel the functorial properties surrounding the ANOVA method:
\begin{itemize}
\item[$\blacktriangleright$] in section \ref{ssec:ANOVA-functor}, we define a category $\mathbf{Data}$ that allows us to formalize the ANOVA formalism \cite{Oehlert} in terms of a functor going from a category of unlabeled partitions to $\mathbf{Data}$ (see Proposition \ref{prop:ANOVA-functor});
\item[$\blacktriangleright$] in section \ref{ssec:between-group-mss-and-codegenerary-morphims}, we re-formalize the well-known concept of between-group mean square sums \cite[page 47]{Oehlert} (see Definition \ref{def:between-group-variance}) and show that these satisfy a certain functorial property on the category $\mathbf{Data}$ (see Proposition \ref{prop:between-group-variance:functoriality}). We take advantage of this functoriality to characterize the numerical variations of between-group mean square sums relative to all morphisms in $\mathbf{Data}$. We do so by retricting our study to a type of generating morphisms, called \emph{degeneracy morphisms} (see Definition \ref{def:codegeneracy-arrows}). Importantly, we relate codegeneracy morphisms to the construction of optimal solutions (Proposition \ref{prop:contractions-codegeneracy-morphisms}). This gives us, in section \ref{sec:association-studies}, a way to control the construction of optimal and minimal solutions statistically.
\item[$\blacktriangleright$] in section \ref{ssec:F-ratios}, we re-formalize the concept of within-group mean square sums \cite[page 47]{Oehlert} (see Definition \ref{def:within-group-variance}) and use it with that of between-group mean square sums (see Definition \ref{def:within-group-variance}) to recover the well-known concept of $\mathsf{F}$-ratio associated with ANOVA \cite[page 48]{Oehlert} (see Definition \ref{def:F-ratio}). Finally, we use the functorial properties shown in section \ref{ssec:between-group-mss-and-codegenerary-morphims} to characterize the numerical variations of $\mathsf{F}$-ratios (see Proposition \ref{prop:changes-F-ratio}).
\end{itemize}
Finally, in section \ref{sec:association-studies} (\emph{Association studies}), we use the theorems of section \ref{sec:causal-inference} and the functorial properties investigated in section \ref{sec:functorial-properties-of-ANOVA} to construct minimal solutions of embedding problems with high statistical significance. A major part of the section consists in simplifying the formalism of \cite{CTGI} and using the resulting language to express our main theorems within a genetic framework. Our main goal is to use this framework to design a general combinatorial GWAS algorithm:
\begin{itemize}
\item[$\blacktriangleright$] in section \ref{ssec:plain-segments}, we borrow from the formalism of \cite{CTGI} and introduce plain segments (see Definition \ref{def:plain-segments}), which are a simplification of the concept of segments \cite{CTGI}. We associate these plain segments with a sum operation (see Definition \ref{def:sums:segments}) that play a central role in the design of our combinatorial GWAS algorithm.
\item[$\blacktriangleright$] in section \ref{ssec:genotypes}, we use the sum operation defined in section \ref{def:plain-segments} to define a \emph{pedigrad} \cite{CTGI} in the category of unlabeled partitions (see Proposition \ref{prop:pedigrad-property:genotype-partition}). This pedigrad turns out to encompass the collection of genotypic partitions naturally associated with a GWAS (see section \ref{ssec:motivations} above).
\item[$\blacktriangleright$] in section \ref{ssec:combinatorial-GWAS}, we define the concept of combinatorial association (see Definition \ref{def:combinatorial-association}) and explain, through an example (see Example \ref{exa:combinatorial-association}) how this concept relates to combinatorial GWASs in the sense of section \ref{ssec:motivations}. Finally, we show how the pedigrad of section \ref{ssec:genotypes} can be used to construct combinatorial associations and we use the resulting procedure to design our combinatorial GWAS algorithm.
\end{itemize}

\section{Partitions}\label{sec:partitions}

\subsection{Labeled partitions}\label{ssec:labeled-partitions}
In this section, we define the category of labeled partitions $\mathbf{LP}(S)$ (Definition \ref{def:labeled-partitions}) and give a construction for its finite products (see Proposition \ref{prop:products_labeled_partitions}). In Proposition \ref{prop:product-idempotency:labeled}, we show that these products are idempotent.

We start with the definition of labeled partitions. In Convention \ref{conv:Representation}, we will see that these objects encode partitions whose parts are each labeled by an element of a finite set (see Convention \ref{conv:Representation}).

\begin{definition}[Labeled partitions]\label{def:labeled-partitions}
For every finite set $S$, we will denote by $\mathbf{LP}(S)$ the category whose objects are surjections of the form $p:S \to K$ for some set $K$ and whose arrows $p_1 \to p_2$ between two objects $p_1:S \to K_1$ and $p_2:S \to K_2$ are functions $f:K_1 \to K_2$ for which the following diagram commutes.
\[
\xymatrix{
S\ar@{=}[r]&S\\
K_1\ar[r]^{f} \ar@{<-}[u]^{p_1}&K_2\ar@{<-}[u]_{p_2}
}
\]
The composition of two arrows $f:p_1 \to p_2$ and $g:p_2 \to p_3$ corresponds to the composition of the underlying functions $f$ and $g$. An object $p:S \to K$ in $\mathbf{LP}(S)$ will be called a \emph{labeled partition of $S$} and its associated set $K$ will be called the \emph{set of labels}.
\end{definition}

The following remark will be useful for proving that two given arrows can be equal.

\begin{remark}[Pre-order category]\label{rem:pre-order-LP}
For every finite set $S$, the category $\mathbf{LP}(S)$ is a pre-order category, meaning that every pair of arrows between the same pair of objects are equal. Indeed, if $f:p_1 \to p_2$ and $g:p_1 \to p_2$ are two morphisms of $\mathbf{LP}(S)$, then the identity $f \circ p_1 = p_2 = g \circ p_1$ holds. Since $p_1$ is an epimorphism of sets, the previous identity gives the identity $f = g$. We deduce that arrows of a given type $p_1 \to p_2$ are unique in $\mathbf{LP}(S)$. As a result, the relation $p_1 \to p_2$ defines a pre-order structure on the objects of $\mathbf{LP}(S)$.
\end{remark}

Below, we give an intuitive representation for labeled partitions. This representation will be useful in our examples.

\begin{convention}[Representation]\label{conv:Representation}
For every finite set $S$, a labeled partition $p:S \to K$ in $\mathbf{LP}(S)$ will be represented by its collection of fibers (also called ``parts''). Recall that, for every $k \in K$, the \emph{fiber of $p$ above $k$} is the subset $p^{-1}(k)$ of $S$ containing the elements $x \in S$ for which the equation $p(x) = k$ holds. 
For example, a labeled partition of the form $$p:S \to \{0,\mathtt{A},\flat\}$$ will be represented by the collection $p^{-1}(0),p^{-1}(\mathtt{A}),p^{-1}(\flat)$.
Because the elements of the set $K$ are not ordered in an obvious way, we will label the fibers of $p$ with the corresponding elements of $K$ in a visible way. For example, if we take $K = \{0,\mathtt{A},\flat\}$ and $S = \{a,b,c,d,e,f\}$ such that the surjection $p$ is equipped with the fibers $p^{-1}(0) = \{b,c,e\}$, $p^{-1}(\mathtt{A}) = \{a,d\}$, and $p^{-1}(\flat) = \{f\}$, then we represent the associated partition of $p$ as $\{b,c,e\}_{0},\{a,d\}_{\mathtt{A}},\{f\}_{\flat}$.
\end{convention}

Definition \ref{def:image} and Proposition \ref{prop:factorization_system} provide a construction that will allow us to turn any function into a labeled partition. In Proposition \ref{prop:products_labeled_partitions}, we will use this construction to define products in categories of labeled partitions.

\begin{definition}[Image]\label{def:image}
For every function $f:X \to Y$, we will call the \emph{image of $f$} the subset $\mathsf{Im}(f)$ of $Y$ that contains the element $f(x)$ for every $x \in X$.
\end{definition}

\begin{proposition}[Factorization system]\label{prop:factorization_system}
Every function $f:X \to Y$ can be factorized as a composite $\mathsf{m}(f) \circ \mathsf{e}(f)$ where $\mathsf{e}(f):S \to \mathsf{Im}(f)$ is a surjection and $\mathsf{m}(f):\mathsf{Im}(f) \to Y$ is an injection of the form given below.
\[
\mathsf{e}(f):
\left\{
\begin{array}{lll}
X &\to &\mathsf{Im}(f)\\
x &\mapsto & f(x)
\end{array}
\right.
\quad\quad\quad\quad
\mathsf{m}(f):
\left\{
\begin{array}{lll}
\mathsf{Im}(f) &\to &Y\\
y &\mapsto & y
\end{array}
\right.
\]
For every other factorization $f = m \circ e$ where $e:X \to Z$ is a surjection and $m:Z \to Y$ is any morphism, there exists a unique arrow $u:Z \to \mathsf{Im}(f)$ making the following diagram commute.
\[
\xymatrix{
X\ar[r]^{\mathsf{e}(f)}\ar[d]_{e}&\mathsf{Im}(f)\ar[d]^{\mathsf{m}(f)}\\
Z\ar@{-->}[ru]^u\ar[r]_{m}&Y
}
\]
\end{proposition}
\begin{proof}
It is easy to verify that $\mathsf{e}(f)$ is a surjection; $\mathsf{m}(f)$ is an injection and that the equation $f = \mathsf{m}(f) \circ \mathsf{e}(f)$ holds. To show the universality property associated with this factorization, let us consider commutative square as given below, on the left, where $e$ is a surjection. 
\[
\xymatrix{
X\ar[r]^-{\mathsf{e}(f)}\ar[d]_{e}&*+!L(.7){\mathsf{Im}(f)}\ar[d]^{\mathsf{m}(f)}\\
Z\ar[r]_-{m}&Y
}
\] 
For every element $z$ in the $Z$, the fiber $e^{-1}(z)$ is not empty since $e$ is a surjection. For every pair of elements $t_1$ and $t_2$ in $e^{-1}(z)$, the following identity holds.
\[
f(t_1) = \mathsf{m}(\mathsf{e}(t_1)) = \mathsf{m}(z) = \mathsf{m}(\mathsf{e}(t_2)) = f(t_2)
\]
Furthermore, the identity $\mathsf{e}(f)(t_1) = \mathsf{e}(f)(t_2)$ holds. Indeed, if it did not, then because $\mathsf{m}(f)$ is an injection and $f = \mathsf{m}(f) \circ \mathsf{e}(f)$, we would have the identity $f(t_1) \neq f(t_2)$, which contradicts our assumption. Thus, the mapping $z \mapsto \mathsf{e}(f)(e^{-1}(z))$ is well-defined and provides a function $u:Z \to \mathsf{Im}(f)$  making the following diagram commute.
\[
\xymatrix{
X\ar[r]^-{\mathsf{e}(f)}\ar[d]_{e}&*+!L(.7){\mathsf{Im}(f)}\ar[d]^{\mathsf{m}(f)}\\
Z\ar[ru]^-{u}\ar[r]_-{m}&Y
}
\]
Since $e$ is an surjection (\emph{i.e.} an epimorphism), such an arrow $u$ must be unique.
\end{proof}

We can now use Proposition \ref{prop:factorization_system} to define products in categories of partitions.

\begin{proposition}[Products]\label{prop:products_labeled_partitions}
For every finite set $S$, the Cartesian product of any pair of labeled partitions $p_1:S \to K_1$ and $p_2:S \to K_2$ in $\mathbf{LP}(S)$ exists and is given by canonical surjection $\mathsf{e}(p_1,p_2):S \to \mathsf{Im}(p_1,p_2)$ induced by the Cartesian pairing $(p_1,p_2):S \to K_1\times K_2$.
\[
\xymatrix@C+10pt{
S \ar[r]^-{\mathsf{e}(p_1,p_2)}\ar@/_1.4pc/[rr]_-{(p_1,p_2)} & \mathsf{Im}(p_1,p_2) \ar[r]^-{\mathsf{m}(p_1,p_2)} & K_1\times K_2
}
\]
\end{proposition}
\begin{proof}
Let $\pi_1:K_1\times K_2 \to K_1$ and $\pi_2:K_1\times K_2 \to K_2$ denote the Cartesian structure of $K_1 \times K_2$.
To show the proposition, we will show that the span of morphisms $\pi_1 \circ \mathsf{m}(p_1,p_2):\mathsf{e}(p_1,p_2) \to p_1$ and $\pi_2 \circ \mathsf{m}(p_1,p_2):\mathsf{e}(p_1,p_2) \to p_2$ in $\mathbf{LP}(S)$ is a product in $\mathbf{LP}(S)$. Showing this statement amounts to showing that for every labeled partition $p:S \to K$ for which there exist two functions $f_1$ and $f_2$ making the following diagram commutes
\[
\xymatrix{
&S\ar[dl]_{p_1}\ar[drrr]|<<<<<<<<<<\hole^>>>>>>>>>>>>>{p_2}\ar[rr]^{p}&&K\ar[dlll]_>>>>>>>>>>>>{f_1}\ar[dr]^{f_2}&\\
K_1&&&&K_2
}
\]
there exists a unique function $f:K \to \mathsf{Im}(p_1,p_2)$ making the following diagrams commute.
\begin{equation}\label{eq:product_labeled_partitions}
\xymatrix{
S\ar@{=}[r]\ar[d]_{p}&S\ar[d]^{\mathsf{e}(p_1,p_2)}\\
K \ar[r]^-{f} &\mathsf{Im}(p_1,p_2)
}
\quad\quad\quad
\xymatrix@C+20pt{
&K\ar[dl]_-{f_1}\ar[dr]^-{f_2}\ar[d]_{f}&\\
K_1&\mathsf{Im}(p_1,p_2)\ar[l]^-{\pi_1 \circ \mathsf{m}(p_1,p_2)}\ar[r]_-{\pi_2 \circ \mathsf{m}(p_1,p_2)}&K_2
}
\end{equation}
To show this, observe that the universality of $K_1\times K_2$ gives us the Cartesian pairing $(f_1,f_2):K \to K_1 \times K_2$ making the following diagrams commute.
\begin{equation}\label{eq:product_labeled_partitions:pre}
\xymatrix{
S\ar@{=}[r]\ar[d]_{p}&S\ar[d]^{(p_1,p_2)}\\
K \ar[r]^-{(f_1,f_2)} &K_1\times K_2
}
\quad\quad\quad
\xymatrix@C+20pt{
&K\ar[dl]_-{f_1}\ar[dr]^-{f_2}\ar[d]|{(f_1,f_2)}&\\
K_1&K_1\times K_2\ar[l]^-{\pi_1}\ar[r]_-{\pi_2}&K_2
}
\end{equation}
By Propostion \ref{prop:factorization_system}, the factorization of $(p_1,p_2) = \mathsf{m}(p_1,p_2) \circ \mathsf{e}(p_1,p_2)$ is associated with a universal property, which we can use in the leftmost diagram of (\ref{eq:product_labeled_partitions:pre}). Doing so gives us a canonical arrow $f:K \to \mathsf{Im}(p_1,p_2)$ for which the equations $(f_1,f_2) = \mathsf{m}(p_1,p_2) \circ f$ and $f \circ p = \mathsf{e}(p_1,p_2)$ hold. While the latter gives us leftmost diagram of (\ref{eq:product_labeled_partitions}), we can the former in the rightmost commutative diagram of (\ref{eq:product_labeled_partitions:pre}) to obtain the rightmost commutative diagram of (\ref{eq:product_labeled_partitions}).
\end{proof}

Throughout the paper, our use of products in categories of partitions will involve a number of shuffling operations. This shuffling will particularly be important in Proposition \ref{prop:product_d:d_star}. To prepare the proof of this proposition, we discuss in Remark \ref{rem:associativity} the associativity of the products from the point of view of their encoding.

\begin{remark}[Associativity]\label{rem:associativity}
Let $S$ be a finite set. It follows from Definition \ref{def:image} and the construction given in Proposition \ref{prop:factorization_system} that, for every triple $(p_1,p_2,p_3)$ of partitions in $\mathbf{LP}(S)$, the surjection $\mathsf{e}(\mathsf{e}(p_1,p_2),p_3)$ is of the following form:
\[
\left(
\begin{array}{lll}
S &\to& \mathsf{Im}(\mathsf{e}(p_1,p_2),p_3)\\
s &\mapsto & \big(\big(p_1(s),p_2(s)\big),p_3(s)\big)
\end{array}
\right).
\]
Composing the surjection $\mathsf{e}(\mathsf{e}(p_1,p_2),p_3)$ with the bijection 
\[
\varphi:\left(
\begin{array}{lll}
\mathsf{Im}(\mathsf{e}(p_1,p_2),p_3) &\to& \mathsf{Im}(p_1,p_2,p_3)\\
((x,y),z) &\mapsto & (x,y,z)
\end{array}
\right).
\]
gives the surjection $\mathsf{e}(p_1,p_2,p_3):S \to \mathsf{Im}(p_1,p_2,p_3)$, which maps every $s \in S$ to the triple $(p_1(s),p_2(s),p_3(s))$. Hence, we obtain an isomorphism in $\mathbf{LP}(S)$ of the following form:
\[
\xymatrix{
S\ar@{=}[r]\ar[d]_{\mathsf{e}(\mathsf{e}(p_1,p_2),p_3)}&S\ar[d]^{\mathsf{e}(p_1,p_2,p_3)}\\
*+!R(.7){\mathsf{Im}(p_1,\mathsf{e}(p_2,p_3))}\ar[r]_{\cong}^{\varphi}&*+!L(.7){\mathsf{Im}(p_1,p_2,p_3)}
}
\]
Extending the previous reasoning to any finite collection $p_1,\dots, p_n$ of partitions in $\mathbf{LP}(S)$ shows that the partition $\mathsf{e}(\dots(\mathsf{e}(p_1,p_2),\dots),p_n)$ is isomorphic to the partition 
$\mathsf{e}(p_1,p_2,\dots,p_n):S \to \mathsf{Im}(p_1,p_2,\dots,p_n)$, which maps every $s \in S$ to the tuple $(p_1(s),p_2(s),\dots,p_n(s))$.
\end{remark}

The following proposition will later be used in Proposition \ref{prop:product-idempotency:unlabeled}, which will play an important in our computations. 

\begin{proposition}[Product idempotency]\label{prop:product-idempotency:labeled}
For every finite set $S$ and every labeled partition $p$ in $\mathbf{LP}(S)$, the diagonal morphism $p \to p \times p$ is a isomorphism whose inverse is the Cartesian projection $p \times p \to p$.
\end{proposition}
\begin{proof}
By Remark \ref{rem:pre-order-LP}, the compositions $p \to p \times p \to p$ and $p \times p \to p \to p \times p$ have unique representative, meaning that they must be identities. This shows the statement.
\end{proof}

\subsection{Unlabeled partitions}\label{ssec:unlabeled-partitions}
In this section, we construct, for every finite strict linear order\footnote{A transitive relation $<$ on a finite set $S$ such that for every pair $(x,y) \in S \times S$, either $x < y$, or $y< x$, or $x = y$ holds.} $(S,\leq)$, an equivalent subcategory $\mathbf{UP}(S)$ of $\mathbf{LP}(S)$ in which all isomorphisms are identities (see Proposition \ref{lem:adjoint_equivalence} and Proposition \ref{prop:reflection_functor_partitions}).

The definition of a subcategory such as $\mathbf{UP}(S)$ usually requires a strengthening of the structure of the objects and the morphisms. In our case, this strengthening will amount to replacing the set of labels associated with a partition with a canonical choice of labels. To do so, we will need to be able to define a strict linear order structure on the labels of any labeled partition.

\begin{definition}[Induced order]\label{def:induced-order}
Let $(S,<)$ be a finite strict linear order and denote by $\mathsf{min}$ the minimum operation for this linear order. For every labeled partition $p:S \to K$ in $\mathbf{LP}(S)$ and every $x,y \in K$, we write $x <_p y$ if and only if the relation $\mathsf{min}(p^{-1}(x)) < \mathsf{min}(p^{-1}(y))$ holds.
\end{definition}

It should be an easy exercise to verify that the relation defined in Definition \ref{def:induced-order} defines a strict linear order.

\begin{proposition}[Induced linear order]\label{prop:linear_order_on_K}
Let $(S,<)$ be a finite strict linear order and $p:S \to K$ be an object in $\mathbf{LP}(S)$. The relation $<_p$ is a strict linear order on $K$.
\end{proposition}
\begin{proof}
Straightforward.
\end{proof}

We now give a proposition that we later be used in Proposition \ref{prop:skeletal}. The latter will allow us to avoid many category theoretical justifications.

\begin{proposition}[Isomorphism]\label{prop:isomorphism-order}
Let $(S,<)$ be a finite strict linear order. For every isomorphism $f:p_1 \to p_2$ in $\mathbf{LP}(S)$, the underlying bijection $f$ preserves the orders defined in Definition \ref{def:induced-order}, meaning that if $x<_{p_1}y$, then $f(x)<_{p_2}f(y)$.
\end{proposition}
\begin{proof}
Suppose that the partitions $p_1$ and $p_2$ are of the form $S \to K_1$ and $S \to K_2$, respectively. In particular, the bijection $f$ is of the form $K_1 \to K_2$. Let $x$ and $y$ two elements of $K_1$ such that the inequality $x <_{p_1} y$ holds. According to Definition \ref{def:induced-order}, this is equivalent to the inequality $\mathsf{min}(p^{-1}_1(x)) < \mathsf{min}(p^{-1}_1(y))$. Because $f:K_1 \to K_2$ is an isomorphism, this inequality is also equivalent to the following one:
\[
\mathsf{min}(p^{-1}_1(f^{-1}(f(x)))) < \mathsf{min}(p^{-1}_1(f^{-1}(f(x))))
\]
Because the identity $p_2 = f \circ p_1$ holds, the previous inequality can be turned into the inequality 
\[
\mathsf{min}(p^{-1}_2(f(x))) < \mathsf{min}(p^{-1}_2(f(x))),
\]
which means that $f(x) <_{p_2} f(y)$. This shows that $f$ is an order-preserving isomorphism from the strict linear order $(K_1,<_{p_1})$ to the strict linear order $(K_2,<_{p_2})$.
\end{proof}

We now introduce a type of set that we will use to label certain of our partitions.

\begin{convention}[Canonical finite set]\label{canonical-finite-set}
For every non-negative integer $n$, we will denote by $[n]$ the set of integer comprised between 1 and $n$. The set $[n]$ will be equipped with the natural order on integers. 
\end{convention}

Below, in Convention \ref{ssec:Cardinal} and Definition \ref{def:from_labeled_to_unlabeled}, we see that we can relate the set of labels of any partition to the type of sets defined in Convention \ref{canonical-finite-set}. This transformation operation will eventually lead to a functor operation.

\begin{convention}[Cardinal]\label{ssec:Cardinal}
Let $(S,<)$ be a finite strict linear order and $p:S \to K$ be an object in $\mathbf{LP}(S)$. Because $S$ is finite and $p$ is a surjection, the set $K$ must also be finite. We will denote the finite cardinality of $K$ as $n_K$. According to Proposition \ref{prop:linear_order_on_K}, the set $K$ is equipped with a strict linear order $<_p$, which implies that there is an order-preserving bijection $\phi_K:K \to [n_K]$.
\end{convention}

\begin{definition}[Pointed structure]\label{def:from_labeled_to_unlabeled}
Let $(S,<)$ be a finite strict linear order. For every object $p:S \to K$ in $\mathbf{LP}(S)$, we will denote by $R(p):S \to [n_K]$ the composite of $p:S \to K$ with the bijection $\phi_K:K \to [n_K]$. The following commutative square gives an isomorphism $\eta_p:p \to R(p)$ in $\mathbf{LP}(S)$.
\[
\xymatrix{
S\ar[d]_{p} \ar@{=}[r]& S\ar[d]^{R(p)}\\
K \ar[r]^{\phi_p} & [n_K]
}
\]
\end{definition}

The remark given below will play an important role in the proof of Proposition \ref{prop:skeletal}.

\begin{remark}[Induced order and integer order]\label{rem:induced-order:integer-order}
Let $(S,<)$ be a finite strict linear order. A consequence of Convention \ref{ssec:Cardinal} is that for every labeled partition $p:S \to K$ in $\mathbf{LP}(S)$, the strict linear order $<_{R(p)}$ on $[n_K]$, induced by the partition $R(p):S \to [n_K]$, corresponds to the natural strict linear order defined on the integers of $[n_K]$. 
\end{remark}

\begin{remark}[Idempotency]\label{rem:idempotency}
Let $(S,<)$ be a finite strict linear order. For every object $p:S \to K$ in $\mathbf{LP}(S)$, the arrow $\eta_{R(p)}:R(p) \to R(R(p))$ is an identity. Indeed, the arrow $\eta_{R(p)}$ is induced by the order-preserving bijection $\phi_{[n_K]}$ (see Definition \ref{def:from_labeled_to_unlabeled}), which must be of the form $[n_K] \to [n_K]$ because the cardinality of $[n_K]$ is $n_K$ itself. By Remark \ref{rem:induced-order:integer-order} and Convention \ref{ssec:Cardinal}, both the source and the target of $\phi_{[n_K]}$ are equipped with the natural order of integers. As a result, the bijection $\phi_{[n_K]}$ must be the identity $[n_K] \to [n_K]$.
\end{remark}

Definition \ref{def:unlabeled_partitions} (see below) is justified by Remark \ref{rem:idempotency}. Here, the main point of Remark \ref{rem:idempotency} and Definition \ref{def:unlabeled_partitions} is that removing the labels of a partition will yield the partition that was obtained after the first removal.

\begin{definition}[Unlabeled partitions]\label{def:unlabeled_partitions}
For every finite strict linear order $(S,<)$, we will denote by $\mathbf{UP}(S)$ the full subcategory of $\mathbf{LP}(S)$ whose objects are of the form $R(p)$ (Definition \ref{def:from_labeled_to_unlabeled}) for some labeled partition $p$ in $\mathbf{LP}(S)$. An object $p:S \to [n]$ in $\mathbf{UP}(S)$ will be called an \emph{unlabeled partition of $S$}.
\end{definition}

The following convention addresses the representation of unlabeled partitions relative to labeled partitions.

\begin{convention}[Representation]
For every finite strict linear order $(S,<)$, an unlabeled partition $p:S \to [n]$ in $\mathbf{UP}(S)$ will be represented by its collection of fibers without labels. For example, if we take $p$ to be the labeled partition 
\[
\{b,c,e\}_{0},\{a,d\}_{\mathtt{A}},\{f\}_{\flat}
\]
considered in Convention \ref{conv:Representation} for the linear order $S = \{a<b<c<d<e<f\}$, then the unlabeled partition $R(p):S \to [3]$ is associated with the collection of fibers $R(p)^{-1}(1) = \{a,d\}$, $R(p)^{-1}(2) = \{b,c,e\}$, and $R(p)^{-1}(3) = \{f\}$ so that its representation is as follows.
\[
\{a,d\},\{b,c,e\},\{f\}
\]
Here, note that the order relation $<_p$ of Proposition \ref{prop:linear_order_on_K} re-orders the fibers $p$ according to their least elements contained in the fibers. Above, the fiber $\{a,d\}$ appeared before the fiber $\{b,c,e\}$ because the least element $a$ of $\{a,d\}$ is less than the least element $b$ of $\{b,c,e\}$. 
\end{convention}

The following proposition shows that categories of unlabeled partitions are \emph{skeletal}, meaning that all isomorphisms are between identical objects. This result constitutes one of the main reasons why categories of unlabeled partitions are better settings than categories of labeled partitions.

\begin{proposition}[Skeletal]\label{prop:skeletal}
Let $(S,<)$ be a finite strict linear order. Every isomorphism in $\mathbf{UP}(S)$ is an identity.
\end{proposition}
\begin{proof}
Let $f:p_1 \to p_2$ be an isomorphism in $\mathbf{UP}(S)$ where we have $p_1:S \to [n_1]$ and $p_2:S \to [n_2]$. Since $[n_2]$ and $[n_1]$ are in bijection through the underlying function $f$, the cardinalities of $[n_2]$ and $[n_1]$ must be equal, meaning that the identity $n_1 = n_2$ holds. According to Proposition \ref{prop:isomorphism-order}, the bijection $f:[n_1] \to [n_2]$ is order-preserving and because the source and target of $f$ are equal and equipped with the same order structure (see Remark \ref{rem:induced-order:integer-order}), the bijection $f$ must be an identity.
This proves that the isomorphism $f:p_1 \to p_2$ is an identity in $\mathbf{UP}(S)$.
\end{proof}

Lemma \ref{lem:from_operation_to_functor} and its application to categories of partitions (see Proposition \ref{prop:reflection_functor_partitions}) will allow us to transfer our reasonings in categories of labeled partitions to categories of unlabeled partitions.

\begin{lemma}[Functor]\label{lem:from_operation_to_functor}
Let $\mathcal{C}$ be a category in which every object $X$ is equipped with an object $R(X)$ in $\mathcal{C}$ and an isomorphism $\eta_X:X \to R(X)$. Then the mapping $X \mapsto R(X)$ extends to a functor $\mathcal{C} \to \mathcal{C}$ by mapping every arrow $f:X \to Y$ in $\mathcal{C}$ to the arrow $R(X) \to R(Y)$ defined by the following composite:
\[
\xymatrix{
R(X) \ar[r]^-{\eta_X^{-1}} & X \ar[r]^{f} & Y \ar[r]^-{\eta_Y} & R(Y).
}
\]
In addition, the arrow $\eta_X:X \to R(X)$ induces a natural transformation $\mathrm{id}_{\mathcal{C}} \Rightarrow R$.
\end{lemma}
\begin{proof}
Straightforward.
\end{proof}

Lemma \ref{lem:from_operation_to_functor} can be applied to the isomorphism $\eta_p:p \to R(p)$ of Definition \ref{def:from_labeled_to_unlabeled} to construct a functor turning labeled partitions into unlabeled partitions. As will be seen in the following results, this functor is at the same time a reflection (as a consequence of  Proposition \ref{prop:reflection_functor_partitions} and Lemma \ref{lem:adjoint_equivalence}) and an equivalence (as a consequence of Proposition \ref{prop:reflection_functor_partitions} and Lemma \ref{lem:from_operation_to_functor}).

\begin{proposition}[Reflection]\label{prop:reflection_functor_partitions}
Let $(S,<)$ be a finite strict linear order. The functor $R:\mathbf{LP}(S) \to \mathbf{LP}(S)$ induced by Lemma \ref{lem:from_operation_to_functor} lifts to an obvious functor $R:\mathbf{LP}(S) \to \mathbf{UP}(S)$.
\end{proposition}
\begin{proof}
Follows from Definition \ref{def:unlabeled_partitions}.
\end{proof}

The following lemma states that categories of unlabeled partitions are very similar to their associated categories of labeled partitions. 

\begin{lemma}[Adjoint equivalence]\label{lem:adjoint_equivalence}
Let $(S,<)$ be a finite strict linear order. For every unlabeled partition $p$ in $\mathbf{UP}(S)$ and labeled partition $q$ in $\mathbf{LP}(S)$, the post-composition operation by the natural isomorphism $\eta_q:q \to R(q)$ gives the following isomorphism, natural in $p$ and $q$.
\[
\begin{array}{clc}
\mathbf{LP}(S)(p,q) &\cong& \mathbf{UP}(S)(p,R(q))\\
f & \mapsto & \eta_q \circ f 
\end{array}
\]
\end{lemma}
\begin{proof}
Follows from the naturality of the isomorphism $\eta_q:q \to R(q)$ (Lemma \ref{lem:from_operation_to_functor}).
\end{proof}

Since categories of labeled partitions have products, so do categories of unlabeled partitions. This is shown in the following proposition.

\begin{proposition}[Products]\label{prop:adjoint_equivalence:products}
For every finite strict linear order $(S,<)$, the category $\mathbf{UP}(S)$ admits binary products, which are the images of the products in $\mathbf{LP}(S)$ through the functor $R$.
\end{proposition}
\begin{proof}
We use the definition of limits in terms of adjunctions. The proposition follows from the following series of natural isomorphisms for every triple of unlabeled partitions $p$, $p_1$ and $p_2$.
\begin{align*}
\mathbf{UP}(S)^{\times 2}((p,p),(p_1,p_2))& \cong \mathbf{LP}(S)^{\times 2}((p,p),(p_1,p_2))&(\textrm{Full subcategory})\\
& \cong \mathbf{LP}(S)(p,p_1 \times p_2) & (\textrm{Lemma \ref{prop:products_labeled_partitions}})\\
& \cong \mathbf{UP}(S)(p,R(p_1 \times p_2)) & (\textrm{Lemma \ref{lem:adjoint_equivalence}})
\end{align*}
The object $R(p_1 \times p_2)$ is therefore the product of $p_1$ and $p_2$ in $\mathbf{UP}(S)^{\times 2}$.
\end{proof}

Products in categories of unlabeled partitions take finite collections of partitions and return a universal refinement of these, meaning that the returned partition will contain smaller parts than the parts contained by the partitions used in the product. This is illustrated below, in Example \ref{exa:product_of_partitions}.

\begin{example}[Products of partitions]\label{exa:product_of_partitions}
Let $S$ denote the set $\{a,b,c,d,e,f,g\}$, which will be considered linearly ordered with respect to its present specification. The product of two unlabeled partitions of $S$ is the \emph{refinement} of the bracketing defining the partitions, that is to say the partition containing all the possible intersections between the different parts of the two partitions. For instance, the product of the two partitions given below, at the top of the deduction rule, is given by the partition displayed at the bottom.
\[
\xymatrix@R-25pt@C-20pt{
&\{a,b,c\},\{d,e\},\{f,g\}\quad\quad\quad\quad\{a,e\},\{b,c,d,g\},\{f\}&\ar@/^10pt/[dd]^{\prod}\\
\ar@{-}[rr]&&\\
&\{a\},\{b,c\},\{d\},\{e\},\{f\},\{g\}&
}
\]
\end{example}

The following proposition makes categories of unlabeled partitions places in which it is easy to do calculations.

\begin{proposition}[Product idempotency]\label{prop:product-idempotency:unlabeled}
For every finite strict linear order $(S,<)$ and every unlabeled partition $p$ in $\mathbf{UP}(S)$, the identity $p = p \times p$ holds.
\end{proposition}
\begin{proof}
By Proposition \ref{prop:product-idempotency:labeled}, the diagonal $p \to p \times p$ in $\mathbf{UP}(S)$ is an isomorphism. Since $p$ is an unlabeled partition, it is an image of $R$ (Definition \ref{def:unlabeled_partitions}) and applying the idempotent functor $R$ on the isomorphism $p \to p \times p$ gives an isomorphism $p \to p \times p$ in $\mathbf{UP}(S)$. The statement follows from Proposition \ref{prop:skeletal}.
\end{proof}

\subsection{Product decompositions}\label{ssec:product-decomposition}
In this section, we show various combinatorial results on products in $\mathbf{UP}(S)$ and $\mathbf{LP}(S)$. Our most general results are Proposition \ref{prop:triple-xyz-delta}, Proposition \ref{prop:delta-complement-fiber} and Proposition \ref{prop:union-property:delta}, which will play a key role in the proof of our main theorems, shown in section \ref{sec:causal-inference}.
 
In the following convention, we introduce a partition that will play a nice intermediate in many of our calculations. We give various examples of these in Example \ref{exa:relative_discrete_partitions}. In Convention \ref{conv:relative_discrete_partitions:familial}, we will introduce a more general object that can be described as a product of these intermediates.

\begin{convention}[Atomic partitions]\label{conv:relative_discrete_partitions}
Let $S$ be a finite set. For every subset $x \subseteq S$, let $K_{x}$ denote the set that contains the set $x$ if $x$ is non-empty and its complement $S\backslash x := \{t \in S~|~t \notin x\}$ if $x$ is not $S$. In other words, we have the following identities:
\[
K_x = \left\{
\begin{array}{ll}
\{x,S\backslash x\}&\textrm{if }x \neq \emptyset\textrm{ and }S\backslash x \neq \emptyset\\
\{x\}&\textrm{if }S\backslash x = \emptyset\\
\{S\backslash x\}&\textrm{if }x = \emptyset
\end{array}
\right.
\]
We will denote by $d(x)$ the obvious surjection $S \to K_{x}$ that maps every $t \in x$ to the set $x$ in $K_x$ and that maps every $t\notin x$ to the set $S\backslash x$ in $K_x$.
\end{convention}

\begin{example}[Atomic partitions]\label{exa:relative_discrete_partitions}
Let $S$ be the set $\{a,b,c,d,e,f\}$. For every non-empty subset $x \subseteq S$, the partition $d(x)$ possesses $x$ as one of its parts while its other part may exist or not depending on whether $x$ equals $S$. For example, the partition $d(\{a,b,e\})$ takes the form
\[
\{a,b,e\}_{a,b,e}, \{c,d,f\}_{c,d,f}
\] 
while the partitions $d(S)$ and $d(\emptyset)$ are equal to the terminal partition $\{a,b,e,c,d,f\}_{S}$.
\end{example}

\begin{remark}[Duality]\label{rem:duality}
Let $S$ be a finite set and let $x$ be a subset of $S$. It directly follows from Definition \ref{conv:relative_discrete_partitions} that the identity $d(x) = d(S \backslash x)$ holds.
\end{remark}

Below, we introduce a generalization of Convention \ref{conv:relative_discrete_partitions}. In Proposition \ref{prop:product_d:d_star}, we relate the two conventions through an isomorphism involving products.

\begin{convention}[Familial partitions]\label{conv:relative_discrete_partitions:familial}
Let $S$ be a finite set. For every finite collection $x= \{x_i\}_{i \in [n]}$ of pairwise disjoint subsets $x_i$ of $S$, we denote by $K_{x}$ the set that contains every non-empty set $x_i$ of $x$ as well as the complement 
\[
S\backslash x := \{t \in S~|~t \notin \bigcup_{i = 1}^n x_i\}
\]
if $\bigcup_{i = 1}^n x_i$ is not equal to $S$. We will denote by $d^*(x)$ the obvious surjection $S \to K_{x}$ that maps every $t \in x_i$ to the set $x_i$ in $K_x$ and that maps every $t\notin \bigcup_{i = 1}^n x_i$ to the set $S \backslash x$ in $K_x$.
\end{convention}

\begin{proposition}[Products and familial partitions]\label{prop:product_d:d_star}
Let $S$ be a finite set. For every finite collection $x= \{x_i\}_{i \in [n]}$ of pairwise disjoint subsets $x_i$ of $S$, there is an isomorphism as shown in (\ref{eq:products-and-familial-partitions}) living in $\mathbf{LP}(S)$.
\begin{equation}\label{eq:products-and-familial-partitions}
\big(\dots\big(d(x_1) \times d(x_2)\big) \times \dots \big)\times d(x_n)\Big) \to d^*(x)
\end{equation}
\end{proposition}
\begin{proof}
According to Remark \ref{rem:associativity}, the product shown in the left-hand side of (\ref{eq:products-and-familial-partitions}) is isomorphic to the labeled partition
\[
\mathsf{e}(d(x_1), d(x_2), \dots, d(x_n)):\left(
\begin{array}{lll}
S &\to& \mathsf{Im}(d(x_1), d(x_2), \dots, d(x_n))\\
s &\mapsto & (d(x_1)(s),d(x_2)(s),\dots,d(x_n)(s))
\end{array}
\right).
\]
Since the elements $x_1,x_2,\dots,x_n$ are all pairwise disjoint, the image of the underlying surjection $\mathsf{e}(d(x_1), d(x_2), \dots, d(x_n))$ at an element $s \in S$ is equal to the following tuple:
\[
\left\{
\begin{array}{ll}
(S \backslash x_1,\dots,S \backslash x_{i-1},x_i,S \backslash x_{i+1},\dots,S \backslash x_{n})&\textrm{if }s \in x_i\textrm{ for }i \in [n];\\
(S \backslash x_1,\dots,S \backslash x_{n})&\textrm{if }s \in S \backslash x.\\
\end{array}
\right.
\]
Because $\mathsf{e}(d(x_1), d(x_2), \dots, d(x_n))$ is a surjection, the previous elements describe all the elements of the set $\mathsf{Im}(d(x_1), d(x_2), \dots, d(x_n))$. Hence, there is a bijection
\[
\varphi:
\left(
\begin{array}{lll}
\mathsf{Im}(d(x_1), d(x_2), \dots, d(x_n)) &\to& K_x\\
(S \backslash x_1,\dots,S \backslash x_{i-1},x_i,S \backslash x_{i+1},\dots,S \backslash x_{n}) &\mapsto & x_i\\
(S \backslash x_1,\dots,S \backslash x_{n})& \mapsto & S \backslash x
\end{array}
\right).
\]
making the following diagram commute.
\[
\xymatrix{
S\ar@{=}[r]\ar[d]_{\mathsf{e}(d(x_1), d(x_2), \dots, d(x_n))}&S\ar[d]^{d^*(x)}\\
*+!R(.7){\mathsf{Im}(d(x_1), d(x_2), \dots, d(x_n))} \ar[r]^{\varphi}_{\cong} &K_x
}
\]
This last commutative diagram encodes an isomorphism in $\mathbf{LP}(S)$, so the statement follows from the stability of isomorphisms under the composition operation.
\end{proof}

In the convention given below, we refine the operations defined in Convention \ref{conv:relative_discrete_partitions:familial} and Convention \ref{conv:relative_discrete_partitions} to categories of unlabeled partitions.

\begin{convention}[Delta partitions]\label{conv:relative_discrete_partitions_delta}
Let $(S,<)$ be a finite strict linear order. For every subset $x$ of $S$, we will denote by $\delta(x)$ the image of the labeled partition $d(x)$ through the functor $R:\mathbf{LP}(S) \to \mathbf{UP}(S)$ (Definition \ref{prop:reflection_functor_partitions}). Similarly, for every finite collection $x= \{x_i\}_{i \in [n]}$ of pairwise disjoint subsets $x_i$ of $S$, we will denote by $\delta^*(x)$ the unlabeled partition $d^*(x)$ through the functor $R:\mathbf{LP}(S) \to \mathbf{UP}(S)$.
\end{convention}

The following proposition is a translation of Proposition \ref{prop:product_d:d_star} in the context of unlabeled partitions. We will use this equation in the proof of Proposition \ref{prop:fibers}.

\begin{proposition}[Products and familial partitions]\label{prop:familial-product:unlabeled-partitions}
Let $S$ be a finite set. For every finite collection $x= \{x_i\}_{i \in [n]}$ of pairwise disjoint subsets $x_i$ of $S$, the following identity holds in $\mathbf{LP}(S)$.
\[
\delta(x_1) \times \delta(x_2) \times \dots \times \delta(x_n) = \delta^*(x)
\]
\end{proposition}
\begin{proof}
Directly follows from applying the functor $R$ on the isomorphism of Proposition \ref{prop:product_d:d_star} and using Proposition \ref{prop:skeletal}.
\end{proof}

\begin{definition}[Product decompositions]\label{def:product-decompositions}
Let $(S,<)$ be a finite strict linear order and let $p:S \to [n]$  be an unlabeled partition in $\mathbf{UP}(S)$. A \emph{product decomposition for $p$} is a non-empty finite collection $x = \{x_i\}_{i \in [r]}$ of subsets $x_i$ of $S$ for which the following equation holds:
\[
p = \delta(x_1) \times \delta(x_2) \times \dots \times \delta(x_n)
\]
\end{definition}

\begin{convention}[Fibers]\label{conv:fibers}
Let $(S,<)$ be a finite strict linear order. For every unlabeled partition $p:S \to [n]$ in $\mathbf{UP}(S)$, we will denote the collection of fibers $\{p^{-1}(i)\}_{i \in [n]}$ as $\mathsf{fib}(p)$.
\end{convention}

The following proposition provides a combinatorial description of unlabeled partitions. This result will be useful for our future calculations.

\begin{proposition}[Fibers]\label{prop:fibers}
Let $(S,<)$ be a finite strict linear order. For every unlabeled partition $p:S \to [n]$ in $\mathbf{UP}(S)$, the collection of fibers $\mathsf{fib}(p)$ is a product decomposition for $p$, meaning that the following identity holds:
\[
p = \delta^*(\mathsf{fib}(p))
\]
\end{proposition}
\begin{proof}
Since the collection $\mathsf{fib}(p) = \{p^{-1}(i)\}_{i \in [n]}$ is pairwise disjoint, the partition $d^*(\mathsf{fib}(p)):S \to K_{\mathsf{fib}(p)}$ is well defined. Because the set $S \backslash \mathsf{fib}(p)$ is empty and all the fibers $p^{-1}(i)$ of the surjection $p:S \to [n]$ are non-empty, we can define the following bijection (see Definition \ref{conv:relative_discrete_partitions:familial}):
\[
\varphi:\left(
\begin{array}{lll}
K_{\mathsf{fib}(p)}&\to&[n]\\
p^{-1}(i)& \mapsto &i
\end{array}
\right)
\]
Because the composite $\varphi \circ d^*(\mathsf{fib}(p)):S \to [n]$ maps every $s \in S$ to the integer $\varphi(p^{-1}(i)) = i$ whenever $s \in p^{-1}(s)$, the identity $\varphi \circ d^*(\mathsf{fib}(p))(s) = p(s)$ holds, which gives the following commutative diagram.
\[
\xymatrix{
S\ar@{=}[r]\ar[d]_{d^*(\mathsf{fib}(p))}&S\ar[d]^{p}\\
K_{\mathsf{fib}(p)} \ar[r]^-{\varphi}_-{\cong} &[n]
}
\]
Since this diagram defines an isomorphism $d^*(\mathsf{fib}(p)) \to p$ in $\mathbf{LP}(S)$ and $p$ is an object of $\mathbf{UP}(S)$ (see Definition \ref{def:unlabeled_partitions}), we can apply the idempotent functor $R$ on the isomorphism $d^*(\mathsf{fib}(p)) \to p$ to obtain the identities $\delta^*(\mathsf{fib}(p)) = R(p) = p$ in $\mathbf{LP}(S)$ (see Proposition \ref{prop:skeletal}).
\end{proof}

The rest of the section provides a collection of results that we will use in our calculations. These results often generalizes relationships holding for sets.

\begin{proposition}[Product rearrangement]\label{prop:product-rearrangement}
Let $(S,<)$ be a finite strict linear order. For every pair $(x,y)$ of subsets of $S$, the following equation holds in $\mathbf{UP}(S)$.
\[
\delta(x) \times \delta(y) = \delta(y \backslash x) \times \delta(x \backslash y) \times \delta(y \cap x)
\]
\end{proposition}
\begin{proof}
If $y$ is empty, then the equation of the statement is equivalent to the equation:
\[
\delta(x) \times \delta(\emptyset) = \delta(\emptyset) \times \delta(x) \times \delta(\emptyset),
\]
Since $\delta(\emptyset)$ is terminal in $\mathbf{UP}(S)$ (see Example \ref{exa:relative_discrete_partitions}) and the isomorphisms of $\mathsf{UP}(S)$ are identities (Proposition \ref{prop:skeletal}), it is straightforward to see that the two sides of the previous equation are equal to $\delta(x)$, which shows the statement. Similarly, the case where $x$ is empty is straightforward.

If $y$ is equal to $S$, then the equation of the statement is equivalent to 
\[
\delta(x) \times \delta(S) = \delta(S \backslash x) \times \delta(x) \times \delta(x).
\]
By Remark \ref{rem:duality} and Proposition \ref{prop:product-idempotency:unlabeled}, we have $\delta(x) = \delta(S \backslash x)$ and since $\delta(x) \times \delta(x) = \delta(x)$, it is straightforward to see that the two sides of the previous equation are equal to $\delta(x)$, which shows the statement. Similarly, the case where $x$ is equal to $S$ is straightforward. 

Suppose that $x$ and $y$ are not empty or equal to $S$.
Let us first describe the fibers of the product partition $\delta(x) \times \delta(y)$ -- we will then conclude with Proposition \ref{prop:familial-product:unlabeled-partitions}. First, let us describe the associated labeled partition $d(x) \times d(y)$ in $\mathbf{LP}(S)$. By Proposition \ref{prop:products_labeled_partitions}, the product $d(x) \times d(y)$ is isomorphic to the labeled partition
\[
\mathsf{e}(d(x), d(y)):\left(
\begin{array}{lll}
S &\to& \mathsf{Im}(d(x), d(y))\\
s &\mapsto & (d(x)(s),d(y)(s))
\end{array}
\right).
\]
By Convention \ref{conv:relative_discrete_partitions:familial}, because $x$ and $y$ are not empty or equal to $S$, the previous function is equal to the following one.
\[
\left(
\begin{array}{llll}
S &\to& \mathsf{Im}(d(x), d(y))&\\
s &\mapsto & (x,y)&\textrm{if }s \in x \cap y\\
s &\mapsto & (x,S\backslash y)&\textrm{if }s \in x\textrm{ and }s \notin y\\
s &\mapsto & (S\backslash x,y)&\textrm{if }s \notin x\textrm{ and }s \in y\\
\end{array}
\right).
\]
In other words, the fibers of $d(x) \times d(y)$ are $x\backslash y$, $x \cap y$ and $y \backslash x$. It follows from Proposition \ref{prop:fibers} that $\delta(x) \times \delta(y)$, which is the image of the partition $d(x) \times d(y)$ through the functor $R$ (Proposition \ref{prop:adjoint_equivalence:products}), is equal to $\delta^*(\{x\backslash y,x \cap y,y \backslash x\})$. Because the sets $x\backslash y$, $x \cap y$ and $y \backslash x$ are pairwise disjoint, we deduce that the equation of the statement holds from Proposition \ref{prop:familial-product:unlabeled-partitions}.
\end{proof}

\begin{proposition}\label{prop:delta:change-backlash-to-cap}
Let $(S,<)$ be a finite strict linear order. For every pair $(x,y)$ of subsets of $S$, the following equation holds in $\mathbf{UP}(S)$.
\[
\delta(x) \times \delta(x \backslash y) = \delta(x) \times \delta(x \cap y)
\]
\end{proposition}
\begin{proof}
By Proposition \ref{prop:product-rearrangement}, we have the following two formulas:
\[
\begin{array}{lll}
\delta(x) \times \delta(x \backslash y) &= &\delta((x \backslash y) \backslash x) \times \delta(x \backslash (x \backslash y)) \times \delta((x \backslash y) \cap x)\\
\delta(x) \times \delta(x \cap y) &=& \delta(x \cap (y \backslash x)) \times \delta(x \backslash (x \cap y)) \times \delta(x \cap y \cap x)
\end{array}
\]
Simplifying the previous equations gives the following pair of relations:
\[
\begin{array}{lll}
\delta(x) \times \delta(x \backslash y) &= &\delta(\emptyset) \times \delta(x \cap y) \times \delta(x \backslash y)\\
\delta(x) \times \delta(x \cap y) &=& \delta(\emptyset) \times \delta(x \backslash y) \times \delta(x \cap y)
\end{array}
\]
By symmetry of products and Proposition \ref{prop:product-idempotency:unlabeled}, the left-hand sides of the two previous equations are equal, which proves the statement.
\end{proof}

\begin{proposition}\label{prop:triple-xyz-delta}
Let $(S,<)$ be a finite strict linear order. For every triple $(x,y,z)$ of subsets of $S$ such that $x \cap z = \emptyset$, the following equations hold in $\mathbf{UP}(S)$.
\[
\begin{array}{llllllll}
\delta(x) \times \delta(y \backslash z) &=& \delta(x  \backslash y) &\times& \delta(x \cap y)&\times &\delta(y \backslash (z \cup x))\\
\delta(x) \times \delta(y \backslash z) &=& \delta(x) &\times& \delta(x \cap y)&\times &\delta(y \backslash (z \cup x))
\end{array}
\]
\end{proposition}
\begin{proof}
By Proposition \ref{prop:product-rearrangement}, we have the following formula:
\[
\delta(x) \times \delta(y \backslash z) = \delta((y \backslash z) \backslash x) \times \delta(x \backslash (y \backslash z)) \times \delta((y \backslash z) \cap x)
\]
Simplifying the previous equations gives the following pair of relations:
\begin{align*}
\delta(x) \times \delta(y \backslash z) &= \delta(y \backslash (z \cup x)) \times \delta((x \backslash y) \cup (x \cap z)) \times \delta(y \cap x)&\\
& =   \delta(y \backslash (z \cup x)) \times \delta(x \backslash y) \times \delta(y \cap x) & (x \cap z = \emptyset)
\end{align*}
By symmetry of products and Proposition \ref{prop:product-idempotency:unlabeled}, we obtain the topmost formula given in the statement. The show the other formula, we multiply the previous one by $\delta(x)$ and use Proposition \ref{prop:delta:change-backlash-to-cap} as shown below.
\begin{align*}
\delta(x) \times\delta(x) \times \delta(y \backslash z) &= \delta(x) \times\delta(x  \backslash y) \times \delta(x \cap y)\times \delta(y \backslash (z \cup x))\\
 &= \delta(x) \times\delta(x \cap y) \times \delta(x \cap y)\times \delta(y \backslash (z \cup x))
\end{align*}
Finally, we conclude by using Proposition \ref{prop:product-idempotency:unlabeled} on both sides of the previous equation to eliminate repeated terms.
\end{proof}

Our next result (Proposition \ref{prop:delta-complement-fiber}) will involve a long product of partitions. The following convention recalls the usual notation for this type of products.

\begin{convention}[Long products]
From now on, we will denote sequences of Cartesian products in $\mathbf{UP}(S)$ by using the conventional operation symbol  $\prod$, namely for every collection of objects $p_1,\dots,p_n$, we have the relation $\prod_{i=1}^n p_i := p_1 \times \dots \times p_n$.
\end{convention}

\begin{proposition}[Complement property]\label{prop:delta-complement-fiber}
Let $(S,<)$ be a finite strict linear order. For every subset $x$ of $S$ and every sequence of pairwise disjoint subsets $y_1,\dots,y_n$ of $x$, the following equation holds.
\[
\delta(x) \times \prod_{i=1}^n\delta(y_i) = \delta(x\backslash \bigcup_{i = 1}^n y_i) \times \prod_{i=1}^n\delta(y_i)
\]
\end{proposition}
\begin{proof}
We show the proposition by induction on $n$. Suppose that $n = 1$. It follows from the second equation of Proposition \ref{prop:triple-xyz-delta} that the following equation holds:
\[
\delta(y_1) \times \delta(x \backslash \emptyset) = \delta(y_1) \times \delta(y_1 \cap x) \times\delta(x \backslash (\emptyset \cup y_1)) 
\]
Because the inclusion $y_1 \subseteq x$ holds, the previous equation gives us the equations:
\begin{align*}
\delta(y_1) \times \delta(x) &=  \delta(y_1) \times \delta(y_1) \times\delta(x \backslash y_1)& \\
&=  \delta(y_1) \times\delta(x \backslash y_1).& (\textrm{Proposition \ref{prop:product-idempotency:unlabeled}})
\end{align*}
This shows the base step of our induction. Let us now suppose that the following equation holds.
\[
\delta(x) \times \prod_{i=1}^{n-1}\delta(y_i) = \delta(x\backslash \bigcup_{i = 1}^{n-1} y_i) \times \prod_{i=1}^{n-1}\delta(y_i)
\]
Multiplying the two sides of the previous equation gives us the following relation:
\begin{equation}\label{eq:inductive-step:delta-complement-fiber}
\delta(y_n)  \times \delta(x) \times \prod_{i=1}^{n-1}\delta(y_i) = \delta(y_n)  \times \delta(x\backslash \bigcup_{i = 1}^{n-1} y_i) \times \prod_{i=1}^{n-1}\delta(y_i)
\end{equation}
It follows from the second equation of Proposition \ref{prop:triple-xyz-delta} that the following equation holds:
\[
\delta(y_n) \times \delta(x \backslash \bigcup_{i = 1}^{n-1} y_i) = \delta(y_n) \times \delta(y_n \cap x) \times\delta(x \backslash \bigcup_{i = 1}^{n} y_i) 
\]
Because the inclusion $y_n \subseteq x$ holds, the previous equation gives us the equations:
\begin{align*}
\delta(y_n) \times \delta(x \backslash \bigcup_{i = 1}^{n-1} y_i) &=  \delta(y_n) \times \delta(y_n) \times\delta(x \backslash \bigcup_{i = 1}^{n} y_i)& \\
&=  \delta(y_n) \times\delta(x \backslash \bigcup_{i = 1}^{n} y_i).& (\textrm{Proposition \ref{prop:product-idempotency:unlabeled}})
\end{align*}
Hence, equation (\ref{eq:inductive-step:delta-complement-fiber}) can be rearranged as follows:
\[
\delta(y_n)  \times \delta(x) \times \prod_{i=1}^{n-1}\delta(y_i) = \delta(y_n)  \times \delta(x\backslash \bigcup_{i = 1}^{n} y_i) \times \prod_{i=1}^{n-1}\delta(y_i)
\]
Finally, we obtain the equation of the statement from the previous one by inserting the term $\delta(y_n)$ inside the products. 
\end{proof}

The property stated in Proposition \ref{prop:union-property:delta} (see below) contains the essence of our main theorems, which, in a sense, are refinements of the equation given by Proposition \ref{prop:union-property:delta}.

\begin{proposition}[Union property]\label{prop:union-property:delta}
Let $(S,<)$ be a finite strict linear order. For every triple $(x,y,z)$ of subsets of $S$ such that $y \subseteq x$ and $x \cap z = \emptyset$, the following equation holds in $\mathbf{UP}(S)$.
\[
\delta(x) \times \delta(y) \times \delta(z) = \delta(x) \times \delta(y \cup z)
\]
\end{proposition}
\begin{proof}
By Proposition \ref{prop:product-rearrangement}, we have the following equation:
\[
\delta(x) \times \delta(y \cup z) = \delta(x\backslash (y \cup z)) \times \delta((y \cup z) \backslash x) \times \delta((y \cup z) \cap x)
\]
Since $y \subseteq x$ and $x \cap z = \emptyset$, the previous identity can be simplified as follows:
\[
\delta(x) \times \delta(y \cup z) = \delta(x\backslash y) \times \delta(z) \times \delta(y).
\]
If we multiply the previous equation by $\delta(x)$, we can use Proposition \ref{prop:delta:change-backlash-to-cap} to deduce the following identities:
\begin{align*}
\delta(x) \times \delta(x) \times \delta(y \cup z) &= \delta(x) \times \delta(x\backslash y) \times \delta(z) \times \delta(y)&\\
&= \delta(x) \times \delta(x\cap y) \times \delta(z) \times \delta(y)& (\textrm{ Proposition \ref{prop:delta:change-backlash-to-cap}})\\
&= \delta(x) \times \delta(y) \times \delta(z) \times \delta(y)& (y \subseteq x)
\end{align*}
After using Proposition \ref{prop:product-idempotency:unlabeled} on the two sides of the last equation, to reduce the extra terms $\delta(x)$ and $\delta(y)$, we obtain the equation of the statement.
\end{proof}

\section{Causal inference}\label{sec:causal-inference}

\subsection{Embedding problems}\label{ssec:embedding-problems}
In this section, we study causal relationships through the concepts of embedding problems and their solutions. This formalization essentially amounts to considering an arrow $a\times x \to b$ in $\mathbf{UP}(S)$ where the pair $(a,b)$ defines the embedding problem, formalizing the potential causal relationship from $a$ to $b$, and where $x$ is a solution, formalizing the causal strength of this relationship (see Definition \ref{def:embedding-problems} and Definition \ref{def:embedding-problems:solutions}). We then characterize a type of solutions said to be ``optimal'' (see Definition \ref{def:optimal-solutions}) and show that these solutions are minimal (see Proposition \ref{prop:minimality}). We finish the section with our first main result, stated in Theorem \ref{theo:optimal-solutions-construction}, which gives a recipe for constructing optimal solutions and hence minimal ones.

\begin{definition}[Embedding problems]\label{def:embedding-problems}
Let $(S,<)$ be a finite strict linear order. An \emph{embedding problems} consists of a pair $(a,b)$ of unlabeled partitions in $\mathbf{UP}(S)$. A \emph{solution} for the embedding problem $(a,b)$ is an unlabeled partition $x$ in $\mathbf{UP}(S)$ for which there exists a morphism $a \times x \to b$ in $\mathbf{UP}(S)$.
\end{definition}

Below, Convention \ref{def:embedding-problems:solutions} establishes an implicit analogy between the solutions of an embedding problem and the concept of an affine algebraic set for a linear polynomial  of the form $ax-b$. In our case, the so-called polynomial is expressed as an arrow $a \times x \to b$ in the category of unlabeled partitions.

\begin{convention}[Solutions]\label{def:embedding-problems:solutions}
Let $(S,<)$ be a finite strict linear order. For every embedding problem $(a,b)$, we will denote by $\mathbb{V}(a,b)$ is set 
\[
\{x\textrm{ in }\mathbf{UP}(S)~|~\mathbf{UP}(S)(a \times x,b) \neq \emptyset\},
\]
which contains all the solutions of the embedding problem $(a,b)$.
\end{convention}

Contrary to linear equations, an embedding problem can be associated with many solutions -- this is illustrated in Example \ref{exa:embedding-problems-and-their-solutions}. Below, in Proposition \ref{prop:existence:embedding-solutions}, we show that embedding problems possess at least one solution.

\begin{proposition}[Existence]\label{prop:existence:embedding-solutions}
Let $(S,<)$ be a finite strict linear order. For every embedding problem $(a,b)$ in $\mathbf{UP}(S)$, then the partition $b$ belongs to $\mathbb{V}(a,b)$;
\end{proposition}
\begin{proof}
Recall that the product $a \times b$ is equipped with two projections $a \times b \to a$ and $a \times b \to b$. The second projection makes $b$ a solution for the embedding problem $(a,b)$. 
\end{proof}





\begin{example}[Embedding problems and their solutions]\label{exa:embedding-problems-and-their-solutions}
Let $(S,<)$ be the finite strict linear order $\{\texttt{A}<\texttt{B}<\texttt{C}<\texttt{D}<\texttt{E}<\texttt{F}<\texttt{G}\}$ and let us consider the embedding problem defined by the following pair $(a,b)$ of partitions in $\mathbf{UP}(S)$:
\[
a = \{\texttt{A},\texttt{F}\},\{\texttt{B}\},\{\texttt{C},\texttt{D}\},\{\texttt{E},\texttt{G}\}\quad\quad\quad
b = \{\texttt{A}\},\{\texttt{B},\texttt{D},\texttt{F}\},\{\texttt{C},\texttt{E},\texttt{G}\}
\]
We can see that we can easily recover the partition $b$ from the partition $a \times b$, displayed below, by merging brackets together.
\[
a \times b = \{\texttt{A}\},\{\texttt{B}\},\{\texttt{C}\},\{\texttt{E},\texttt{G}\},\{\texttt{D}\},\{\texttt{F}\} \longrightarrow \{\texttt{A}\},\{\texttt{B},\texttt{D},\texttt{F}\},\{\texttt{C},\texttt{E},\texttt{G}\}
\]
We can also see that $x = b$ is not the only solution for the existence of a morphism $a \times x \to b$. For instance, the partition $a \times b$ is also a solution since $a \times a \times b = a \times b$ (by Proposition \ref{prop:product-idempotency:unlabeled}). Below, we give several examples of solutions for the problem $(a,b)$ -- these solutions can all be obtained by merging brackets in the fiber representation of the partition $a \times b$ .
\[
\begin{array}{ll}
x_1 &= \{\texttt{A}\},\{\texttt{B}\},\{\texttt{C}\},\{\texttt{D},\texttt{E},\texttt{F},\texttt{G}\}\\
x_2 &= \{\texttt{A}\},\{\texttt{C}\},\{\texttt{B},\texttt{D},\texttt{E},\texttt{F},\texttt{G}\}\\
x_3 &= \{\texttt{A}\},\{\texttt{B},\texttt{C}\},\{\texttt{D},\texttt{E},\texttt{F},\texttt{G}\}\\
\end{array}
\quad\quad\quad
\begin{array}{ll}
x_4 &= \{\texttt{A},\texttt{B}\},\{\texttt{C}\},\{\texttt{D},\texttt{E},\texttt{F},\texttt{G}\}\\
x_5 &= \{\texttt{A},\texttt{C}\},\{\texttt{B},\texttt{D},\texttt{E},\texttt{F},\texttt{G}\}\\
x_6 &= \{\texttt{A},\texttt{B},\texttt{C}\},\{\texttt{D},\texttt{E},\texttt{F},\texttt{G}\}
\end{array}
\]
Note that the solutions $x_6$ and $x_5$ are minimal, meaning that we cannot form new solutions of $\mathbb{V}(a,b)$ by merging their brackets further. Formally, this means that, for every $x \in \{x_5,x_6\}$, if there is a solution $y \in \mathbb{V}(a,b)$ such that there exists an arrow $x \to y$ in $\mathbf{UP}(S)$, then the identity $x = y$ holds. Intuitively, we can see the minimal solutions $x_5$ and $x_6$ as ``classifying'' partitions that would separate information from noise in distinct brackets. This is precisely the type of solutions that we want to consider in order to assess the causal strength of the relationship linking $a$ to $b$. We will show how to assess the causal strength of an embedding problem statistically throughout section \ref{sec:functorial-properties-of-ANOVA} and section \ref{sec:association-studies}. In this section, we focus on constructing minimal solutions of embedding problems through the characterization of a sub-type of solutions called ``optimal solutions'' (see Proposition \ref{prop:minimality} and Theorem \ref{theo:optimal-solutions-construction}).
\end{example}


The following proposition gives a combinatorial description of the solutions of an embedding problem expressed in terms of the fibers of the partitions.

\begin{proposition}[Properties]\label{prop:properties-of-solutions}
Let $(S,<)$ be a finite strict linear order, $(a,b)$ be an embedding problem in $\mathbf{UP}(S)$ and $x$ be a solution in $\mathbb{V}(a,b)$ such that $\mathsf{fib}(a) = \{a_i\}_{i \in [n]}$, $\mathsf{fib}(b) = \{b_k\}_{k \in [m]}$, and $\mathsf{fib}(x) = \{x_j\}_{j \in [p]}$ (see Convention \ref{conv:fibers}). The following statements hold:
\begin{itemize}
\item[1)] the equation $(a_i \cap x_j) \cap (a_{i'} \cap x_{j'})  = \emptyset$ holds for every $(i,j) \neq (i',j')$ in $[n] \times [p]$;
\item[2)] for every $k \in [m]$, if we denote $U(k)=\{(i,j)~|~a_i \cap x_j \cap b_k \neq \emptyset\}$, then:
\[
b_k = \bigcup_{(i,j) \in U(k)}a_i \cap x_j 
\]
As a result, any inequality of the form  $a_i \cap x_j \cap b_k \neq \emptyset$ implies that $a_i \cap x_j \subseteq b_k$.
\end{itemize}
\end{proposition}
\begin{proof}
We show each item in the order given in the statement.

\fbox{1} First, recall that the collections $\mathsf{fib}(a) = \{a_i\}_{i \in [n]}$ and $\mathsf{fib}(x) = \{x_j\}_{j \in [p]}$ are pairwise disjoint. Also, recall that any inequality of the form $(i,j) \neq (i',j')$ in $[n] \times [p]$ means that either $i \neq i'$ or $j \neq j'$. In the context of our assumptions, if $i \neq i'$, then $a_i \cap a_{i'} = \emptyset$, which implies that the identity $a_i \cap x_j \cap a_{i'} \cap x_{j'}  = \emptyset$ holds. On the other hand, if $j \neq j'$, then $x_j \cap x_{j'} = \emptyset$, which also implies that the identity $a_i \cap x_j \cap a_{i'} \cap x_{j'}  = \emptyset$ holds. This shows item 1.

\fbox{2} Let $f:a \times x \to b$ be the morphism of $\mathbf{UP}(S)$ that encodes the solution structure of $x$ in $\mathbb{V}(a,b)$. Since $\mathbf{UP}(S)$ is a subcategory of $\mathbf{LP}(S)$, we can interpret the product $a \times x$ in $\mathbf{LP}(S)$ so that we can consider the underlying function $f$ to be of the form $\mathsf{Im}(a,x) \to [m]$. It them follows from Definition \ref{def:labeled-partitions} that the equation $b(s) = f(a(s),x(s))$ holds for every $s \in S$. As a result, we have the following two equivalent equations:
\begin{equation}\label{eq:b_a_x-expression}
b^{-1}(k) = \bigcup_{(i,j) \in f^{-1}(k)} a^{-1}(i) \cap x^{-1}(j) \quad\quad \Leftrightarrow \quad\quad b_k = \bigcup_{(i,j) \in f^{-1}(k)} a_i \cap x_j.
\end{equation}
To show item 2, we need to show that $f^{-1}(k) = U(k)$. We will proceed by a double inclusion using a proof by contradiction. 

\fbox{2.1} Suppose that there exists $(i',j') \in U(k)$ such that $(i',j') \notin f^{-1}(k)$. First, the leftmost equation of (\ref{eq:b_a_x-expression}) gives us the equation:
\begin{equation}\label{eq:properties-of-solutions:2-1}
b_k \cap a_{i'} \cap x_{j'} = \bigcup_{(i,j) \in f^{-1}(k)} a_i \cap x_j \cap a_{i'} \cap x_{j'}.
\end{equation}
Since $(i',j') \notin f^{-1}(k)$, item 1 implies that we have te inequality $(a_i \cap x_j) \cap (a_{i'} \cap x_{j'}) = \emptyset$ for every $(i,j) \in f^{-1}(k)$. This inequality and equation (\ref{eq:properties-of-solutions:2-1}) imply that the set $b_k \cap a_{i'} \cap x_{j'}$ is not empty. However, this contradicts our assumption that $(i',j') \in U(k)$. We deduce that the inclusion $U(k) \subseteq f^{-1}(k)$ holds for every $k \in [m]$. 

\fbox{2.2} Suppose that exists $(i',j') \in f^{-1}(k)$ such that $(i',j') \notin U(k)$. Since the relation $(i',j') \in f^{-1}(k)$ holds, the leftmost equation of (\ref{eq:b_a_x-expression}) gives rise to the following identities:
\[
b_k \cap a_{i'} \cap x_{j'} = \bigcup_{(i,j) \in f^{-1}(k)} a_i \cap x_j \cap a_{i'} \cap x_{j'} = a_{i'} \cap x_{j'}.
\]
Since $(i',j') \notin U(k)$, we have $a_{i'} \cap x_{j'} = \emptyset$. Because $f$ is of the form $\mathsf{Im}(a,x) \to [m]$, we have $(i',j') \in f^{-1}(k) \subseteq \mathsf{Im}(a,x)$, which means that there exists $s \in S$ such that $(i',j') = (a(s),x(s))$. In other words, the element $s$ is in the intersection $a_{i'} \cap x_{j'}$. This is not possible since $a_{i'} \cap x_{j'} = \emptyset$. This shows that the inclusion $f^{-1}(k) \subseteq U(k)$ holds for every $k \in [m]$.  

\fbox{2.3} Hence, we have shown that the identity $f^{-1}(k) = U(k)$ holds for every $k \in [m]$. The formula of item 2 follows from the rightmost equation of (\ref{eq:b_a_x-expression}). The last statement of item 2 is a direct consequence of that formula: the relation $a_i \cap x_j \cap b_k \neq \emptyset$ is equivalent to the relation $(i,j) \in U(k)$, which, by the formula, means that the inclusion $a_i \cap x_j \subseteq b_k$ holds.
%
\end{proof}

In the definition given below, we define a type of solution whose fibers minimally describe the subdivision of the fibers associated with the partition $b$ of an embedding problem $(a,b)$. As will be shown in Proposition \ref{prop:minimality}, this characterization defines a type of solution that is minimal in $\mathbb{V}(a,b)$.

\begin{definition}[Optimal solutions]\label{def:optimal-solutions}
Let $(S,<)$ be a finite strict linear order, $(a,b)$ be an embedding problem in $\mathbf{UP}(S)$ and $x$ be a solution in $\mathbb{V}(a,b)$ such that $\mathsf{fib}(a) = \{a_i\}_{i \in [n]}$, $\mathsf{fib}(b) = \{b_k\}_{k \in [m]}$, and $\mathsf{fib}(x) = \{x_j\}_{j \in [p]}$. The solution $x$ is said to be \emph{optimal} if for every pair of distinct elements $j_1,j_2 \in [p]$, there exist $i \in [n]$ and two distinct elements $k_1,k_2 \in [m]$ such that $a_i \cap x_{j_1} \cap b_{k_1} \neq \emptyset$ and $a_i \cap x_{j_2} \cap b_{k_2} \neq \emptyset$.
\end{definition}

\begin{proposition}[Minimality]\label{prop:minimality}
Let $(S,<)$ be a finite strict linear order and $(a,b)$ be an embedding problem in $\mathbf{UP}(S)$. If $x$ is an optional solution of $\mathbb{V}(a,b)$, then for every $y \in \mathbb{V}(a,b)$ such that there exists an arrow $x \to y$ in $\mathbf{UP}(S)$, the equation $x = y$ holds.
\end{proposition}
\begin{proof}
Suppose that $\mathsf{fib}(a) = \{a_i\}_{i \in [n]}$, $\mathsf{fib}(b) = \{b_k\}_{k \in [m]}$, $\mathsf{fib}(y) = \{y_k\}_{k \in [p]}$ and $\mathsf{fib}(x) = \{x_j\}_{j \in [q]}$. The arrow $x \to y$ implies that there exists a function $f:[q] \to [p]$ such that, for every $k \in [p]$, the following equation holds:
\begin{equation}\label{eq:x-y-minimality}
y_k = \bigcup_{j \in f^{-1}(k)} x_j.
\end{equation}
Suppose that $f$ is not an identity arrow. By Proposition \ref{prop:skeletal}, the function $f:[q] \to [p]$ cannot be a bijection. Because the equation $f \circ x = y$ holds and $y:S \to [p]$ is a surjection (Definition \ref{def:labeled-partitions}), so is the function $f$. We deduce that $f$ cannot be an injection, which implies that there exists $k \in [p]$ and two distinct elements $j_1,j_2 \in [q]$ such that $j_1,j_2 \in f^{-1}(k)$.
Because $x$ is optimal (Definition \ref{def:optimal-solutions}), there exists $i \in  [n]$ and two distinct elements $l_1,l_2 \in [m]$ such that $a_i \cap x_{j_1} \cap b_{l_1} \neq \emptyset$ and $a_i \cap x_{j_2} \cap b_{l_2} \neq \emptyset$. Since $j_1,j_2 \in f^{-1}(k)$, equation (\ref{eq:x-y-minimality}) gives us the following relations:
\[
y_k \cap a_i \cap b_{l_1} = \bigcup_{j \in f^{-1}(k)} a_i \cap x_{j} \cap b_{l_1} \neq \emptyset,
\quad\quad
y_k \cap a_i \cap b_{l_2} = \bigcup_{j \in f^{-1}(k)} a_i \cap x_{j} \cap b_{l_2} \neq \emptyset.
\]
In particular, this shows that $y_k \cap a_i$ is not empty. Since $y \in \mathbb{V}(a,b)$, it follows from the inequalities $y_k \cap a_i \cap b_{l_1} \neq \emptyset$, $y_k \cap a_i \cap b_{l_2} \neq \emptyset$ and item 2 of Proposition \ref{prop:properties-of-solutions} that the following inclusions hold:
\[
y_k \cap a_i \subseteq b_{l_1}, \quad\quad\quad y_k \cap a_i \subseteq b_{l_2}.
\]
Since $y_k \cap a_i$ is not empty, the intersection $b_{l_1} \cap b_{l_2}$ is not empty, which contradicts the fact that the collection $\mathsf{fib}(b) = \{b_k\}_{k \in [m]}$ is pairwise disjoint. As a result, the function $f$ is an identity and $x = y$.
\end{proof}

In the rest of the section, we describe an algorithm that constructs optimal solutions of embedding problems. As was suggested at the end of Example \ref{exa:embedding-problems-and-their-solutions}, our algorithm will consist in merging the parts of a given solution until the parts of the resulting partitions cannot be merged further. These successive merging steps will be called ``contractions'' and will be characterized as follows.

\begin{definition}[Contractions]\label{def:contraction}
Let $(S,<)$ be a finite strict linear order and $x$ be a partition in $\mathbf{UP}(S)$ such that $\mathsf{fib}(x) = \{x_j\}_{j \in [p]}$. For every pair $(j_1,j_2) \in [p]\times [p]$, we define the \emph{contraction of $x$ at $(j_1,j_2)$} as the following partition
\[
x[j_1,j_2] := \prod_{i = [p] \backslash \{j_1,j_2\}}\delta(x_i) \times \delta(x_{j_1} \cup x_{j_2})
\]
\end{definition}

Importantly, any contraction can be related to its associated partition through an arrow, as shown in Proposition \ref{prop:contraction:canonical-arrow}. In section \ref{sec:functorial-properties-of-ANOVA}, these arrows will allow us to control contractions statistically.

\begin{proposition}[Canonical arrow]\label{prop:contraction:canonical-arrow}
Let $(S,<)$ be a finite strict linear order. For every partition $x$ in $\mathbf{UP}(S)$ such that $\mathsf{fib}(x) = \{x_j\}_{j \in [p]}$ and every pair $(j_1,j_2) \in [p]\times [p]$, there is an arrow $x \to x[j_1,j_2]$ in $\mathbf{UP}(S)$. 
\end{proposition}
\begin{proof}
The product structure gives us a projection of the following form.
\[
\delta(x_{j_1}) \times \delta(x_{j_2}) \times x[j_1,j_2] \to x[j_1,j_2]
\]
We are going to show that the domain of this arrow is $x$. For convenience, we will denote this domain as $z$. By assumption, we have
\begin{align*}
z & = \prod_{i = [p] \backslash \{j_1,j_2\}}\delta(x_i) \times \delta(x_{j_1}) \times \delta(x_{j_2}) \times \delta(x_{j_1} \cup x_{j_2}) &(\textrm{Definition \ref{def:contraction}})\\
& = \prod_{i = [p] \backslash \{j_1,j_2\}}\delta(x_i) \times \delta(x_{j_1}) \times \Big(\delta(x_{j_2}) \times \delta(x_{j_2}) \times \delta(x_{j_1})\Big) & (\textrm{Proposition \ref{prop:union-property:delta}})\\
& = \prod_{i = [p] \backslash \{j_1,j_2\}}\delta(x_i) \times \delta(x_{j_1}) \times \delta(x_{j_2}) &(\textrm{Proposition \ref{prop:product-idempotency:unlabeled}})
\end{align*}
The previous expression clearly shows that the product $z$ is equal to the partition  $\prod_{i = 1}^p\delta(x_i) = \delta^*(\mathsf{fib}(x)) = x$ (see Proposition \ref{prop:fibers} and Proposition \ref{prop:familial-product:unlabeled-partitions}). This finishes the proof of the proposition.
\end{proof}

The following definition formalizes our algorithm in terms of an object called a ``reduction''. At every step, the algorithm requires to make a choice of contraction. Eventually, these successive contractions lead us to find a minimal solution for a given embedding problem (see Theorem \ref{theo:optimal-solutions-construction}).

\begin{definition}[Reduction]\label{def:reduction}
Let $(S,<)$ be a finite strict linear order and $a$ be a partition in $\mathbf{UP}(S)$ such that $\mathsf{fib}(a) = \{a_i\}_{i \in [n]}$. For every partition $x$ in $\mathbf{UP}(S)$ such that $\mathsf{fib}(x) = \{x_j\}_{j \in [p]}$, we define the set 
\[
L_{a}(x) := \{(j_1,j_2) \in [p] \times [p]~|~j_1 \neq j_2\textrm{ and } a_i \cap x_{j_1} \neq \emptyset \Rightarrow a_i \cap x_{j_2} = \emptyset \}.
\]
We will call a \emph{reduction of $x$} any infinite sequence $(u_{k})_{k\geq 0}$ of partitions in $\mathbf{UP}(S)$ satisfying the following construction:
\[
\left\{
\begin{array}{l}
u_0 = x\\
u_{k+1} = 
\left\{
\begin{array}{ll}
u_{k}[j_1,j_2]&\textrm{for some }(j_1,j_2) \in L_{a}(u_{k})\textrm{ if }L_{a}(u_{k}) \neq \emptyset\\
u_{k}&\textrm{if }L_{a}(u_{k}) = \emptyset\\
\end{array}
\right.
\end{array}
\right.
\]
\end{definition}

In the following proposition, we show a completeness result that ensures that our algorithm always terminates and returns an output. Theorem \ref{theo:optimal-solutions-construction} uses this algorithm to yield optimal solutions for a given embedding problem. 

\begin{proposition}[Completeness]
Let $(S,<)$ be a finite strict linear order and $a$ be a partition in $\mathbf{UP}(S)$. For every partition $x$ in $\mathbf{UP}(S)$ and every reduction $(u_{k})_{k\geq 0}$ of $x$, there exists an non-negative integer $k_0$ such that $L_{a}(u_{k_0})$ is empty.
\end{proposition}
\begin{proof}
Suppose that for every non-negative integer $k$, the set $L_{a}(u_k)$ is not empty. This means that we can write $u_k$ as $u_{k-1}[j_1,j_2]$. If we suppose that the partition $u_k$ is of the form $S \to [p_k]$, then we have $p_{k-1} < p_k$. Hence we obtain an infinite sequence of decreasing non-negative integers as follows:
\[
p_0 > p_1 > p_2 > \dots > p_{p_0} > \underbrace{p_{p_0+1}}_{<0} > \dots
\] 
But this is impossible, otherwise the $(p_0+1)$-th term $p_{p_0+1}$ must be negative. Hence, there exists a non-negative integer $k_0$ such that $L_{a}(u_{k_0})$ is empty.
\end{proof}

In the following definition, we name the outputs of our algorithm ``reduced elements''. 

\begin{definition}[Reduced elements]\label{def:reduced-elements}
Let $(S,<)$ be a finite strict linear order and $a$ be a partition in $\mathbf{UP}(S)$. For every partition $x$ in $\mathbf{UP}(S)$ and every reduction $u = (u_{k})_{k\geq 0}$ of $x$, we denote by $k(u)$ the smallest non-negative integer for which $L_{a}(u_{k(u)})$ is empty. We then denote by $u^*$ the term $u_{k(u)}$, which we call the \emph{reduced element of $u$}. 
\end{definition}

The following theorem shows that we can construct an optimal solution by merging the parts of some other non-optimal solution satisfying certain properties. Later, in section \ref{ssec:embedding-theorems}, we will construct such an non-optimal solution. Overall, this will give us an algorithm that constructs optimal solutions for a given embedding problem.

\begin{theorem}[Optimal solutions]\label{theo:optimal-solutions-construction}
Let $(S,<)$ be a finite strict linear order, $(a,b)$ be a pair of partitions in $\mathbf{UP}(S)$ such that $\mathsf{fib}(a) = \{a_i\}_{i \in [n]}$ and $\mathsf{fib}(b) = \{b_k\}_{k \in [m]}$. For every partition $x$ in $\mathbf{UP}(S)$ and every reduction $u = (u_{k})_{k\geq 0}$ of $x$, the following statements hold:
\begin{itemize}
\item[1)] there exists an arrow $x \to u^{*}$ in $\mathbf{UP}(S)$;
\item[2)] if $x$ is in $\mathbb{V}(a,b)$, then so is $u^{*}$;
\item[3)] if $x$ is in $\mathbb{V}(a,b)$ such that there exist a product decomposition $\{v_k\}_{k \in [m]}$ of $x$ and a function $r:[n] \to [m]$ satisfying the following properties:
\begin{itemize}
\item[$\triangleright$] the inclusion $v_{k} \subseteq b_{k}$ holds for every $k \in [m]$,
\item[$\triangleright$] the complement $S \backslash \bigcup_{k \in [m]} v_{k}$ is included in $\bigcup_{i \in [n]} a_i \cap b_{r(i)}$,
\item[$\triangleright$] if the intersection $a_{i} \cap v_{k}$ is not empty, then $k \neq r(i)$,
\end{itemize}
then $u^{*}$ is an optimal solution of $\mathbb{V}(a,b)$;
\end{itemize}
\end{theorem}
\begin{proof}
We show each item in the order given in the statement.

\fbox{1} By Definition \ref{def:reduction}, for every $k \in [k(u)]$, there exists pair $(j_1^k,j_2^k)$ in $L_{a}(u_k)$ such that we have the relation $u_{k+1} = u_k[j_1^k,j_2^k]$.
It follows from Proposition \ref{prop:contraction:canonical-arrow} and Definition \ref{def:reduced-elements} that we have a finite sequence of arrows in $\mathbf{UP}(S)$ as follows:
\[
x = u_0 \to u_1 \to \dots \to u_{k(u)} = u^*
\]
The composition of these arrows give us an arrow $x \to u^*$ in $\mathbf{UP}(S)$.

\fbox{2}
Let us show, by induction, that for every $k \in \{0,1,\dots,k(u)\}$, the identity $a \times x = a \times u_k$ holds. The case where $k=0$ is straightforward since $x = u_0$. 
Let us now suppose that the equation $a \times x = a \times u_k$ holds for some $k < k(u)$ and let us show that we have the equation $a \times x = a \times u_{k+1}$. We will denote $\mathsf{fib}(u_k)$ as $\{u_{k,j}\}_{j \in[p_k]}$. 

First, since $k$ is smaller than $k(u)$, the set $L_{a}(u_k)$ is not empty. Let $(j_1,j_2)$ be the pair of $L_{a}(u_k)$ for which the identity $u_{k+1} = u_k[j_1,j_2]$ holds. By Definition \ref{def:reduction}, if we denote $\mathsf{fib}(a) = \{a_i\}_{i \in [n]}$, then we have the implication
\begin{equation}\label{eq:theo-optimal-solution:implication}
a_i \cap u_{k,j_1} \neq \emptyset \Rightarrow a_i \cap u_{k,j_2} = \emptyset
\end{equation}
Denote by $I_k$ the set $\{i \in [n]~|~a_i \cap u_{k,j_1} \neq \emptyset\}$, 

\fbox{2.1} It directly follows from implication (\ref{eq:theo-optimal-solution:implication}) that the intersection $u_{k,j_2} \cap \bigcup_{i \in I_k} a_i$ is empty.

\fbox{2.2} Let us show that $u_{k,j_1}$ is included in $\bigcup_{i \in I_k} a_i$. If it was not, then $u_{k,j_1} \cap S \backslash  \bigcup_{i \in I_k} a_i$ would be non-empty. Since $a:S \to [n]$ is a surjection, the complement $S \backslash  \bigcup_{i \in I_k} a_i$ is equal to the union
\[
\bigcup_{i \in [p] \backslash I_k} a_i
\]
This means that there exists an integer $i_0 \in [p] \backslash I_k$ for which the relation $u_{k,j_1} \cap a_{i_0}$ is not empty. By definition of the set $I_k$, this implies that $i_0 \in I_k$, which is impossible since $i_0$ belongs to the complement of $I_k$. We deduce that the following inclusion holds:
\[
u_{k,j_1} \subseteq \bigcup_{i \in I_k} a_i
\]

\fbox{2.3} Recall that, by Proposition \ref{prop:fibers}, the unlabeled partition $a$ can be written as $\delta^*(\mathsf{fib}(a))$. In fact, Proposition \ref{prop:familial-product:unlabeled-partitions} shows that $a$ is the product of the partitions $\delta(a_i)$ for every $i \in [m]$. Because the collection $\{a_i\}_{i \in [n]}$ is pairwise disjoint, we can use Proposition \ref{prop:union-property:delta} to show that the following identity holds:
\begin{equation}\label{eq:optimal-solutions-construction:a-equals-a-a_i}
a = a \times \delta(\bigcup_{i \in I_k} a_i)
\end{equation}
The previous identity combined with the observation made in \fbox{2.1} allow us to show that the following identities hold:
\begin{align*}
a \times u_k & = a \times \delta(\bigcup_{i \in I_k} a_i) \times \prod_{j=1}^{p_k} \delta(u_{k,j})&(\textrm{Equation (\ref{eq:optimal-solutions-construction:a-equals-a-a_i})})\\
 & = a \times \delta(\bigcup_{i \in I_k} a_i) \times \delta(u_{k,j_1}) \times \delta(u_{k,j_2}) \times \prod_{j \in [p_k]\backslash\{j_1,j_2\}} \delta(u_{k,j}) &\\
  & = a \times \delta(\bigcup_{i \in I_k} a_i) \times \delta(u_{k,j_1} \cup u_{k,j_2}) \times \prod_{j \in [p_k]\backslash\{j_1,j_2\}} \delta(u_{k,j}) &(\textrm{Proposition \ref{prop:union-property:delta}})\\
    & = a  \times \delta(u_{k,j_1} \cup u_{k,j_2}) \times \prod_{j \in [p_k]\backslash\{j_1,j_2\}} \delta(u_{k,j}) &(\textrm{Equation (\ref{eq:optimal-solutions-construction:a-equals-a-a_i})})\\
        & = a  \times u_{k}[j_1,j_2] &(\textrm{Definition \ref{def:contraction}})
\end{align*}
The previous equation shows that the identity $a \times u_k = a \times u_{k+1}$ holds, which concludes our induction. In fact, by induction, we have shown that the identity $a \times x = a \times u^{*}$ holds. Now, because $x \in \mathbb{V}(a,b)$, there exists an arrow $a \times x \to b$ in $\mathbf{UP}(S)$. As a result, we also have an arrow $a \times u^{*} \to b$, which proves item 2.

\fbox{3} By item 2, we already know that $u^{*}$ is in $\mathbb{V}(a,b)$. Let us show that the reduced element $u^{*}$ is an optimal solution of $\mathbb{V}(a,b)$. To do so, we will have to express the fibers of $u^{*}$ in terms of the sets $\{v_{k}\}_{k \in [m]}$.

First, we can express the fibers of $u^{*}$ in terms of unions of fibers of $x$. More specifically, if we denote $\mathsf{fib}(u^{*}) = \{u^{*}_{l}\}_{l \in [q]}$ and we denote by $g$ the function $[p] \to [q]$ inducing the arrow $x \to u^{*}$ in $\mathbf{UP}(S)$ (by item 1), then the following equation holds for every $l \in [q]$:
\begin{equation}\label{eq:optimal-solutions-construction:x-to-reduced-elt}
u^{*}_l = \bigcup_{j \in g^{-}(l)} x_j
\end{equation}
Now, we can express the fibers of $x$ in terms of the product decomposition of $x$. By Definition \ref{def:product-decompositions}, the equation $x = \prod_{k = 1}^m \delta(v_k)$ holds. Since we have an inclusion $v_k \subseteq b_k$ for every $k \in [m]$ and the collection $\{b_k\}_{k \in [m]}$ is pairwise disjoint, so is the collection the collection $\{v_k\}_{k \in [m]}$. By Proposition \ref{prop:familial-product:unlabeled-partitions}, this means that the identity $x = \delta^*(\{v_k\}_{k \in [m]})$ holds. It is then easy to describe the fibers $\{x_j\}_{j \in [p]}$ of the partition $x = \delta^*(\{v_k\}_{k \in [m]})$ by using Definition \ref{conv:relative_discrete_partitions:familial}. Specifically, Definition \ref{conv:relative_discrete_partitions:familial} implies that $m$ must be equal to either $p$ or $p-1$ such that there exists an injection $\iota:[m] \to [p]$ for which the following equations hold:
\begin{equation}\label{eq:optimal-solutions-construction:x-as-v}
x_{j} =
\left\{
\begin{array}{ll}
v_k&\textrm{if }j = \iota(k)\\
S \backslash \bigcup_{k=1}^m v_k & \textrm{if $j$ is not in the image of }\iota
\end{array}
\right.
\end{equation}

Let us now show that the solution $u^{*} \in \mathbb{V}(a,b)$ is optimal by using the characterization of Definition \ref{def:optimal-solutions}. Consider two distinct elements $l_1,l_2 \in [q]$. By Definition \ref{def:reduced-elements}, the set $L_{a}(u^{*})$ is empty, so the pair $(l_1,l_2)$ cannot belong to $L_{a}(u^{*})$. As a result, there exists $i_0 \in [n]$ such that $a_{i_0} \cap u^{*}_{l_1} \neq \emptyset$ and $a_{i_0} \cap u^{*}_{l_2} \neq \emptyset$. It follows from equation (\ref{eq:optimal-solutions-construction:x-to-reduced-elt}) that there exists $j_1 \in g^{-}(l_1)$ and $j_2 \in g^{-}(l_2)$ such that:
\begin{equation}\label{eq:optimal-solutions-construction:item3:a-x-non-empty}
a_{i_0} \cap x_{j_1} \neq \emptyset \quad\quad\quad\quad\quad a_{i_0} \cap x_{j_2} \neq \emptyset
\end{equation}
The following points conclude the proof of item 3. Each point makes an assumption about the form taken by the sets $x_{j_1}$ and $x_{j_2}$. In particular, their forms must follow the formulas of (\ref{eq:optimal-solutions-construction:x-as-v}).

\fbox{3.1} Suppose that $x_{j_1}$ is of the form $v_{k_1}$. By our assumptions on $x$,n the inclusion $v_{k_1} \subseteq b_{k_1}$ holds, which gives the inclusion:
\[
a_{i_0} \cap x_{j_1} = a_{i_0} \cap x_{j_1} \cap v_{k_1} \subseteq a_{i_0} \cap x_{j_1} \cap b_{k_1}.
\]
The leftmost inequality of (\ref{eq:optimal-solutions-construction:item3:a-x-non-empty}) then iplies that $a_{i_0} \cap x_{j_1} \cap b_{k_1} \neq \emptyset$. 

\fbox{3.1.1} If $x_{j_2}$ is also of the form $v_{k_2}$, then the inclusion $v_{k_2} \subseteq b_{k_2}$ (which is part of our assumptions on $x$) gives us the inclusion
\[
a_{i_0} \cap x_{j_2} = a_{i_0} \cap x_{j_2} \cap v_{k_2} \subseteq a_{i_0} \cap x_{j_2} \cap b_{k_2}.
\]
Then, the rightmost inequality of (\ref{eq:optimal-solutions-construction:item3:a-x-non-empty}) implies that $a_{i_0} \cap x_{j_1} \cap b_{k_1} \neq \emptyset$. Let us show that the inequality $k_1 \neq k_2$ holds. If the equality $k_1 = k_2$ was true, then the sets $x_{j_1}$ and $x_{j_2}$ would contain the following non-empty set:
\[
a_{i_0} \cap x_{j_1} = a_{i_0} \cap v_{k_1} = a_{i_0} \cap v_{k_2} = a_{i_0} \cap x_{j_2}
\]
Because we have the inequality $j_1 \ne j_2$ and the collection $\{x_j\}_{j \in [p]}$ is pairwise disjoint, the sets $x_{j_1}$ and $x_{j_2}$ must be disjoint and the inequality $k_1 \neq k_2$ cannot hold. As a result, we have shown that the following relations hold:
\[
a_{i_0} \cap x_{j_1} \cap b_{k_1} \neq \emptyset,\quad
a_{i_0} \cap x_{j_2} \cap  b_{k_2} \neq \emptyset
\quad\textrm{ where }k_1 \neq k_2.
\]

\fbox{3.1.2} On the other hand, if $x_{j_2}$ is of the form $S \backslash \bigcup_{k=1}^m v_k$, then we have the inclusions
\[
a_{i_0} \cap x_{j_2} \subseteq a_{i_0} \cap x_{j_2} \cap (S \backslash \bigcup_{k=1}^m v_k) \subseteq a_{i_0} \cap x_{j_2} \cap (\bigcup_{i \in [n]} a_i \cap b_{r(i)})
\]
Since the collection $\{a_i\}_{i \in [n]}$ is pairwise disjoint, the intersection of the set $a_{i_0} \cap x_{j_2}$ must have a non-empty intersection with the set $a_{i_0} \cap b_{r(i_0)}$, meaning that the inequality $a_{i_0} \cap x_{j_2} \cap  b_{r(i_0)} \neq \emptyset$ holds. Now, because $a_{i_0} \cap v_{k_1}$ is not empty (see the leftmost inequality of (\ref{eq:optimal-solutions-construction:item3:a-x-non-empty})), our assumptions on $x$ give us the inequality $k_1 \neq r(i_0)$. As a result, we have shown that the following relations hold:
\[
a_{i_0} \cap x_{j_1} \cap b_{k_1} \neq \emptyset,\quad
a_{i_0} \cap x_{j_2} \cap  b_{r(i_0)} \neq \emptyset
\quad\textrm{ where }k_1 \neq r(i_0).
\]
\fbox{3.2} Suppose that $x_{j_1}$ is of the form $S \backslash \bigcup_{k=1}^m v_k$. By our assumptions on $x$, we have the inclusion:
\[
a_{i_0} \cap x_{j_1} = a_{i_0} \cap x_{j_1} \cap (S \backslash \bigcup_{k=1}^m v_k) \subseteq a_{i_0} \cap x_{j_1} \cap (\bigcup_{i \in [n]} a_i \cap b_{r(i)}).
\]
Since the collection $\{a_i\}_{i \in [n]}$ is pairwise disjoint, the intersection of the set $a_{i_0} \cap x_{j_1}$ must have a non-empty intersection with the set $a_{i_0} \cap b_{r(i_0)}$, meaning that the inequality $a_{i_0} \cap x_{j_1} \cap  b_{r(i_0)} \neq \emptyset$ holds.
Because $x_{j_1}$ and $x_{j_2}$ are disjoint sets and $x_{j_1} = S \backslash \bigcup_{k=1}^m v_k$, theset $x_{j_2}$ must necessarily be of the form $v_{k_2}$. By our assumptions on $x$, the inclusion $v_{k_2} \subseteq b_{k_2}$ holds, which gives us the inclusion
\[
a_{i_0} \cap x_{j_2} = a_{i_0} \cap x_{j_2} \cap v_{k_2} \subseteq a_{i_0} \cap x_{j_1} \cap b_{k_1}
\]
It follows from the rightmost inequality of (\ref{eq:optimal-solutions-construction:item3:a-x-non-empty}) that $a_{i_0} \cap x_{j_1} \cap b_{k_1}$ is not empty. In addition, the rightmost inequality of (\ref{eq:optimal-solutions-construction:item3:a-x-non-empty}) implies that the intersection $a_{i_0} \cap v_{k_2}$ is not empty, which, by our assumptions on $x$, gives us the inequality $k_2 \neq r(i_0)$. As a result, we have shown that the following relations hold:
\[
a_{i_0} \cap x_{j_1} \cap b_{r(i_0)} \neq \emptyset,\quad
a_{i_0} \cap x_{j_2} \cap  b_{k_2} \neq \emptyset
\quad\textrm{ where }r(i_0) \neq k_2.
\]
\fbox{3.3} We now conclude the proof. According to formula (\ref{eq:optimal-solutions-construction:x-to-reduced-elt}), we have the following inclusions:
\[
a_{i_0} \cap x_{j_1} \subseteq a_{i_0} \cap u^*_{l_1}\quad\quad\quad
a_{i_0} \cap x_{j_2} \subseteq a_{i_0} \cap u^*_{l_2}
\]
The previous points have shown that there exist distinct elements $k_1,k_2 \in [m]$ for which the two relations displayed in the following two arrays hold (the bottom relations directly follow form the previous two inclusions):
\[
\left\{
\begin{array}{ll}
a_{i_0} \cap x_{j_1} \cap b_{k_1} &\neq \emptyset\\
a_{i_0} \cap x_{j_1} \cap b_{k_1} &\subseteq a_{i_0} \cap u^*_{l_1} \cap b_{k_1}
\end{array}
\right.
\quad\quad\quad
\left\{
\begin{array}{ll}
a_{i_0} \cap x_{j_2} \cap b_{k_2} &\neq \emptyset\\
a_{i_0} \cap x_{j_2} \cap b_{k_2} &\subseteq a_{i_0} \cap u^*_{l_2} \cap b_{k_2}
\end{array}
\right.
\]
By Definition \ref{def:optimal-solutions}, the previous set of relations implies that $u^{*}$ is optimal.
\end{proof}

\subsection{Embedding theorems}\label{ssec:embedding-theorems}
In this section, we use item 3 of Theorem \ref{theo:optimal-solutions-construction} to construct optimal solutions of embedding problems. To do so, we use Proposition \ref{prop:embedding-decomposition-III}, which gives us a solution that satisfies all the conditions required by Theorem \ref{theo:optimal-solutions-construction} (see Proposition \ref{prop:embedding-IV}). We conclude the section with our major result, stated in Theorem \ref{theo:minimal-solutions}, that gives us a recipe for constructing minimal solutions of embedding problems.

We start with Proposition \ref{prop:embedding-decomposition-I}, which can be used to refine solutions of embedding problems into larger solutions, meaning that these partitions contain smaller parts. Later, this refinement will be used to construct partitions satisfying the properties listed in Theorem \ref{theo:optimal-solutions-construction}. Below, we express this refinement in terms of a product of images of the operation $\delta$ (see Convention \ref{conv:relative_discrete_partitions_delta}).

\begin{proposition}[Embedding I]\label{prop:embedding-decomposition-I}
Let $(S,<)$ be a finite strict linear order and let $a$ be an unlabeled partition in $\mathbf{UP}(S)$ such that $\mathsf{fib}(a) = \{a_i\}_{i \in [n]}$. For every subset $x$ of $S$, the following equation holds:
\[
a \times \delta(x) = a \times \prod_{i = 1}^n \delta(a_i \cap x) 
\]
\end{proposition}
\begin{proof}
By Proposition \ref{prop:fibers}, we can rewrite the unlabeled partition $a$ as $\delta^*(\mathsf{fib}(a))$ and it follows from Proposition \ref{prop:familial-product:unlabeled-partitions} that we have the following equation:
\[
a \times \delta(x) = \Big(\prod_{i=1}^n \delta(a_i)\Big) \times \delta(x) 
\]
By Proposition \ref {prop:triple-xyz-delta}, we have the decomposition 
\[
\delta(a_1) \times \delta(x\backslash \emptyset) = \delta(a_1) \times \delta(a_1 \cap x) \times  \delta(x \backslash (\emptyset \cup a_1)).
\]
As a result, we obtain the following equation:
\[
a \times \delta(x) = \Big(\prod_{i=1}^n \delta(a_i)\Big) \times \delta(a_1 \cap x) \times  \delta(x \backslash a_1)
\]
Our goal is to reduce the partition $\delta(x \backslash a_1)$ to a terminal partition by using a similar argument for each partition $\delta(a_i)$ where the index $i$ is between 2 and $n$ -- we proceed by induction as follows. Suppose that we have the following formula for some positive integer $k-1 \in [n]$.
\[
a \times \delta(x) = \Big(\prod_{i=1}^n \delta(a_i)\Big) \times \Big(\prod_{i=1}^{k-1} \delta(a_{i} \cap x) \Big) \times  \delta(x \backslash \bigcup_{i=1}^{k-1} a_i)
\] 
By Proposition \ref{prop:triple-xyz-delta}, we have the following formula (since the fibers of $a$ are pairwise disjoint):
\[
\delta(a_{k}) \times \delta(x \backslash \bigcup_{i=1}^{k-1} a_i) = \delta(a_k) \times \delta(x \backslash a_k) \times \delta(x \backslash \bigcup_{i=1}^{k} a_i)
\]
Hence, we now have the following relation:
\[
a \times \delta(x) = \Big(\prod_{i=1}^n \delta(a_i)\Big) \times \Big(\prod_{i=1}^{k} \delta(a_{i} \cap x) \Big) \times  \delta(x \backslash \bigcup_{i=1}^{k} a_i)
\] 
By induction, this shows that the following formula holds:
\[
a \times \delta(x) = \Big(\prod_{i=1}^n \delta(a_i)\Big) \times \Big(\prod_{i=1}^{n} \delta(a_{i} \cap x) \Big) \times  \delta(x \backslash \bigcup_{i=1}^{n} a_i)
\]
Since the fibers of the surjection $a:S \to [n]$ satisfy the equation $\bigcup_{i=1}^{n} a_i = S$, we deduce that $\delta(x \backslash \bigcup_{i=1}^{n} a_i)$ is equal to $\delta(\emptyset)$, which is terminal in $\mathbf{UP}(S)$ (see Example \ref{exa:relative_discrete_partitions}). In other words, the equation of the statement holds.
\end{proof}

The following proposition refines the partition provided by Proposition \ref{prop:embedding-decomposition-I} to a partition containing fewer parts. Specifically, this partition can be used to construct solutions (of embedding problems) satisfying the third property listed in item 3 of Theorem \ref{theo:optimal-solutions-construction}.

\begin{proposition}[Embedding II]\label{prop:embedding-decomposition-II}
Let $(S,<)$ be a finite strict linear order and let $(a,x)$ be a pair of unlabeled partitions in $\mathbf{UP}(S)$ such that $\mathsf{fib}(a) = \{a_i\}_{i \in [n]}$ and $\mathsf{fib}(x) = \{x_j\}_{j \in [m]}$. For every function $r:[n] \to [m]$, the following equation holds:
\[
a \times x = a \times \prod_{i=1}^n\prod_{j=1,j \neq r(i)}^m \delta(a_i \cap x_j)
\]
\end{proposition}
\begin{proof}
First, it follows from Proposition \ref{prop:delta-complement-fiber} that we have the following equation:
\[
\delta(a_i) \times \prod_{j=1,j \neq r(i)}^m \delta(a_i \cap x_j) = \delta(a_i \backslash \bigcup_{j=1,j \neq r(i)}^n a_i \cap x_j) \times \prod_{j=1,j \neq r(i)}^m \delta(a_i \cap x_j)
\]
Because the collection $x_1,\dots,x_m$ constitutes the collection of fibers of the surjection $b:S \to [m]$, the equation $a_i \backslash \bigcup_{j=1,j \neq r(i)}^n a_i \cap x_j = a_i \cap x_{r(i)}$ holds. As a result, the equation previous equation is equivalent to the following one:
\begin{equation}\label{eq:embedding-decomposition-II}
\delta(a_i) \times \prod_{j=1,j \neq r(i)}^m \delta(a_i \cap x_j) = \prod_{j=1}^m \delta(a_i \cap x_j)
\end{equation}
Recall that, by Proposition \ref{prop:fibers}, the unlabeled partition $a$ can be written as $\delta^*(\mathsf{fib}(a))$ and the unlabeled partition $x$ can be written as $\delta^*(\mathsf{fib}(x))$.
We can use equation (\ref{eq:embedding-decomposition-II}) and Proposition \ref{prop:familial-product:unlabeled-partitions} to show that the following identities hold:
\begin{align*}
a \times x &= a \times \prod_{j=1}^m \delta(x_j)&\\
&= \prod_{j=1}^m \Big(a \times\delta(x_j)\Big)&(\textrm{Prosition \ref{prop:product-idempotency:unlabeled}})\\
&= \prod_{j=1}^m \Big(a \times \prod_{i=1}^n\delta(a_i \cap x_j)\Big)&(\textrm{Proposition \ref{prop:embedding-decomposition-I}})\\
&= \prod_{i=1}^n\Big(\delta(a_i) \times \prod_{j=1}^m \delta(a_i \cap x_j)\Big) &(\textrm{Proposition \ref{prop:product-idempotency:unlabeled}})\\
&= \prod_{i=1}^n\Big(\delta(a_i) \times \prod_{j=1,j \neq r(i)}^m \delta(a_i \cap x_j)\Big) &(\textrm{Equation \ref{eq:embedding-decomposition-II}})
\end{align*}
It follows from an obvious use of Proposition \ref{prop:product-idempotency:unlabeled} and a rearrangement of the products that the last equation given above is equivalent to the equation of the statement.
\end{proof}

While the new partition provided by Proposition \ref{prop:embedding-decomposition-II} satisfies the third property listed in item 3 of Theorem \ref{theo:optimal-solutions-construction}, it may not satisfy the two other properties listed there. To force these properties to hold, we need to merge some of the parts of the partition constructed in Theorem \ref{theo:optimal-solutions-construction}. We do so through the construction of Convention \ref{conv:chi-notation} and show, in Proposition \ref{prop:embedding-decomposition-III}, that the resulting partition satisfies the type of equation shown in Proposition \ref{prop:embedding-decomposition-I} and Proposition \ref{prop:embedding-decomposition-II}.

\begin{convention}[Notation]\label{conv:chi-notation}
Let $(S,<)$ be a finite strict linear order and let $(a,x)$ be a pair of unlabeled partitions in $\mathbf{UP}(S)$ such that $\mathsf{fib}(a) = \{a_i\}_{i \in [n]}$ and $\mathsf{fib}(x) = \{x_j\}_{j \in [m]}$. For every function $r:[n] \to [m]$, we denote by $I(j,r,a,x)$ the set
\[
I(j,r,a,x) =\{i \in [n]~|~r(i) \neq j\textrm{ and } a_i \cap x_j \notin \{\emptyset, a_i\}\}
\]
and we denote by $\chi(r,a,x)$ the partition
\[
\prod_{j=1}^m \delta\Big(\Big(\bigcup_{i \in I(j,r,a,x)} a_i\Big) \cap x_j\Big)
\]
\end{convention}

\begin{proposition}[Embedding III]\label{prop:embedding-decomposition-III}
Let $(S,<)$ be a finite strict linear order, let $(a,x)$ be a pair of unlabeled partitions in $\mathbf{UP}(S)$ such that $\mathsf{fib}(a) = \{a_i\}_{i \in [n]}$ and $\mathsf{fib}(x) = \{x_j\}_{j \in [m]}$. For every function $r:[n] \to [m]$, the following equation holds:
\[
a \times x = a \times \chi(r,a,x)
\]
\end{proposition}
\begin{proof}
It follows from permuting the products of the formula given in Proposition \ref{prop:embedding-decomposition-II} that the following equations hold (see Convention \ref{conv:chi-notation}):
\begin{equation}\label{eq:embedding-decomposition-III-1}
a \times x =  a \times \prod_{j=1}^m \prod_{i =1, r(i) \neq j}^n \delta(a_i \cap x_j) = a \times \prod_{j=1}^m \prod_{i \in I(j,r,a,x)} \delta(a_i \cap x_j)
\end{equation}
Note that, for every $j \in [m]$, we can use Proposition \ref{prop:union-property:delta} to show that the following identity holds:
\begin{equation}\label{eq:embedding-decomposition-III-2}
\prod_{i \in I(j,r,a,x)} \delta(a_i) \times \prod_{i \in I(j,r,a,x)} \delta(a_i \cap x_j) = \prod_{i \in I(j,r,a,x)} \delta(a_i) \times \delta\Big(\Big(\bigcup_{i \in I(j,r,a,x)} a_i\Big) \cap x_j\Big)
\end{equation}
Since, for every $i \in I(j,r,a,x)$, the partition $a$ is a product made of the partition $\delta(a_i)$, we can use equation (\ref{eq:embedding-decomposition-III-2}) inside equation (\ref{eq:embedding-decomposition-III-1}) to deduce the following identities:
\begin{align*}
a \times x & = a \times \prod_{j=1}^m \prod_{i \in I(j,r,a,x)} \delta(a_i \cap x_j)&\\
& = a \times \prod_{j=1}^m \Big(\prod_{i \in I(j,r,a,x)} \delta(a_i) \times \prod_{i \in I(j,r,a,x)} \delta(a_i \cap x_j)\Big)&(\textrm{Proposition \ref{prop:product-idempotency:unlabeled}})\\
& = a \times \prod_{j=1}^m \Big(\prod_{i \in I(j,r,a,x)} \delta(a_i) \times \delta\Big(\Big(\bigcup_{i \in I(j,r,a,x)} a_i\Big) \cap x_j\Big)\Big)&(\textrm{Equation \ref{eq:embedding-decomposition-III-2}})\\
& = a \times \prod_{j=1}^m \delta\Big(\Big(\bigcup_{i \in I(j,r,a,x)} a_i\Big) \cap x_j\Big)&(\textrm{Proposition \ref{prop:product-idempotency:unlabeled}})
\end{align*}
By Convention \ref{conv:chi-notation}, the last equation shows the statement.
\end{proof}


The following proposition essentially shows that the partition constructed in Convention \ref{conv:chi-notation} satisfies the properties listed in Theorem \ref{theo:optimal-solutions-construction}.

\begin{proposition}[Embedding IV]\label{prop:embedding-IV}
Let $(S,<)$ be a finite strict linear order, let $(a,x)$ be a pair of unlabeled partitions in $\mathbf{UP}(S)$ such that $\mathsf{fib}(a) = \{a_i\}_{i \in [n]}$ and $\mathsf{fib}(x) = \{x_j\}_{j \in [m]}$ and let $r$ be a function $[n] \to [m]$. If, for every $j \in [m]$, we write
\begin{equation}\label{eq:embediing-IV}
v_j = \Big(\bigcup_{i \in I(j,r,a,x)} a_i\Big) \cap x_j\\
\end{equation}
then the following properties hold:
\begin{itemize}
\item[1)] if the function $r:[n] \to [m]$ satisfies the logical implication $(a_i \subseteq x_{j}) \Rightarrow (j = r(i))$ for every $i \in [n]$ and $j\in [m]$, then the following identity holds:
\[
S \backslash \bigcup_{j =1}^m v_j = \bigcup_{i=1}^n a_i \cap x_{r(i)}
\]
\item[2)] if the inequality $a_i \cap v_{j} \neq \emptyset$ holds for some $j \in [m]$ and $i \in [n]$, then $j \neq r(i)$.
\end{itemize}
\end{proposition}
\begin{proof}
We show each item in the order given in the statement.

\fbox{1} We proceed by a double inclusion. Our first inclusion will use a proof by contradiction.

\fbox{1.1} Let $i_0$ be an element in $[n]$. For every element $t \in a_{i_0} \cap x_{r(i_0)}$, suppose that $t \in \bigcup_{j = 1}^m v_j$. This means that there exists $j_0 \in [m]$ such that: 
\[
t \in v_{j_0} = \Big(\bigcup_{i \in I(j_0,r,a,x)} a_i\Big) \cap x_{j_0}.
\]
As a result, we have $t \in x_{j_0}$ and there exists $i_1 \in I(j_0,r,a,x)$ such that $t \in a_{i_1}$. By assumption, we also have $t \in x_{r(i_0)}$ and $t \in a_{i_0}$. Since the collections $\{a_i\}_{i \in [n]}$ and $\{x_j\}_{j \in [m]}$ are pairwise disjoint, the identities $r(i_0) = j_0$ and $i_0 = i_1$ must hold. Since we have $i_1 \in I(j_0,r,a,x)$, we must also have $r(i_1) \neq j_0$. The equation $r(i_0) = j_0$ implies that the relation $r(i_1) = r(i_0)$, but since the equation $i_0 = i_1$ holds, this is impossible. As a result, the element $t$ cannot belong to $\bigcup_{j = 1}^m v_j$. This shows that the following inclusion holds:
\[
\bigcup_{i=1}^n a_i \cap x_{r(i)} \subseteq S \backslash \bigcup_{j =1}^m v_j
\]
\fbox{1.2} If $t \in S \backslash \bigcup_{j =1}^m v_j$, then for every $j \in [m]$, the relation $t \notin v_{j}$ holds. It then follows from equation (\ref{eq:embediing-IV}) that we have the following implication
\begin{equation}\label{eq:complement-expression:1-2}
\left\{
\begin{array}{lll}
t \in \bigcup_{i \in I(j,r,a,x)} a_i &\Rightarrow& t \notin x_{j}\\
t \in x_{j} &\Rightarrow& t \notin \bigcup_{i \in I(j,r,a,x)} a_i
\end{array}
\right.
\end{equation}
If we denote $j_0 = x(t)$ and $i_0 = a(t)$, the definition of fibers of $x$ and $a$ imply that we have the relations $t \in x_{j_0}$ and $t \in a_{i_0}$. It then follows from the bottom implication of (\ref{eq:complement-expression:1-2}) that $t \notin \bigcup_{i \in I(j_0,r,a,x)} a_i$. Since $t \in a_{i_0}$, the element $i_0$ cannot be in $I(j_0,r,a,x)$. By Convention \ref{conv:chi-notation}, this means that $r(i_0) = j_0$ or $a_{i_0} \cap x_{j_0} \in \{ \emptyset, a_{i_0}\}$. Since $t \in a_{i_0} \cap x_{j_0}$, we must have $r(i_0) = j_0$ or $a_{i_0} \cap x_{j_0} = a_{i_0}$. Since the latter equation is equivalent to the inclusion $a_{i_0} \subseteq x_{j_0}$, our assumption on the function $r:[n] \to [m]$ implies that the equation $r(i_0) = j_0$ is satisfied in any case. We conclude that the relation $t \in a_{i_0} \cap x_{j_0}$ can be rewritten as $t \in a_{i_0} \cap x_{r(i_0)}$. Hence, we have shown the other inclusion:
\[
S \backslash \bigcup_{j =1}^m v_j \subseteq \bigcup_{i=1}^n a_i \cap x_{r(i)}
\]
\fbox{2} Let $i_0 \in [n]$ and $j_0 \in [m]$ such that the set $a_{i_0} \cap v_{j_0}$ is not empty. Because the collection $\{a_i\}_{i \in [n]}$ is pairwise disjoint, equation (\ref{eq:embediing-IV}) implies that the element $i_0$ belongs to $I(j_0,r,a,x)$. By Convention \ref{conv:chi-notation}, this means that the equation $r(i_0) \neq j_0$ holds.
\end{proof}

We now prove our main result, which gives an algorithm to construct optimal solutions. Note that this algorithm is not completely deterministic as it requires to pick a reduction of the partition constructed in Convention \ref{conv:chi-notation}. Indeed, picking such a reduction may involve various choices of contractions for which we do not have criteria yet (see Definition \ref{def:reduction}). In section \ref{sec:functorial-properties-of-ANOVA}, we will see how one can make this process deterministic by picking contractions that optimize the statistical properties of the partitions involved in the reduction.

\begin{theorem}[Construction of minimal solutions]\label{theo:minimal-solutions}
Let $(S,<)$ be a finite strict linear order, let $(a,b)$ be an embedding problem in $\mathbf{UP}(S)$ such that $\mathsf{fib}(a) = \{a_i\}_{i \in [n]}$ and $\mathsf{fib}(b) = \{b_j\}_{j \in [m]}$ and let $r:[n] \to [m]$ be a function satisfying the logical implication for every $i \in [n]$ and $j\in [m]$.
\[
(a_i \subseteq b_{j}) \Rightarrow (j = r(i))
\]
If $u$ is a reduction of the partition $\chi(r,a,b)$ (see Definition \ref{def:reduction}), then the reduced element $u^*$ is an optimal solution in $\mathbb{V}(a,b)$. In addition, there is an arrow $\chi(r,a,b) \to u^*$ in $\mathbf{UP}(S)$ and $u^*$ is minimal in $\mathbb{V}(a,b)$ for the canonical pre-order structure on $\mathbf{UP}(S)$ (see Remark \ref{rem:pre-order-LP}).
\end{theorem}
\begin{proof}
By Proposition \ref{prop:embedding-decomposition-III}, the equation $a \times \chi(r,a,b) = a \times b$ holds. Because, by Proposition \ref{prop:existence:embedding-solutions}, the partition $b$ is an element of $\mathbb{V}(a,b)$, we have an arrow $a \times \chi(r,a,b) = a \times b \to b$, which shows that $\chi(r,a,b)$ is in $\mathbb{V}(a,b)$. By Convention \ref{conv:chi-notation}, the partition $\chi(r,a,b)$ admits a product decomposition of the form $\{v_j\}_{j \in [m]}$ where the term $v_j$ satisfies the following equation for every $j \in [m]$:
\[
v_j = \Big(\bigcup_{i \in I(j,r,a,x)} a_i\Big) \cap b_j
\]
By assumption on the function $r$ and Proposition \ref{prop:embedding-IV}, the following properties hold:
\begin{itemize}
\item[$\triangleright$] the inclusion $v_{k} \subseteq b_{j}$ holds for every $j \in [m]$,
\item[$\triangleright$] the complement $S \backslash \bigcup_{j \in [m]} v_{j}$ is included in $\bigcup_{i \in [n]} a_i \cap b_{r(i)}$,
\item[$\triangleright$] if the intersection $a_{i} \cap v_{j}$ is not empty, then $j \neq r(i)$,
\end{itemize}
By item 3 of Proposition \ref{theo:optimal-solutions-construction}, this means that the reduced element $u^{*}$ is an optimal solution of $\mathbb{V}(a,b)$. Furthermore, item 1 of Proposition \ref{theo:optimal-solutions-construction} implies that there exists an arrow $\chi(r,a,b) \to u^*$ in $\mathbf{UP}(S)$. Since $u^*$ is an optimal solution in $\mathbb{V}(a,b)$, Proposition \ref{prop:minimality} implies that $u^*$ is minimal in $\mathbb{V}(a,b)$.
\end{proof}

\section{Functorial properties of ANOVA}\label{sec:functorial-properties-of-ANOVA}

Theorem \ref{theo:minimal-solutions} depends on two undetermined variables: a function $r$ and a reduction $u$. While the statement of this theorem tells us how to choose the function $r$, it does not give any information regarding the choices involved in the construction of the reduction $u$ (see Definition \ref{def:reduction}). In this section, we show how to construct a reduction that statistically optimizes the amount of information and noise detected by the minimal solution. We do so by studying the functorial properties of the  statistical method ANOVA and show how these properties can help us to statistically analyze the contraction arrows (Proposition \ref{prop:contraction:canonical-arrow}) associated with this reductions.

\subsection{ANOVA Functor}\label{ssec:ANOVA-functor}
In this section, we formalize the ANOVA formalism \cite{Oehlert} in terms of a functor going from a category of unlabeled partitions to a certain category $\mathbf{Data}$ (see Proposition \ref{prop:ANOVA-functor}). Intuitively, the objects of the category $\mathbf{Data}$ contain the minimal information required to reason in the ANOVA formalism.

From Definition \ref{def:fiber-matrix} to Proposition \ref{prop:functoriality-fiber-matrix}, our goal is to translate the formalism of partitions into linear algebra.

\begin{definition}[Fiber matrix]\label{def:fiber-matrix}
Let $g$ and $g'$ be two non-negative integers and let $f:[g] \to [g']$ be a function. We denote by $\mathsf{M}(f)$ the $g \times g'$-matrix whose coefficient $m_{i,j}(f)$ is defined as follows: 
\[
m_{i,j}(f) = 
\left\{
\begin{array}{ll}
1&\textrm{if }i = f(j)\\
0&\textrm{if }i \neq f(j)
\end{array}
\right.
\]
\end{definition}

\begin{definition}[Matrix category]
We will denote by $\mathbf{Mat}$ the category whose objects are non-negative integers and whose arrows $g \to g'$ are given by $g \times g'$-matrices of real numbers. The composition is given by the usual matrix multiplication.
\end{definition}

\begin{proposition}[Functoriality of fiber matrices]\label{prop:functoriality-fiber-matrix}
Let $S$ be a finite strict linear order. The mapping rule $f \mapsto \mathsf{M}(f)$ defined in Definition \ref{def:fiber-matrix} induces a functor $\mathbf{UP}(S) \to \mathbf{Mat}$ whose mapping rule on the objects sends a partition $p:S \to [n]$ to the integer $n$.
\end{proposition}
\begin{proof}
For every pair of composable arrows $f:p_1 \to p_2$ and $g:p_2 \to p_3$ in $\mathbf{UP}(S)$, the multiplication $\mathsf{M}(g)\mathsf{M}(f)$ is defined by the following coefficients:
\[
\sum_{k} m_{i,k}(g)m_{k,j}(f) = 
\left\{
\begin{array}{ll}
1&\textrm{if }\exists k\,:\,i = g(k)\textrm{ and }k = f(j)\\
0&\textrm{otherwise }
\end{array}
\right.
=
\left\{
\begin{array}{ll}
1&\textrm{if }i = g \circ f(j)\\
0&\textrm{otherwise }
\end{array}
\right.
\]
The rightmost brackets shows that $\mathsf{M}(g)\mathsf{M}(f)$ is equal to $\mathsf{M}(g \circ f)$ (by Definition \ref{def:fiber-matrix}). Finally, for every non-negative integer $n$, the image $\mathsf{M}(\mathsf{id}_{[n]})$ of the identity $\mathsf{id}_{[n]}:[n] \to [n]$ is such that $m_{i,j}(f) = 1$ if $i =j$ and $m_{i,j}(f) = 0$ otherwise. In other words, the matrix $\mathsf{M}(\mathsf{id}_{[n]})$ is the identity $n\times n$-matrix.
\end{proof}

We recall the definition of the Hadamard product, which we will use in the definition of the category $\mathbf{Data}$ (see Definition \ref{def:data-category}).

\begin{convention}[Hadamard product]
Let $g$ be a non-negative integer and let $v$ and $c$ be two vectors in $\mathbb{R}^g$. We will denote by $v \odot c$ the Hadamard product of $v$ and $c$, namely the componentwise multiplication $(v_1c_1,v_2c_2,\dots,v_gc_g)$ of $c$ and $v$.
\end{convention}

We now introduce the category $\mathbf{Data}$. Each object of $\mathbf{Data}$ contains a vector $v$ representing a sequence of measurement values and a vector $c$ of non-negative integers representing the number of times a given measurement has been observed.

\begin{definition}[Data category]\label{def:data-category}
We denote by $\mathbf{Data}$ the category whose objects are pairs $(g,v,c)$ where $g$ is a non-negative integer, $v$ is a vector in $\mathbb{R}^g$ and $c$ is a vector in $\mathbb{N}^g_{*}$ and whose arrows $(g,v,c) \to (g',v',c')$ are surjections $f:[g] \to [g']$ satisfying the following equations:
\[
\mathsf{M}(f) c = c'
\quad\quad
\mathsf{M}(f)(v \odot c) = (v' \odot c')
\]
The composition operation of $\mathbf{Data}$ is induced by the composition of surjections.
\end{definition}

Using the interpretation mentioned above Definition \ref{def:data-category}, we can associate every object of $\mathbf{Data}$ with a total number of observations and an average measurement for these observations.

\begin{definition}[Cardinal and mean]\label{def:cardinal-mean}
For every object $(g,v,c)$ in $\mathbf{Data}$, we define the \emph{cardinal $N(g,v,c)$} and the \emph{mean $\mu(g,c,v)$} of the object $(g,c,v)$ by the following quantities:
\[
N(g,v,c) = \sum_{k=1}^{g} c_k\quad\quad\quad
\mu(g,v,c) = \frac{1}{N(g,v,c)}\sum_{i=1}^g c_i v_i 
\]
\end{definition}

An important property of $\mathbf{Data}$ is that its cardinal and mean are invariant under the arrows. 

\begin{proposition}[Invariance]\label{prop:invariance:1}
For every arrow $f:(g_1,v_1,c_1) \to (g_2,v_2,c_2)$ in $\mathbf{Data}$, the identities
$N(g_1,v_1,c_1) = N(g_2,v_2,c_2)$ and $\mu(g_1,v_1,c_1) = \mu(g_2,v_2,c_2)$ hold.
\end{proposition}
\begin{proof}
Let us show the fist equation. It follows from Definition \ref{def:cardinal-mean} and Definition \ref{def:data-category} that the following identities hold:
\[
N(g_2,v_2,c_2) = \sum_{k=1}^{g_2} c_{2,k} \quad\quad\quad c_{2,k} = \sum_{i = 1}^{g_1}m_{k,i}(f)c_{1,i}.
\]
By Definition \ref{def:fiber-matrix}, we have the following relations:
\[
\sum_{k=1}^{g_2} m_{k,i}(f) = m_{f(i),i}(f) + \sum_{k=1,k \neq f(i)}^{g_2} m_{k,i}(f) = 1 + 0 = 1
\]
Finally, we deduce from the previous equations that the following series of equations hold:
\[
N(g_2,v_2,c_2) = \sum_{k=1}^{g_2} \sum_{i = 1}^{g_1}m_{k,i}(f)c_{1,i} = \sum_{i = 1}^{g_1} \Big(\sum_{k=1}^{g_2} m_{k,i}(f)\Big)c_{1,i} = \sum_{i = 1}^{g_1} 1 \times c_{1,i} = N(g_1,v_1,c_1)
\]
Let us now show the second equation. It follows from Definition \ref{def:cardinal-mean} and Definition \ref{def:data-category} that the following identities hold:
\[
\mu(g_2,v_2,c_2) = \sum_{k=1}^{g_2} c_{2,k}v_{2,k} \quad\quad\quad c_{2,k}v_{2,k} = \sum_{i = 1}^{g_1}m_{k,i}(f)c_{1,i}v_{1,i}
\]
Using the previous equations, we deduce the following series of identities:
\[
\mu(g_2,v_2,c_2) = \sum_{k=1}^{g_2} \sum_{i = 1}^{g_1}m_{k,i}(f)c_{1,i}v_{1,i} = \sum_{i = 1}^{g_1} \Big(\sum_{k=1}^{g_2} m_{k,i}(f)\Big)c_{1,i}v_{1,i} = \sum_{i = 1}^{g_1} c_{1,i}v_{1,i} = \mu(g_1,v_1,c_1)
\]
This finishes the proof.
\end{proof}

In the rest of the section, we relate categories of partitions to the category $\mathbf{Data}$. This will allow us to relate the abstract concepts of Definition \ref{def:cardinal-mean} to actual cardinals and means. In the following definition, we recall the concepts of mean and variance for a given list $y$ of measurements .

\begin{definition}[Cardinal, mean and variance]\label{def:cardinal-mean-variance}
Let $S$ be a finte set and $y:S \to \mathbb{R}$ be a function. For every subset $T \subseteq S$, we denote by $|T|$ the cardinal of $T$ and we define the \emph{mean $\mu_y(T)$} and the \emph{variance $\sigma^2_y(T)$} of the object $T$ by the following quantities:
\[
\mu_y(T) = \frac{1}{|T|}\sum_{s \in T} y(s) \quad\quad\quad
\sigma^2_y(T) = \frac{1}{|T|}\sum_{s \in T} (y(s) - \mu_y(T))^2
\]
\end{definition}

The previous definition allows us to recover the usual setting needed to perform ANOVA. 

\begin{definition}[ANOVA setting]\label{def:Anova-setting}
Let $(S,<)$ be a finite strict linear order and $y:S \to \mathbb{R}$ be a function. For every partition $p:S \to [g]$ in $\mathbf{UP}(S)$, we will denote by $\mathsf{Anova}_y(p)$ the object $(g,v,c)$ of $\mathbf{Data}$ defined by the following equations.
\[
v = \Big(\mu_y(p^{-1}(1))),\mu_y(p^{-1}(2)), \dots, \mu_y(p^{-1}(g))\Big)
\quad\quad
c = \Big(|p^{-1}(1))|,|p^{-1}(2))|,\dots,|p^{-1}(g))|\Big)
\]
\end{definition}

Below, Proposition \ref{prop:anova-mapping} and Proposition \ref{prop:ANOVA-functor} show that the mapping $\mathsf{Anova}$ is functorial over categories of unlabeled partitions. Note that, while categories of unlabeled partitions are pre-order categories (see Remark \ref{rem:pre-order-LP}), the category $\mathbf{Data}$ is not. Showing that $\mathsf{Anova}$ is a functor will permit us to focus our reasonings in the category $\mathbf{Data}$, which is more adequate to performing ANOVA.

\begin{proposition}[ANOVA mapping]\label{prop:anova-mapping}
Let $(S,<)$ be a finite strict linear order, let $y:S \to \mathbb{R}$ be a function and let $p_1:S \to [g_1]$ and $p_2:S \to [g_2]$ be two unlabeled partitions in $\mathbf{UP}(S)$. For every morphism $f:p_1 \to p_2$ in $\mathbf{UP}(S)$, the underlying function $f:[g_1] \to [g_2]$ defines a morphism $\mathsf{Anova}_y(p_1) \to \mathsf{Anova}_y(p_2)$ in $\mathbf{Data}$
\end{proposition}
\begin{proof}
Let the triples $(g_1,v_1,c_1)$ and $(g_2,v_2,c_2)$ denote the objects $\mathsf{Anova}_y(p_1)$ and $\mathsf{Anova}_y(p_2)$, respectively. By Definition \ref{def:data-category}, we need to show that $f:[g_1] \to [g_2]$ is a surjection and that the following two equations hold:
\begin{equation}\label{eq:anova-mapping}
\mathsf{M}(f) c_1 = c_2
\quad\quad
\mathsf{M}(f)(v_1 \odot c_1) = (v_2 \odot c_2)
\end{equation}
Since the identity $f \circ p_1 = p_2$ holds and $p_2$ is a surjection (by Definition \ref{def:labeled-partitions}), the function $f$ must be surjective. The rest of the section focuses on proving the two identities of (\ref{eq:anova-mapping}).

\fbox{1} Let us show that the equation $\mathsf{M}(f) c_1 = c_2$ holds. By Definition \ref{def:Anova-setting}, the $i$-th coefficient $c_{2,i}$ of $c_2$ is the cardinal $|p_2^{-1}(i)|$ for every $i \in [g_2]$. Since the equation $f \circ p_1 = p_2$ holds (by Definition \ref{def:labeled-partitions}), we have $c_{2,i} = |p_1^{-1}(f^{-1}(i))|$. Now, because the fiber $p_1^{-1}(f^{-1}(i))$ is the union of the fibers $p_1^{-1}(j)$ for every $j \in f^{-1}(i)$, the following identities hold for every $i \in [g_2]$:
\begin{align*}
c_{2,i} & = \left|\bigcup_{j \in f^{-1}(i)} p_1^{-1}(j)\right|&\\
& = \sum_{j \in f^{-1}(i)} |p_1^{-1}(j)|&(\textrm{$p_1$ is a function})\\
& = \sum_{j = 1}^{g_1} m_{i,j}(f)\cdot |p_1^{-1}(j)|&(\textrm{Definition \ref{def:fiber-matrix} and $i = f(j)$})
\end{align*}
The previous equation shows that the identity $c_{2,i} = \sum_{j = 1}^{g_1} m_{i,j}(f)\cdot c_{1,j}$ holds for every $i \in [g_2]$. In particular, this proves the identity $\mathsf{M}(f) c_1 = c_2$.

\fbox{2} Let us show that the equation $\mathsf{M}(f)(v_1 \odot c_1) = (v_2 \odot c_2)$ holds. We will need to use the definition of the operator $\mu_y$ given in Definition \ref{def:cardinal-mean-variance}.
\begin{align*}
(\mathsf{M}(f)(v_1 \odot c_1))_i & = \sum_{j = 1}^{g_1} m_{i,j}(f)\cdot v_{1,j} \cdot c_{1,j} &\\
& = \sum_{j = 1}^{g_1} m_{i,j}(f)\cdot \mu_y(p^{-1}(j)) \cdot |p_1^{-1}(j)| &(\textrm{Definition \ref{def:Anova-setting}})\\
& = \sum_{j = 1}^{g_1} m_{i,j}(f)\sum_{s \in p_1^{-1}(j)} y(s) &(\textrm{Definition \ref{def:cardinal-mean-variance}})\\
& = \sum_{j \in f^{-1}(i)}^{g_1}\sum_{s \in p_1^{-1}(j)} y(s) &(\textrm{Definition \ref{def:fiber-matrix} and $i = f(j)$})
\end{align*}
Since the union of the fiber $p_1^{-1}(j)$ for every $j \in f^{-1}(i)$ is equal to the fiber $p_1^{-1}(f^{-1}(i))$, which is in turn equal to th fiber $p_2^{-1}(i)$, the previous series of identities gives us the following equations:
\[
(\mathsf{M}(f)(v_1 \odot c_1))_i =  \sum_{s \in p_1^{-1}(f^{-1}(i))} y(s) =  \sum_{s \in p_2^{-1}(i)} y(s) = \mu_{y}(p_2^{-1}(i))
\]
By Definition \ref{def:Anova-setting}, this shows that the equation $(\mathsf{M}(f)(v_1 \odot c_1))_i = c_{2,i}$ holds. In other words, we have the identity $\mathsf{M}(f)(v_1 \odot c_1) = (v_2 \odot c_2)$.
\end{proof}

\begin{proposition}[ANOVA functor]\label{prop:ANOVA-functor}
Let $(S,<)$ be a finite strict linear order and $y:S \to \mathbb{R}$ be a function. The mapping $\mathsf{Anova}_{y}$ defined in Proposition \ref{prop:anova-mapping} defines a functor $\mathbf{UP}(S) \to \mathbf{Data}$ that maps any partition $p$ in $\mathbf{UP}(S)$ to the object $\mathsf{Anova}_y(p)$ and any arrow $f:p_1 \to p_2$ in $\mathbf{UP}(S)$ to the matrix $\mathsf{M}(f)$ defined on the underlying function of $f$.
\end{proposition}
\begin{proof}
Directly follows from the functorial properties shown in Proposition \ref{prop:functoriality-fiber-matrix}.
\end{proof}

The following proposition confirms our intuition regarding the concepts of Definition \ref{def:cardinal-mean}.

\begin{proposition}[Invariance properties]\label{prop:invariance:2}
Let $(S,<)$ be a finite strict linear order and $y:S \to \mathbb{R}$ be a function. The following identities hold:
\[
N(\mathsf{Anova}_y(p)) = |S|
\quad\quad
\mu(\mathsf{Anova}_y(p)) = \mu_y(S)
\]
\end{proposition}
\begin{proof}
Let $p:S \to [g]$ be an object of $\mathbf{UP}(S)$. Since $p:S \to [g]$ describes a surjection, its collection of fibers $p^{-1}(i)$ form a partition of the set $S$, meaning that $S$ is the union of the sets $p^{-1}(i)$ and each intersection $p^{-1}(i) \cap p^{-1}(i)$ is empty. As a result, we deduce that the following equation holds:
\[
N(\mathsf{M}(p)) = \sum_{i=1}^p|p^{-1}(i)| = \left|\bigcup p^{-1}(i)\right| = |S|
\]
Similarly, we can show that the double sum $\sum_{i=1}^g\sum_{s \in p^{-1}(i)}$ is equal to the sum $\sum_{s \in S}$, which allows us to deduce the following equations:
\[
\mu(\mathsf{M}(p)) = \frac{1}{|S|}\sum_{i=1}^g|p^{-1}(i)|\times \mu_y(p^{-1}(i)) = \frac{1}{|S|}\sum_{i=1}^g\sum_{s \in p^{-1}(i)} y(s) = \frac{1}{|S|}\sum_{s \in S} y(s) = \mu_y(S)
\]
This finishes the proof.
\end{proof}

\subsection{Between-group mean square sum and codegeneracy morphisms}\label{ssec:between-group-mss-and-codegenerary-morphims}
In this section, we re-formalize the well-known concept of between-group mean square sums \cite[page 47]{Oehlert} (see Definition \ref{def:between-group-variance}) and show that these satisfy a certain functorial property on the category $\mathbf{Data}$ (see Proposition \ref{prop:between-group-variance:functoriality}). We take advantage of this functoriality to characterize their numerical variations relative to morphisms in $\mathbf{Data}$ by only looking at a specific type of morphisms, called \emph{degeneracy morphisms} (see Definition \ref{def:codegeneracy-arrows}). Importantly, we relate codegeneracy morphisms to the construction of optimal solutions (Proposition \ref{prop:contractions-codegeneracy-morphisms}). This gives us, in section \ref{sec:association-studies}, a way to control the construction of optimal and minimal solutions statistically.

\begin{definition}[Between-group mean square sums]\label{def:between-group-variance}
For every object $(g,v,c)$ of $\mathbf{Data}$, we define the \emph{between-group mean square sum of $(g,v,c)$} as the following quantity:
\[
\eta(g,v,c) = 
\left\{
\begin{array}{ll}
\displaystyle0&\textrm{if }g=1\\
\displaystyle\frac{1}{g-1}\sum_{i=1}^{g}c_{i}(v_{i}-\mu(g,c,v))^2&\textrm{if }g > 1\\
\end{array}
\right.
\]
\end{definition}

\begin{remark}[A unique formula]
It follows from Definition \ref{def:cardinal-mean} that if an object $(g,v,c)$ of $\mathbf{Data}$ is such that $g=1$, then the identity $\mu(g,c,v) = v_1$ holds. It then follows from Definition \ref{def:between-group-variance} that the following equation holds for every object $(g,v,c)$ of $\mathbf{Data}$.
\[
(g-1)\eta(g,v,c) = \sum_{i=1}^{g}c_{i}(v_{i}-\mu(g,c,v))^2
\]
\end{remark}

While the quantity introduce in Definition \ref{def:between-group-change} may seem formal, it carries interesting properties (Proposition \ref{prop:between-group-change:codegeneracy}) that will allow us to characterize changes in the significance of the outputs of ANOVA.

\begin{definition}[Between-group change]\label{def:between-group-change}
For every arrow $f:(g_1,v_1,c_1) \to (g_2,v_2,c_2)$ of $\mathbf{Data}$, we define the \emph{between-group change of $f$} as the following quantity:
\[
\varepsilon(f) = (g_1-1)\eta(g_1,v_1,c_1) - (g_2-1)\eta(g_2,v_2,c_2)
\]
\end{definition}

A direct consequence of the formula given in Definition \ref{def:between-group-change} is the functoriality of the between-group change. We will later use this functoriality to show that the between-group change always takes non-negative values.

\begin{proposition}[Functoriality]\label{prop:between-group-variance:functoriality}
The mapping $\varepsilon:f \mapsto \varepsilon(f)$ induces a functor from $\mathbf{Data}$ to the monoid $\mathbb{R}$ of real numbers. In others words, for every pair $f:(g_1,v_1,c_1) \to (g_2,v_2,c_2)$ and $g:(g_2,v_2,c_2) \to (g_3,v_3,c_3)$ of arrows in $\mathbf{Data}$, the following identity holds:
\[
\varepsilon(g \circ f) = \varepsilon(g) + \varepsilon(f).
\]
\end{proposition}
\begin{proof}
The equation of the statement directly follows from the formula of Definition \ref{def:between-group-change}.
\end{proof}

\begin{remark}[Computing a between-group change]\label{rem:computing-between-group-change-panda}
It follows from Proposition \ref{prop:between-group-variance:functoriality} that we can decompose the between-group change of an arrow $f$ in $\mathbf{Data}$ as a sum 
\[
\varepsilon(f_1) + \varepsilon(f_2) + \dots + \varepsilon(f_n)
\]
if the equation $f = f_n \circ \dots \circ f_2 \circ f_1$ holds in $\mathbf{Data}$. For instance, the arrows $f_1,f_2,\dots,f_n$ could all be such that only one of their fibers contains 2 elements, and the others only contain one element. This kind of decomposition is particularly relevant when sending sequences of contraction arrows (Proposition \ref{prop:contraction:canonical-arrow}) through a composition of functors as shown below:
\[
\xymatrix@C+10pt{
\mathbf{UP}(S) \ar@{-->}[rd]\ar[r]^{\mathsf{Anova}_y}& \mathbf{Data}\ar[d]^{\varepsilon}\\
&\mathbb{R}
}
\]
Proposition \ref{prop:between-group-change:codegeneracy} (given below) takes care of computing the between-group changes for such arrows, which we formally define in Definition \ref{def:codegeneracy-arrows}.
\end{remark}

In section \ref{ssec:ANOVA-functor}, we mentioned that the category $\mathbf{Data}$ can be seen as a more flexible setting than categories of unlabeled partitions. In particular, concepts holding in categories of partitions should be, to some extent, extendable to the category $\mathbf{Data}$. Below, from Definition \ref{def:codegeneracy-arrows} to Proposition \ref{prop:contractions-codegeneracy-morphisms}, our goal is to find a notion of arrow in $\mathbf{Data}$ that extends the concept of contractions (see Definition \ref{def:contraction}) in categories of unlabeled partitions.

\begin{definition}[Codegeneracy morphisms]\label{def:codegeneracy-arrows}
A morphism $f:(g_1,v_1,c_1) \to (g_2,v_2,c_2)$ in $\mathbf{Data}$ is called a \emph{codegeneracy morphism} if the identity $g_1 = g_2+1$ holds (also, see Proposition \ref{prop:codegeneracy-arrows:equivalent}).
\end{definition}

\begin{proposition}[Equivalent reformulation]\label{prop:codegeneracy-arrows:equivalent}
A morphism $f:(g_1,v_1,c_1) \to (g_2,v_2,c_2)$ in $\mathbf{Data}$ is a codegeneracy morphism if, and only if, there exists $i_0 \in [g_2]$ such that the underlying function $f:[g_1] \to [g_2]$ satisfies the following conditions:
\[
|f^{-1}(i_0)| = 2,\quad\quad\quad
\forall\,i \in [g_2] \backslash \{i_0\}: |f^{-1}(i)| = 1.
\]
If $f^{-1}(i_0)$ is of the form $\{j_1,j_2\}$, then we say that $f$ \emph{contracts the pair $(j_1,j_2)$}.
\end{proposition}
\begin{proof}
For every morphism $f:(g_1,v_1,c_1) \to (g_2,v_2,c_2)$ in $\mathbf{Data}$ and every $i_0 \in [g_2]$, we have:
\begin{equation}\label{eq:codegeneracy-arrows:equivalent}
g_1 = |\{1,2,\dots,g_1\}| = \left|\bigcup_{i = 1}^{g_2} f^{-1}(i)\right| = \sum_{i = 1}^{g_2} |f^{-1}(i)| = |f^{-1}(i_0)| + \sum_{i \in [g_2] \backslash \{i_0\}} |f^{-1}(i)|
\end{equation}
In addition, because the underlying function $f:[g_1] \to [g_2]$ is surjective, each fiber of $f$ is non-empty, which implies that the inequality $|f^{-1}(i)|\geq 1$ holds for every $i \in [g_2]$.

\fbox{$\Rightarrow$} If $f$ is a codegeneracy morphism contracting a pair to an element $i_0$, then the equations of (\ref{eq:codegeneracy-arrows:equivalent}) imply that $g_1 = g_2+1$. 

\fbox{$\Leftarrow$} We shows that converse by using a proof by contradiction. Suppose that the identity $g_1 = g_2+1$ holds and suppose that $f$ is not a degeneracy morphism. Since the inequality $|f^{-1}(i)| \geq 1$ holds for every $i \in [g_2]$, there must either be two elements $i_1 \neq i_0$ such that $|f^{-1}(i_1)| \geq 2$ and $|f^{-1}(i_1)| \geq 2$ or be an element $i_0$ such that $|f^{-1}(i)| > 2$. In any case, the relations of (\ref{eq:codegeneracy-arrows:equivalent}) imply that $g_1$ is greater than $g_2+1$, which contradicts our assumption.
\end{proof}

The following proposition shows that codegeneracy morphisms corresponds to contraction morphisms through the lens of the functor $\mathsf{Anova}$.

\begin{proposition}[Contractions and codegeneracy morphisms]\label{prop:contractions-codegeneracy-morphisms}
Let $(S,<)$ be a finite strict linear order, $y:S \to\mathbb{R}$ be a function and $x:S \to [p]$ be a partition in $\mathbf{UP}(S)$. For every pair $(j_1,j_2) \in [p]\times [p]$, the image of the canonical arrow $x \to x[j_1,j_2]$ (see Proposition \ref{prop:contraction:canonical-arrow}) through the functor $\mathsf{Anova}_{y}$ is a codegeneracy morphism that contracts the pair $(j_1,j_2)$.
\end{proposition}
\begin{proof}
Because the collection of fibers of $x$ is pairwise disjoint, the formula of Definition \ref{def:contraction} and Proposition \ref{prop:familial-product:unlabeled-partitions} imply that the partition $x[j_1,j_2]$ is of the form $\delta^{*}(x')$ where $x'$ is a collection such that there exists $i_0 \in [p-1]$ for which 
\begin{itemize}
\item[-] the identity $x'_{i_0} = x^{-1}(j_1)\cup x^{-1}(j_2)$ holds
\item[-]  and, for every $i \in [p-1]\backslash \{i_0\}$, there exists an element $j_i \in [p]$ such that $x'_{i} = x^{-1}(j_i)$. 
\end{itemize}
By Convention \ref{conv:relative_discrete_partitions:familial} and Convention \ref{conv:relative_discrete_partitions_delta}, the unlabeled partition $x[j_1,j_2] = \delta^{*}(x')$ can be described as a surjection $y:S \to [p-1]$ whose fibers $\mathsf{fib}(y) = \{y_k\}_{k \in [p-1]}$ are such that there exists $k_0 \in [p-1]$ for which the identity
\[
y_{k_0} = x'_{i_0} = x^{-1}(j_1)\cup x^{-1}(j_2)
\]
holds and, for every $k \in [p-1]\backslash \{k_0\}$, there exists an element $i_k \in [p]$ such that 
\[
y_{k} = x'_{i_k} = x^{-1}(j_{i_k}).
\]
The function defined by the mappings $j_1,j_2 \mapsto k_0$ and $k \mapsto j_{i_k}$, call this function $f$, satisfies the equation $f \circ x = y$ and thus defines a morphism $f:x \to y$ in $\mathbf{UP}(S)$ (see Definition \ref{def:unlabeled_partitions} and Definition \ref{def:labeled-partitions}).
By Remark \ref{rem:pre-order-LP}, the arrow $f:x \to x'$ must be the canonical arrow $x \to x[j_1,j_2]$, which means that the image of $x \to x[j_1,j_2]$ through the functor $\mathsf{Anova}_{y}$ is the function $f:[p] \to [p-1]$. It directly follows from Definition \ref{def:codegeneracy-arrows} (and Proposition \ref{prop:codegeneracy-arrows:equivalent}) that the image of the canonical arrow $x \to x[j_1,j_2]$ through the functor $\mathsf{Anova}_{y}$ is a codegeneracy morphism contracting $(j_1,j_2)$.
\end{proof}

In the rest of this section, we characterize the values taken by between-group changes. To do so, we express the morphisms of $\mathbf{Data}$ as compositions of codegeneracy morphisms and we take advantage of the functoriality property shown in Proposition \ref{prop:between-group-variance:functoriality} to show that the between-group change of any morphism in $\mathbf{Data}$ is non-negative.

\begin{proposition}[Codegeneracy morphisms]\label{prop:between-group-change:codegeneracy}
If $f:(g_1,v_1,c_1) \to (g_2,v_2,c_2)$ is a codegeneracy morphism of $\mathbf{Data}$ that contracts a pair $(j_0,k_0)$, then the following identity holds:
\[
\varepsilon(f) =\frac{c_{1,j_0}c_{1,k_0}}{c_{1,j_0}+c_{1,k_0}}(v_{1,j_0} - v_{1,k_0})^2
\]
\end{proposition}
\begin{proof}
We will suppose that $f$ contracts the pair $(j_0,k_0)$ to an element $i_0 \in [g_2]$. We start with a few observations. First, by Proposition \ref{prop:invariance:1}, we can denote the two quantities $\mu(g_1,c_1,v_1)$ and $\mu(g_2,c_2,v_2)$ as $\mu$. Also, because $f:[g_1] \to [g_2]$ is a function, the sum $\sum_{i=1}^{g_2}m_{i,j}(f)$ must be equal to 1 (see Definition \ref{def:fiber-matrix}). As a result, we have the following identities:
\begin{align*}
\varepsilon(f) & = \sum_{j=1}^{g_1}c_{1,j}(v_{1,j}-\mu)^2 - \sum_{i=1}^{g_2}c_{2,i}(v_{2,i}-\mu)^2\\
&=  \sum_{j =1}^{g_1}\Big(\sum_{i=1}^{g_2}m_{i,j}(f)\Big)c_{1,j}(v_{1,j}-\mu)^2 - \sum_{i=1}^{g_2}c_{2,i}(v_{2,i}-\mu)^2 \\
& = \sum_{i=1}^{g_2}\sum_{j =1}^{g_1}m_{i,j}(f)c_{1,j}(v_{1,j}-\mu)^2 - \sum_{i=1}^{g_2}c_{2,i}(v_{2,i}-\mu)^2 
\end{align*}
We can rearrange the previous equation to obtain the following expression for $\varepsilon(f)$:
\begin{equation}\label{eq:between-group-change:codegeneracy}
\varepsilon(f) = \sum_{i=1}^{g_2} \Big(\Big(\sum_{j =1}^{g_1}m_{i,j}(f)c_{1,j}(v_{1,j}-\mu)^2\Big)- c_{2,i}(v_{2,i}-\mu)^2 \Big)
\end{equation}
Note that the summands of sum (\ref{eq:between-group-change:codegeneracy}) at any index $i$ different from $i_0$ must be zero. Indeed, if $i \neq i_0$, then we have $f^{-1}(i) = \{j_1\}$ and Definition \ref{def:fiber-matrix} implies that the following equations hold:
\[
c_{2,i} = \sum_{j=1}^{g_1} m_{i,j}(f) c_{1,j} = c_{1,j_1}
\quad\quad\quad
v_{2,i} = \frac{1}{c_{2,i}}\sum_{j=1}^{g_1} m_{i,j}(f) c_{1,j}v_{1,j} = \frac{c_{1,j_1}}{c_{2,i}} v_{1,j_1} = v_{1,j_1}
\]
We also deduce the following equations from the previous ones.
\[
\sum_{j =1}^{g_1}m_{i,j}(f)c_{1,j}(v_{1,j}-\mu)^2 = c_{1,j_1}(v_{1,j_1}-\mu)^2 = c_{2,i}(v_{2,i}-\mu)^2
\]
Hence, the summand of sum (\ref{eq:between-group-change:codegeneracy}) at $i \neq i_0$ is zero and equation (\ref{eq:between-group-change:codegeneracy}) can be rewritten as follows:
\begin{equation}\label{eq:between-group-change:change:short_formula_with_mu}
\varepsilon(f) = c_{1,j_0}(v_{1,j_0}-\mu)^2 + c_{1,k_0}(v_{1,k_0}-\mu)^2 - c_{2,i_0}(v_{2,i_0}-\mu)^2
\end{equation}
Now, by Definition \ref{def:data-category}, observe that we have the relations $c_{2,i_0} = c_{1,j_0} + c_{1,k_0}$ and $c_{2,i_0}v_{2,i_0} = c_{1,j_0}v_{1,j_0} + c_{1,k_0}v_{1,k_0}$. We can expend expression (\ref{eq:between-group-change:change:short_formula_with_mu}) as a quadratic in $\mu$ and use the two previous formulas to cancel out the terms in $\mu$ and $\mu^2$ as follows:
\begin{equation}\label{eq:between-group-change:change:short_formula}
\varepsilon(f) = c_{1,j_0}v_{1,j_0}^2 + c_{1,k_0}v_{1,k_0}^2 - c_{2,i}v_{2,i}^2 
\end{equation}
Futhermore, we can substitute $c_{2,i_0}$ and $v_{2,i_0}$ in equation (\ref{eq:between-group-change:change:short_formula}) for the relations $c_{2,i_0} = c_{1,j_0} + c_{1,k_0}$ and $c_{2,i_0}v_{2,i_0} = c_{1,j_0}v_{1,j_0} + c_{1,k_0}v_{1,k_0}$, then expand and simply, to obtain the following expressions:
\[
\varepsilon(f)  = c_{1,j_0}v_{1,j_0}^2 + c_{1,k_0}v_{1,k_0}^2 - \frac{(c_{1,j_0}v_{1,j_0} + c_{1,k_0}v_{1,k_0})^2}{c_{1,j_0} + c_{1,k_0}} = \frac{c_{1,j_0}c_{1,k_0}}{c_{1,j_0}+c_{1,k_0}}(v_{1,j_0} - v_{1,k_0})^2
\]
This last equation finishes the proof.
\end{proof}

\begin{proposition}[Factorization]\label{prop:factorization}
For every morphism $f:(g_1,v_1,c_1) \to (g_2,v_2,c_2)$ in $\mathbf{Data}$ such that $g_1>g_2+1$, there exists a codegeneracy morphism $h:(g_1,v_1,c_1) \to (g_3,v_3,c_3)$ and a morphism $f':(g_3,v_3,c_3) \to (g_2,v_2,c_2)$ in $\mathbf{Data}$ for which the equation $f = f' \circ h$ holds.
\end{proposition}
\begin{proof}
Since $g_1 > g_2+1$, the function $f:[g_1] \to [g_2]$ admits a fiber $f^{-1}(i_0)$ that contains at least two elements $j_1<j_2$. Hence, the following two functions $h:[g_1] \to [g_1-1]$ and $f':[g_1-1] \to [g_2]$ are well-defined:
\[
h:\left(
\begin{array}{llll}
j&\mapsto&j&\textrm{if }j < j_2\\
j_2&\mapsto&j_1&\\
j&\mapsto&j-1&\textrm{if }j > j_2\\
\end{array}
\right)
\quad\quad\quad
f':\left(
\begin{array}{llll}
j&\mapsto&f(j)&\textrm{if }j < j_2\\
j&\mapsto&f(j+1)&\textrm{if }j \geq j_2\\
\end{array}
\right)
\]
It is straightforward to see that the composition $f' \circ h$ is equal to the function $f$. FInally, we need to check that $f'$ and $h$ are morphism in $\mathbf{Data}$. First, because $f$ is a morphism in $\mathbf{Data}$ (Definition \ref{def:data-category}), we have the equation $c_2 = \mathsf{M}(f)c_1$. By Proposition \ref{prop:functoriality-fiber-matrix}, we also have the decomposition 
\[
c_2 = \mathsf{M}(f)c_1 = \mathsf{M}(f' \circ h)c_1  = \mathsf{M}(f') ( \mathsf{M}(h)c_1)
\]
and if we denote the vector $\mathsf{M}(h)c_1 \in \mathbb{N}_{*}^{q_1-1}$ as $c_3$, then we have the two relations $c_3 = \mathsf{M}(h)c_1$ and $c_2 = \mathsf{M}(f')c_3$. Since $c_3$ is in $\mathbb{N}_{*}^{q_1-1}$, each of its components $(c_3)_i$ is positive. Hence, the vector $v_3 \in \mathbb{R}^{g_1-1}$ determined by the following equation for every $i \in [g_1-1]$ is well-defined:
\[
(v_3)_i = \frac{1}{(c_3)_i}(\mathsf{M}(h)(v_1 \cdot c_1))_i
\]
Note that the previous equation is equivalent to the equation $v_3\cdot c_3 = \mathsf{M}(h)(v_1 \cdot c_1)$. This equation also implies that the following identities:
\[
c_2v_2 = \mathsf{M}(f)(v_1 \cdot c_1) = \mathsf{M}(f')\mathsf{M}(h)(v_1 \cdot c_1) = \mathsf{M}(h)(v_3 \cdot c_3)
\]
If we denote $g_1-1$ as $g_3$, then we have constructed two morphisms $h:(g_1,v_1,c_1) \to (g_3,v_3,c_3)$ and $f':(g_3,v_3,c_3) \to (g_2,v_2,c_2)$ in $\mathbf{Data}$ for which the equation $f = f' \circ h$ holds. Because $g_1 = g_3+1$, the morphism $f$ is a codegeneracy morphism.
\end{proof}

\begin{remark}[Changes are non-negative]\label{rem:changes-non-negative}
A proof by induction using the result of Proposition \ref{prop:factorization} shows that every morphism $f$ of $\mathbf{Data}$ is either an isomorphism or a composition $f_n \circ \dots \circ f_2 \circ f_1$ of codegeneracy morphisms $f_1,f_2,\dots,f_n$. It then follows from Proposition \ref{prop:between-group-change:codegeneracy} and Proposition \ref{prop:between-group-variance:functoriality} that the quantity $\varepsilon(f)$ is non-negative (see Remark \ref{rem:computing-between-group-change-panda}).
\end{remark}

\subsection{\texorpdfstring{$\mathsf{F}$}{}-ratios and their functorial properties}\label{ssec:F-ratios}
In this section, we re-formalize the concept of within-group mean square sums \cite[page 47]{Oehlert} (see Definition \ref{def:within-group-variance}) and use it with that of between-group mean square sums (see Definition \ref{def:within-group-variance}) to recover the well-known concept of $\mathsf{F}$-ratio associated with ANOVA \cite[page 48]{Oehlert} (see Definition \ref{def:F-ratio}). Finally, we use the functorial properties shown in section \ref{ssec:between-group-mss-and-codegenerary-morphims} to characterize the numerical variations of $\mathsf{F}$-ratios (see Proposition \ref{prop:changes-F-ratio}).

\begin{convention}[Notation]
Let $(S,<)$ be a finite strict linear order and $y:S \to \mathbb{R}$ be a function. From now on, for every partition $p$ in $\mathbf{UP}(S)$ and every morphism $f:p_1 \to p_2$ in $\mathbf{UP}(S)$, we will use the following notations:
\[
\eta_y(p_1) := \eta(\mathsf{Anova}_y(p_1)),
\quad\quad
\quad\quad
\varepsilon_y(f) := \varepsilon(\mathsf{Anova}_y(f)).
\]
\end{convention}

The following definition uses a well-known implicit equation to characterize the usual concept of within-group mean square. We show that we can recover the text book definition of within-group mean squares \cite[page 48]{Oehlert} in Proposition \ref{prop:formula-within-group-mss}.

\begin{definition}[Within-group mean square sums]\label{def:within-group-variance}
Let $(S,<)$ be a finite strict linear order and $y:S \to \mathbb{R}$ be a function. For every partition $p:S \to [g]$ in $\mathbf{UP}(S)$ whose underlying surjection is not a bijection, we will call the \emph{within-group mean square sum} the real number $\rho_y(p)$ determined by the following equation:
\[
|S|\sigma^2_y(S) = (|S|-g)\rho_y(p) + (g-1) \eta_y(p)
\]
If the surjection $p:S \to [g]$ is a bijection, then we let $\rho_y(p)$ be zero.
\end{definition}

\begin{proposition}[A formula for within-group mean square sums]\label{prop:formula-within-group-mss}
Let $(S,<)$ be a finite strict linear order and $y:S \to \mathbb{R}$ be a function. For every partition $p:S \to [g]$ in $\mathbf{UP}(S)$ whose underlying surjection is not a bijection, the following identity holds:
\[
\rho_y(p) = \frac{1}{|S|-g}\sum_{i =1}^g \sum_{s \in p^{-1}(i)}\Big(y(s) - \mu_y(p^{-1}(i))\Big)^2
\]
The previous equation implies that $\rho_y(p)$ is non-negative for every partition $p$ in $\mathbf{UP}(S)$.
\end{proposition}
\begin{proof}
By Definition \ref{def:within-group-variance}, we have the following formula:
\[
(|S|-g)\rho_y(p) = |S|\sigma^2_y(S) - (g-1) \eta_y(p)
\]
By Definition \ref{def:cardinal-mean-variance} and Definition \ref{def:between-group-variance}, this means that the following identity holds:
\begin{equation}\label{eq:formula-within-group-mss}
(|S|-g)\rho_y(p) = \sum_{s \in S} (y(s) - \mu_y(S))^2 - \sum_{i=1}^{g}\Big(\mu_y(p^{-1}(i))-\mu(\mathsf{Anova}_y(p))\Big)^2
\end{equation}
From now on, the proof focuses on re-writing the leftmost sum of (\ref{eq:formula-within-group-mss}). To start with, observe that this sum can be re-written as:
\begin{equation}\label{eq:intermedite-formula-mss}
\sum_{s \in S} (y(s) - \mu_y(S))^2 = \sum_{i =1}^g \sum_{s \in p^{-1}(i)} (y(s) - \mu_y(S))^2
\end{equation}
For every element $s \in p^{-1}(i)$, we can use the decomposition
\[
y(s) - \mu_y(S) = \Big(y(s) - \mu_y(p^{-1}(i))\Big) + \Big(\mu_y(p^{-1}(i)) - \mu_y(S)\Big)
\]
in the expression of (\ref{eq:intermedite-formula-mss}) and expand the square according to this decomposition to obtain: 
\begin{align*}
\sum_{s \in S} (y(s) - \mu_y(S))^2 & = \sum_{i =1}^g \sum_{s \in p^{-1}(i)} \Big(y(s) - \mu_y(p^{-1}(i))\Big)^2 +  \sum_{i =1}^g \sum_{s \in p^{-1}(i)}\Big(\mu_y(p^{-1}(i)) - \mu_y(S)\Big)^2\\
& +  \sum_{i =1}^g \sum_{s \in p^{-1}(i)} 2\Big(y(s) - \mu_y(p^{-1}(i))\Big)\Big(\mu_y(p^{-1}(i)) - \mu_y(S)\Big)
\end{align*}
The last sum of the previous expression can be factorized as follows:
\[
\sum_{i =1}^g2\Big(\mu_y(p^{-1}(i)) - \mu_y(S)\Big) \sum_{s \in p^{-1}(i)} \Big(y(s) - \mu_y(p^{-1}(i))\Big)
\]
and, by Definition \ref{def:cardinal-mean-variance}, the leftmost sum of the previous expression is equal to zero, as shown below:
\[
\sum_{s \in p^{-1}(i)} \Big(y(s) - \mu_y(p^{-1}(i))\Big) = \Big(\sum_{s \in p^{-1}(i)}y(s)\Big) - |p^{-1}(i)| \times \mu_y(p^{-1}(i)) = 0
\]
In other words, we have shown that the following identity holds:
\[
\sum_{s \in S} (y(s) - \mu_y(S))^2 = \sum_{i =1}^g \sum_{s \in p^{-1}(i)} \Big(y(s) - \mu_y(p^{-1}(i))\Big)^2 +  \sum_{i =1}^g \sum_{s \in p^{-1}(i)}\Big(\mu_y(p^{-1}(i)) - \mu_y(S)\Big)^2
\]
By Proposition \ref{prop:invariance:2}, we can replace $\mu_y(S)$ with $\mu(\mathsf{Anova}_y(p))$. Doing so and combining the previous equation with  (\ref{eq:formula-within-group-mss}) gives us the following expression:
\[
(|S|-g)\rho_y(p)= \sum_{i =1}^g \sum_{s \in p^{-1}(i)} \Big(y(s) - \mu_y(p^{-1}(i))\Big)^2
\]
This proves the statement when $|S| - g > 0$, hence when $p$ is not a bijection.
\end{proof}

We now introduce $\mathsf{F}$-ratios, which we shall use to measure the likelihood of a given partition to have arisen from noisy or structured data.

\begin{definition}[$\mathsf{F}$-ratio]\label{def:F-ratio}
Let $(S,<)$ be a finite strict linear order and $y:S \to \mathbb{R}$ be a function. For every partition $p:S \to [g]$ in $\mathbf{UP}(S)$ for which $\rho_y(p)$ is non-zero, we define the \emph{$\mathsf{F}$-ratio of $p$} as the ratio:
\[
\mathsf{F}_y(p) = \frac{\eta_y(p)}{\rho_y(p)}
\]
\end{definition}

For convenience, we introduce two operations $\kappa$ and $\nu$. The reader should be able to see that the latter is somewhat close to the value 1 -- this will be relevant to interpret the result of Proposition \ref{prop:changes-F-ratio}.

\begin{convention}[Notation]\label{conv:kappa-nu}
Let $(S,<)$ be a finite strict linear order and $y:S \to \mathbb{R}$ be a function. For every partition $p:S \to [g]$ in $\mathbf{UP}(S)$ for which $g > 1$ and every morphism $f:p_1 \to p_2$ in $\mathbf{UP}(S)$ for which $p_1:S \to [g_1]$ and $p_2:S \to [g_2]$ and $1<g_2\leq g_1<|S|$, we will use the following notations:
\[
\kappa(p) = \frac{(|S|-g)}{(g-1)}
\quad\quad\quad
\nu(f) = \frac{\kappa(p_2)}{\kappa(p_1)} = \Big(1+\frac{g_1-g_2}{g_2-1}\Big)\Big(1+\frac{g_1-g_2}{|S|-g_1}\Big)
\]
In particular, the two operations $\kappa$ and $\nu$ take non-negative values.
\end{convention}

The following proposition implies that a morphism of partition will likely decrease the $\mathsf{F}$-ratio of the source if the morphism is associated with a high between+group change. The idea behind this result is that we can control the $\mathsf{F}$-ratio of the reduced element (Definition \ref{def:reduced-elements}) of a reduction (Definition \ref{def:reduction}) by controlling the between-group changes of its contractions through the formula of Proposition \ref{prop:between-group-change:codegeneracy}. 

\begin{proposition}[Changes in $\mathsf{F}$-ratios]\label{prop:changes-F-ratio}
Let $(S,<)$ be a finite strict linear order and $y:S \to \mathbb{R}$ be a function. For every morphism $f:p_1 \to p_2$ in $\mathbf{UP}(S)$ for which $p_1:S \to [g_1]$ and $p_2:S \to [g_2]$ and $1<g_2 \leq g_1<|S|$, the following relation holds:
\[
\nu(f)\mathsf{F}_y(p_1) - \frac{\varepsilon_y(f)}{|S|\sigma^2_y(S)}\kappa(p_2) \geq \mathsf{F}_y(p_2) 
\]
\end{proposition}
\begin{proof}
First, Definition \ref{def:within-group-variance} gives us the following identities for the two partitions $p_1$ and $p_2$:
\begin{equation}\label{eq:Changes-F-ratio:1}
\left\{
\begin{array}{ll}
|S|\sigma^2_y(S) &= (|S|-g_1)\rho_y(p_1) + (g_1-1) \eta_y(p_1),\\
|S|\sigma^2_y(S) &= (|S|-g_2)\rho_y(p_2) + (g_2-1) \eta_y(p_2).
\end{array}
\right.
\end{equation}
By Definition \ref{def:between-group-change}, subtracting the topmost equation of (\ref{eq:Changes-F-ratio:1}) from the bottommost one gives us the topmost relation of (\ref{eq:Changes-F-ratio:2}), given below. The other relation of (\ref{eq:Changes-F-ratio:2}) is a rearrangement of the equation given in Definition \ref{def:between-group-change}:
\begin{equation}\label{eq:Changes-F-ratio:2}
\left\{
\begin{array}{rl}
(|S|-g_1)\rho_y(p_1)  & = (|S|-g_2)\rho_y(p_2) - \varepsilon_y(f)\\
(g_1-1) \eta_y(p_1)   & = (g_2-1) \eta_y(p_2) + \varepsilon_y(f)
\end{array}
\right.
\end{equation}
We can now use the two equations of (\ref{eq:Changes-F-ratio:2}) and the notations of Convention \ref{conv:kappa-nu} to rewrite the expression of the $\mathsf{F}$-ratio $\mathsf{F}_y(p_1)$ (see Definition \ref{def:F-ratio}). Specifically, we have the following relation:
\[
\frac{\mathsf{F}_y(p_1)}{\kappa(p_1)}  = \frac{(g_2-1) \eta_y(p_2) + \varepsilon_y(f)}{(|S|-g_2)\rho_y(p_2) - \varepsilon_y(f)}
\]
Subtracting the previous relation by the ratio $\mathsf{F}_y(p_2)/\kappa(p_2)$ and using the expression given by Definition \ref{def:F-ratio} for the right-hand side gives us the following identity:
\begin{equation}\label{eq:Changes-F-ratio:3}
\frac{\mathsf{F}_y(p_1)}{\kappa(p_1)} - \frac{\mathsf{F}_y(p_2)}{\kappa(p_2)}  = \frac{(g_2-1) \eta_y(p_2) + \varepsilon_y(f)}{(|S|-g_2)\rho_y(p_2) - \varepsilon_y(f)} - \frac{(g_2-1) \eta_y(p_2)}{(|S|-g_2)\rho_y(p_2)}
\end{equation}
A straightforward simplification of the right-hand side of equation (\ref{eq:Changes-F-ratio:3}) gives us the following identity:
\[
\frac{\mathsf{F}_y(p_1)}{\kappa(p_1)} - \frac{\mathsf{F}_y(p_2)}{\kappa(p_2)}  = \frac{\varepsilon_y(f)(|S|-g_2)\rho_y(p_2)+ \varepsilon_y(f)(g_2-1) \eta_y(p_2)}{\big((|S|-g_2)\rho_y(p_2) - \varepsilon_y(f)\big)(|S|-g_2)\rho_y(p_2)}
\]
Factorizing the dividend of the right-hand side fraction of the previous expression by $\varepsilon_y(f)$ and using the bottommost expression of (\ref{eq:Changes-F-ratio:1}) gives us the following simplification:
\[
\frac{\mathsf{F}_y(p_1)}{\kappa(p_1)} - \frac{\mathsf{F}_y(p_2)}{\kappa(p_2)}  = \frac{\varepsilon_y(f)|S|\sigma^2_y(S)}{\big((|S|-g_2)\rho_y(p_2) - \varepsilon_y(f)\big)(|S|-g_2)\rho_y(p_2)}
\]
Finally, because the operations $\rho_y$ and $\varepsilon_y$ only take non-negative values (see Proposition \ref{prop:formula-within-group-mss} and Remark \ref{rem:changes-non-negative}), the divisor of the right-hand side fraction of the previous expression is always less than the quantity $(|S|-g_2)^2\rho_y(p_2)^2$. As a result, we obtain the following lower bound:
\begin{equation}\label{eq:Changes-F-ratio:4}
\frac{\mathsf{F}_y(p_1)}{\kappa(p_1)} - \frac{\mathsf{F}_y(p_2)}{\kappa(p_2)}  \geq \frac{\varepsilon_y(f)|S|\sigma^2_y(S)}{(|S|-g_2)^2\rho_y(p_2)^2}
\end{equation}
Now, because the operations $\eta_y$ and $\rho_y$ take positive values (see Definition \ref{def:between-group-variance} and Proposition \ref{prop:formula-within-group-mss}), we deduce from the bottommost equation of (\ref{eq:Changes-F-ratio:1}) that the relation $(|S|-g_2)\rho_y(p_2) \leq |S|\sigma^2_y(S)$ holds. Using this upper bound in the right-hand side term of (\ref{eq:Changes-F-ratio:4}) gives us the following inequality:
\[
\frac{\mathsf{F}_y(p_1)}{\kappa(p_1)} - \frac{\mathsf{F}_y(p_2)}{\kappa(p_2)}  \geq \frac{\varepsilon_y(f)}{|S|\sigma^2_y(S)}
\]
Since $\kappa$ takes positive values (see Convention \ref{conv:kappa-nu}), the previous inequality is equivalent to the inequality of the statement up to rearrangement.
\end{proof}

In section \ref{ssec:combinatorial-GWAS}, we will use the statement of Proposition \ref{prop:changes-F-ratio} to construct reductions (see Definition \ref{def:reduction}) whose reduced elements are associated with relatively high $\mathsf{F}$-ratio values.

\section{Association studies}\label{sec:association-studies}
The present section shows how we can use the results of the previous sections to design a combinatorial GWAS. To do so, we will use the language of segments and pedigrads \cite{CTGI}, which we will slightly simplify for the scope of this article. We will use the full expressive power of pedigrads in a subsequent article, in which we will refine our combinatorial GWAS to a more advanced type of analysis.

\subsection{Plain segments}\label{ssec:plain-segments}
In this section, we define ``plain segments'' (see Definition \ref{def:plain-segments}), which are a subtype of segments, as defined in \cite[Sect. 2.14]{CTGI}. We associate these plain segments with a sum operation (see Definition \ref{def:sums:segments}) that will play an important role in the design of our combinatorial GWAS algorithm.

\begin{definition}[Plain segments]\label{def:plain-segments}
For every non-negative integer $N$, we denote by $\mathbf{PS}(N)$ the category whose objects are functions $\tau:[N] \to \{0,1\}$, which we will call \emph{plain segments}, and whose arrows $\tau_1 \to \tau_2$ are given by relations $\tau_2(i) \leq \tau_1(i)$ for every integer $i \in [N]$. Note that this makes the category $\mathbf{PS}(N)$ a pre-ordered set (also referred to as a `pre-order category').
\end{definition}

While the following example focuses on a genetic interpretation of segments, note that segments can be more generally interpreted as any type of ordered data. Since data is, in practice, often ordered or indexed, this makes segments very general objects to encode data variables.

\begin{example}[Plain segments]
Note that a segment $\tau:[N] \to \{0,1\}$ in $\mathbf{PS}(N)$ can be described as a tuple $(\tau(1),\tau(2),\dots,\tau(N))$ of its images $\tau(i)$ -- we will use this representation for all our examples. 

Now, the point of a segment $\tau \in \mathbf{PS}(N)$ is to tell us, through its values 1, which traits are to be considered simultaneously during our analysis. In the case of our combinatorial GWAS, these traits could be nucleotides, of which we would like to measure the combined effects. For instance, the segment of $\mathbf{PS}(22)$ given below, on the left, specifies that we only want to look at the set of traits enumerated on the right.
\[
\tau = (0,0,0,0,0,1,1,1,1,0,1,1,0,1,0,1,0,0,0,0,0,0) \Leftrightarrow 
\left\{
\begin{array}{l}
\textrm{We look at the traits:}\\
6,7,8,9,11,12,14,16\\
\end{array}
\right.
\]
The morphism of $\mathbf{PS}(N)$ allow us to reduce the number of values 1 in a segment. For instance, we have an arrow $(1,1,1,1,1,0,1) \to (0,0,0,1,0,0,1)$ in $\mathbf{PS}(7)$.
\end{example}

\begin{definition}[Global plain segments]\label{def:global-plain-segments}
For every non-negative integer $N$, we will denote by $\mathbf{1}_N$ the segment defined by the constant function $[N] \to \{0,1\}$ mapping every $i\in [N]$ to $1$.
\end{definition}

\begin{proposition}[Global plain segments]\label{prop:global-plain-segments}
For every non-negative integer $N$, the segment $\mathbf{1}_N \in \mathbf{PS}(N)$ is initial, meaning that for every segment $\tau \in \mathbf{PS}(N)$, there is a unique arrow $!_{\tau}:\mathbf{1}_N \to \tau$.
\end{proposition}
\begin{proof}
Directly follows from Definition \ref{def:plain-segments}.
\end{proof}

We now introduce a sum operation of plain segments. In the case of \cite{CTGI}, this operation would have been translated in terms of ``limit cones''.

\begin{definition}[Sums]\label{def:sums:segments}
For every non-negative integer $N$ and every pair $(\tau_1,\tau_2)$ of objects in $\mathbf{PS}(N)$, we denote by $\tau_1+\tau_2$ the segment $[N] \to \{0,1\}$ mapping every integer $i \in [N]$ to the maximum value of the set $\{\tau_1(i),\tau_2(i)\}$ (See Example \ref{exa:sums-exact-sums}).
\end{definition}

\begin{remark}[Cone property]\label{rem:cone-property}
For every non-negative integer $N$ and every pair $(\tau_1,\tau_2)$ of objects in $\mathbf{PS}(N)$, Definition \ref{def:sums:segments} implies that there exist two morphisms $\tau_1+\tau_2 \to \tau_1$ and $\tau_1+\tau_2 \to \tau_2$ in $\mathbf{PS}(N)$. We will later use these two arrows in Remark \ref{rem:cone-property-environment} and Proposition \ref{prop:pedigrad-property}.
\end{remark}

\begin{definition}[Exactness]\label{def:exactness}
For every non-negative integer $N$, any pair $(\tau_1,\tau_2)$ of objects in $\mathbf{PS}(N)$ is said to be \emph{exact} if the intersection $\tau_1^{-1}(1) \cap \tau_2^{-1}(1)$ is empty. By extension, if $(\tau_1,\tau_2)$ is exact, then we say that the sum $\tau_1+\tau_2$ is exact (See Example \ref{exa:sums-exact-sums}).
\end{definition}

\begin{example}[Sums and exact sums]\label{exa:sums-exact-sums}
If 1 and 0 are thought of as Boolean values, then sums and exact sums are equivalent to $\mathsf{OR}$ and $\mathsf{XOR}$ operations, respectively.
For instance, the following equation gives an example of a sum in $\mathbf{PS}(12)$
\[
(0,0,0,0,1,0,0,0,0,1,1,1) + (1,1,1,1,1,0,0,0,0,0,1,0) = (1,1,1,1,1,0,0,0,0,1,1,1)
\]
An exact sum is a sum for which one does not add two values 1 together. For instance, the following sum is exact in $\mathbf{PS}(12)$.
\[
(1,1,1,1,1,0,0,0,0,0,0,0) + (0,0,0,0,0,0,0,0,1,1,1,1) = (1,1,1,1,1,0,0,0,1,1,1,1)
\]
\end{example}

\subsection{Genotypes}\label{ssec:genotypes}
In the present section, we use the sum operation defined in section \ref{def:plain-segments} to define what we could think of as a \emph{pedigrad} \cite{CTGI} in the category of unlabeled partitions (see Proposition \ref{prop:pedigrad-property:genotype-partition}). This pedigrad turns out to model the collection of genotypic partitions that is naturally associated with a GWAS (see section \ref{ssec:motivations} above). We will denote by $\mathbf{Set}$ the category of sets and functions.

\begin{definition}[Truncation]\label{def:truncation}
For every non-negative integer $N$ and every segment $\tau \in \mathbf{PS}(N)$, we denote by $\mathsf{Tr}(\tau)$ the fiber $\tau^{-1}(1)$, which is a subset of $[N]$.
\end{definition}

\begin{proposition}[Truncation of sums]\label{prop:truncations-of-sums}
For every non-negative integer $N$ and every pair $(\tau_1,\tau_2)$ of objects in $\mathbf{PS}(N)$,
the following identity, for which we use the union of subsets of $[N]$, holds.
\[
\mathsf{Tr}(\tau_1+\tau_2) = \mathsf{Tr}(\tau_1) \cup \mathsf{Tr}(\tau_2)
\]
\end{proposition}
\begin{proof}
We need to show that the identity $(\tau_1+\tau_2)^{-1}(1) = \tau_1^{-1}(1) \cup \tau_2^{-1}(1)$ holds. We proceed by double induction. Let $i \in (\tau_1+\tau_2)^{-1}(1)$. This means that we have $(\tau_1+\tau_2)(i) = 1$. We deduce from Definition \ref{def:sums:segments} one of the following conditions hold:
\[
\textrm{($\tau_1(i) = 1$ and $\tau_2(i) = 0$);}\quad\quad
\textrm{($\tau_1(i) = 0$ and $\tau_2(i) = 1$);}\quad\quad
\textrm{($\tau_1(i) = 1$ and $\tau_2(i) = 1$).}
\]
As a result, we deduce that $i$ belongs to one of the sets $\tau_1^{-1}(1)$ or $\tau_2^{-1}(1)$, which shows the inclusion $(\tau_1+\tau_2)^{-1}(1) \subseteq \tau_1^{-1}(1) \cup \tau_2^{-1}(1)$.
The other inclusion is shown as follows: if $i \in \tau_1^{-1}(1) \cup \tau_2^{-1}(1)$, then the set $\{\tau_1(i),\tau_2(i)\}$ contains 1, which implies, by Definition \ref{def:sums:segments}, that $(\tau_1+\tau_2)(i) = 1$. This shows the other inclusion and finish the proof of the proposition.
\end{proof}

\begin{proposition}[Truncation functor]\label{prop:truncation-functor}
Let $N$ be a non-negative integer. For every morphism $f:\tau_1 \to \tau_2$, the inclusion $\tau_2^{-1}(1) \subseteq \tau_1^{-1}(1)$ holds. If we denote by $\mathsf{Tr}(f)$ the resulting injection $\mathsf{Tr}(\tau_2) \to \mathsf{Tr}(\tau_1)$, then the mapping $\mathsf{Tr}$ defines a functor from $\mathbf{PS}(N)$ to the opposite category $\mathbf{Set}^{\mathrm{op}}$ of sets and functions.
\end{proposition}
\begin{proof}
The morphism $f:\tau_1 \to \tau_2$ implies that the inequality $\tau_1(i) \leq \tau_2(i)$ holds for every $i \in [N]$. As a result, if $\tau_2(i) = 1$, then $\tau_1(i) = 1$, which means that we have an inclusion $\tau_2^{-1}(1) \subseteq \tau_1^{-1}(1)$. Since $\mathbf{PS}(N)$ is a pre-order category (see Remark \ref{def:plain-segments}) and inclusions between a given pair of sets are unique, the mapping $\mathsf{Tr}$ is necessarily functorial from $\mathbf{PS}(N)$ to $\mathbf{Set}^{\mathrm{op}}$.
\end{proof}

Below, we define a functor referred to as \emph{environment functor}. The name for this functor comes from the formalism of \cite{CTGI}, in which the word `environment' referred to the space of symbols needed to encode the space of traits considered in an analysis.

\begin{definition}[Environment functor]\label{def:environment-functor}
For every set $G$ and non-negative integer $N$, we denote by $\mathsf{E}_N(G)$ the functor resulting from the composition of the functor $\mathsf{Tr}:\mathbf{PS}(N) \to \mathbf{Set}^{\mathrm{op}}$ with the hom-set functor $\mathbf{Set}(\_,G):\mathbf{Set}^{\mathrm{op}} \to \mathbf{Set}$. We will refer to $\mathsf{E}_N(G)$ as an \emph{environment functor}.
\end{definition}

\begin{example}[Genome]
If we let $G$ denote the set $\{\mathtt{A},\mathtt{C},\mathtt{G},\mathtt{T}\}$, then the image of the segment $(0,0,1,1,1,0,1,0,1)$ in $\mathbf{PS}(9)$ through the functor $\mathsf{E}_9(G):\mathbf{PS}(9) \to \mathbf{Set}$ is the set of function $\{3,4,5,7,9\} \to G$, namely the set of $5$-tuples of elements of $G$.
\[
\mathtt{ACGTT} \in \mathsf{E}_9(G)(0,0,1,1,1,0,1,0,1)
\]
Note that, for our combinatorial GWAS, one would prefer to use the Cartesian product $G^{\times 2} = G \times G$ instead of $G$. This is because our genome is made of two copies of the each chromosomes. In this case, the first and second components of $G^{\times 2}$ would each represent a copy of a chromosome, as shown below.
\[
\begin{array}{l}
\xymatrix@R-28pt{
\texttt{ATCGC}\\
\texttt{ATTGA}
}
\end{array}
\in
\begin{array}{l}
\mathsf{E}_9(G^{\times 2})(1,1,1,1,1,0,0,0,0)
\end{array}
\]
In the previous relation, the first columns $\texttt{A}\texttt{A}$ represents a pair $(\texttt{A},\texttt{A})$ in $G^{\times 2}$.
\end{example}

\begin{remark}[Cone property]\label{rem:cone-property-environment}
Let $G$ be a set, $N$ be a non-negative integer and $(S,<)$ be a finite strict linear oder. For ever pair of segments $(\tau_1,\tau_2) \in \mathbf{PS}(N)$, Remark \ref{rem:cone-property} give us two arrows $\tau_1+\tau_2 \to \tau_1$ and $\tau_1+\tau_2 \to \tau_2$ that gives us the following cone in $\mathbf{Set}$:
\[
\xymatrix@R-20pt{
&\mathsf{E}_N(G)(\tau_1+\tau_2)\ar[dr]_{\pi_1}\ar[dl]^{\pi_2}&\\
\mathsf{E}_N(G)(\tau_1)&&\mathsf{E}_N(G)(\tau_2)
}
\]
The previous cone of function gives us the following canonical function 
\[
(\pi_1,\pi_2):\mathsf{E}_N(G)(\tau_1+\tau_2) \to \mathsf{E}_N(G)(\tau_1)\times \mathsf{E}_N(G)(\tau_2)
\]
which divides any word in $\mathsf{E}_N(G)(\tau_1+\tau_2)$ into two subwords respecting the specification of the segments $\tau_1$ and $\tau_2$. For instance, we have the following mapping if we take $G = \{\texttt{A},\texttt{C},\texttt{G},\texttt{T}\}$.
\[
\begin{array}{llclclc}
(\pi_1,\pi_2)&:&\mathsf{E}_N(G)(1,1,1,1,1,1) &\to& \mathsf{E}_N(G)(1,1,1,0,0,0)&\times& \mathsf{E}_N(G)(0,1,0,1,1,1)\\
&&\quad\quad\quad\texttt{AGTAGT} &\mapsto &\quad\quad\quad\texttt{AGT\_\_\_}&,& \quad\quad\quad\texttt{\_G\_AGT}
\end{array}
\]
\end{remark}

\begin{proposition}[Pedigrad property]\label{prop:pedigrad-property}
Let $G$ be a set, $N$ be a non-negative integer and $(S,<)$ be a finite strict linear oder. For ever exact pair of segments $(\tau_1,\tau_2) \in \mathbf{PS}(N)$, then the canonical function $(\pi_1,\pi_2):\mathsf{E}_N(G)(\tau_1+\tau_2) \to \mathsf{E}_N(G)(\tau_1)\times \mathsf{E}_N(G)(\tau_2)$ of Remark \ref{rem:cone-property} is an isomorphism.
\end{proposition}
\begin{proof}
By Proposition \ref{prop:truncations-of-sums}, we have the identity $\mathsf{Tr}(\tau_1+\tau_2) = \mathsf{Tr}(\tau_1) \cup \mathsf{Tr}(\tau_2)$. In addition, because $(\tau_1,\tau_2)$ is exact, the  intersection $\mathsf{Tr}(\tau_1) \cap \mathsf{Tr}(\tau_2) = \tau_1^{-1}(1) \cap \tau_2^{-1}(1)$ holds (See Definition \ref{def:truncation} and Definition \ref{def:exactness}). Hence, the set $\mathsf{Tr}(\tau_1+\tau_2)$ is a coproduct $\mathsf{Tr}(\tau_1) + \mathsf{Tr}(\tau_2)$ and the arrows $\mathsf{Tr}(\tau_1+\tau_2) \to \mathsf{Tr}(\tau_1)$ and $\mathsf{Tr}(\tau_1+\tau_2) \to \mathsf{Tr}(\tau_1)$ are the associated coproduct arrows (see Proposition \ref{prop:truncation-functor}). Since the hom-set $\mathbf{Set}(\_,G)$ sends coproducts to products, it follows from Definition \ref{def:environment-functor} that the pair $(\tau_1,\tau_2)$ (see Remark \ref{rem:cone-property}) defines a coproduct in $\mathbf{Set}$. As a result, the canonical arrow $\mathsf{E}_N(G)(\tau_1+\tau_2) \to \mathsf{E}_N(G)(\tau_1)\times \mathsf{E}_N(G)(\tau_2)$ is an isomorphism.
\end{proof}

Definition \ref{def:genotypes}, given below, encodes the genotypes of a given population as a function going from a finite set to the image of the environment functor on the plain segment of Definition \ref{def:global-plain-segments}. Semantically, we use this segment to mean that the mapping covers the whole genome. In Convention \ref{conv:genotype-transformation} and Proposition \ref{prop:genotype-transformation}, we send these genotypes to smaller segments by extending the mapping of Definition \ref{def:genotypes} to a natural transformation over the category $\mathbf{PS}(N)$.

\begin{definition}[Genotypes]\label{def:genotypes}
Let $G$ be a set, $N$ be a non-negative integer and $(S,<)$ be a finite strict linear oder. We define a \emph{genotype morphism} for $S$ over $G$ as a function $\gamma:S \to \mathsf{E}_N(G^{\times 2})(\mathbf{1}_N)$. The set of such functions will be denoted as $\mathsf{Gen}_N(S,G)$.
\end{definition}

\begin{convention}[Genotype transformations]\label{conv:genotype-transformation}
Let $G$ be a set, $N$ be a non-negative integer and $(S,<)$ be a finite strict linear oder. For every genotype morphism $\gamma:S \to \mathsf{E}_N(G^{\times 2})(\mathbf{1}_N)$, we denote by $\gamma_{\tau}$ the composition of $\gamma$ with the function $\mathsf{E}_N(G^{\times 2})(\mathbf{1}_N) \to \mathsf{E}_N(G^{\times 2})(\tau)$ induced by the unique arrow $!_{\tau}:\mathbf{1}_N \to \tau$ of $\mathbf{PS}(N)$ (see Proposition \ref{prop:global-plain-segments}).
\end{convention}

\begin{proposition}[Genotype transformations]\label{prop:genotype-transformation}
Let $G$ be a set, $N$ be a non-negative integer, $(S,<)$ be a finite strict linear oder and let $\gamma$ be a genotype morphism in $\mathsf{Gen}_N(S,G)$. The collection of functions $\gamma_{\tau}:S \to \mathsf{E}_N(G^{\times 2})(\tau)$ is natural in $\tau$.
\end{proposition}
\begin{proof}
It follows from Proposition \ref{prop:global-plain-segments} that, for every arrow $f:\tau_1 \to \tau_2$, the relation $f \circ !_{\tau_1}  = !_{\tau_2}$ holds. By Convention \ref{conv:genotype-transformation}, this means that the following diagram commutes, which shows the statement.
\[
\xymatrix@R-20pt{
&&\mathsf{E}_N(G^{\times 2})(\tau_1)\ar[dd]^{\mathsf{E}_N(G^{\times 2})(f)}\\
S \ar[r]^-{\gamma} & \mathsf{E}_N(G^{\times 2})(\mathbf{1}_N) \ar[ru]\ar@{}@<+1ex>[ru]^<{\mathsf{E}_N(G^{\times 2})(!_{\tau_1})}\ar[rd]\ar@{}@<-1ex>[rd]_<{\mathsf{E}_N(G^{\times 2})(!_{\tau_2})} &\\
&&\mathsf{E}_N(G^{\times 2})(\tau_2)
}
\]
\end{proof}

The following definition uses the universal construction of Proposition \ref{prop:factorization_system} to turn the natural transformation of Proposition \ref{prop:genotype-transformation} into a functor taking its values in a category of unlabeled partitions.

\begin{convention}[Genotype partitioning]\label{conv:genotype-partitioning}
Let $G$ be a set, $N$ be a non-negative integer, $(S,<)$ be a finite strict linear oder and let $\gamma$ be a genotype morphism in $\mathsf{Gen}_N(S,G)$. We will denote by $P_{\gamma}$ the functor $\mathbf{PS}(N) \to \mathbf{UP}(S)$ induced by the mapping:
\[
\tau \mapsto R\big(\mathsf{e}(\gamma_{\tau}):S \to \mathsf{Im}(\gamma_{\tau})\big).
\] 
The functoriality of $P_{\gamma}$ follows from the functorial property of Proposition \ref{prop:genotype-transformation}, the universal property of Proposition \ref{prop:factorization_system} and the functoriality property of Proposition \ref{prop:reflection_functor_partitions}.
\end{convention}

The following example shows that the functor of Convention \ref{conv:genotype-partitioning} essentially clusters individuals with respect to their genotypes at given locations in the genome.

\begin{example}[Genotype partitioning]\label{exa:genotype-partitioning}
Let $G$ denote the set $\{\mathtt{A},\mathtt{C},\mathtt{G},\mathtt{T}\}$, let $S$ denote the finite strict linear order $\{\texttt{Angie}<\texttt{Bob}<\texttt{Charles}<\texttt{Doug}<\texttt{Eric}\}$ and let $\gamma$ denote the function shown below, on the left. The images of the corresponding functor $P_{\gamma}:\mathbf{PS}(5) \to \mathbf{UP}(S)$ on the objects of the form $(0,\dots,0,1,0,\dots,0)$ are as shown on the right:
\[
\begin{array}{lll}
S&\to&\mathsf{E}_{5}(G^{\times 2})(\mathbf{1}_5)\\
\texttt{Angie}&\mapsto&\mathtt{ATTCG}\\
\texttt{Bob}&\mapsto&\mathtt{TTTCC}\\
\texttt{Charles}&\mapsto&\mathtt{ATTGG}\\
\texttt{Doug}&\mapsto&\mathtt{ATACC}\\
\texttt{Eric}&\mapsto&\mathtt{ATTCG}\\
\end{array}
\quad
\left\{
\begin{array}{l}
P_{\gamma}(1,0,0,0,0) = \{\texttt{Angie},\texttt{Charles},\texttt{Doug},\texttt{Eric}\},\{\texttt{Bob}\}\\
P_{\gamma}(0,1,0,0,0) = \{\texttt{Angie},\texttt{Bob},\texttt{Charles},\texttt{Doug},\texttt{Eric}\}\\
P_{\gamma}(0,0,1,0,0) =  \{\texttt{Angie},\texttt{Bob},\texttt{Charles},\texttt{Eric}\},\{\texttt{Doug}\}\\
P_{\gamma}(0,0,0,1,0) =  \{\texttt{Angie},\texttt{Bob},\texttt{Doug},\texttt{Eric}\},\{\texttt{Charles}\}\\
P_{\gamma}(0,0,0,0,1) = \{\texttt{Angie},\texttt{Charles},\texttt{Eric}\},\{\texttt{Bob},\texttt{Doug}\}\\
\end{array}
\right.
\]
We now give other examples of images of $P_{\gamma}$ for objects of $\mathbf{PS}(5)$ that can be described as sums of objects of the form $(0,\dots,0,1,0,\dots,0)$ (see Definition \ref{def:sums:segments}).
\[
\left\{
\begin{array}{l}
P_{\gamma}(1,0,1,0,0) = \{\texttt{Angie},\texttt{Charles},\texttt{Eric}\},\{\texttt{Bob}\},\{\texttt{Doug}\}\\
P_{\gamma}(0,0,0,1,1) = \{\texttt{Angie},\texttt{Eric}\},\{\texttt{Bob},\texttt{Doug}\},\{\texttt{Charles}\}\\
P_{\gamma}(0,0,0,0,0) = \{\texttt{Angie},\texttt{Charles},\texttt{Doug},\texttt{Eric},\texttt{Bob}\}\\
P_{\gamma}(1,1,1,1,1) = \{\texttt{Angie},\texttt{Eric}\},\{\texttt{Bob}\},\{\texttt{Charles}\},\{\texttt{Doug}\}\\
\end{array}
\right.
\]
\end{example}

The property shown in the following proposition will be central to the algorithm described in section \ref{ssec:combinatorial-GWAS}. In essence, this type of property makes it possible to use a divide-and-conquer paradigm to analyze the genome of a set of individuals.

\begin{proposition}[Pedigrad property]\label{prop:pedigrad-property:genotype-partition}
Let $G$ be a set, $N$ be a non-negative integer,  $(S,<)$ be a finite strict linear oder and let $\gamma$ be a genotype morphism in $\mathsf{Gen}_N(S,G)$. For every exact sum $\tau_1+\tau_2$ in $\mathbf{PS}(N)$, the canonical arrow $P_{\gamma}(\tau_1+\tau_2) \to P_{\gamma}(\tau_1) \times P_{\gamma}(\tau_2)$ is an identity.
\end{proposition}
\begin{proof}
First, it follows from Proposition \ref{prop:pedigrad-property} and Proposition \ref{prop:genotype-transformation} that the canonical arrow $\gamma_{\tau_1+\tau_2} \to \gamma_{\tau_1} \times \gamma_{\tau_2}$ is an isomorphism. By Proposition \ref{prop:factorization_system}, this means that the canonical arrow $\mathsf{e}(\gamma_{\tau_1+\tau_2}) \to \mathsf{e}(\gamma_{\tau_1} \times \gamma_{\tau_2})$ is an isomorphism in $\mathbf{LP}(S)$. Then, Proposition \ref{prop:products_labeled_partitions} shows that the object $\mathsf{e}(\gamma_{\tau_1} \times \gamma_{\tau_2})$ is the product $\mathsf{e}(\gamma_{\tau_1}) \times \mathsf{e}(\gamma_{\tau_2})$ in $\mathbf{LP}(S)$. As a result, Proposition \ref{prop:adjoint_equivalence:products} implies that $R\mathsf{e}(\gamma_{\tau_1} \times \gamma_{\tau_2})$ is the product $R\mathsf{e}(\gamma_{\tau_1}) \times R\mathsf{e}(\gamma_{\tau_2})$ in $\mathbf{UP}(S)$. Finally, we use Proposition \ref{prop:skeletal} to show that the image of the isomorphism $\mathsf{e}(\gamma_{\tau_1+\tau_2}) \to \mathsf{e}(\gamma_{\tau_1}) \times \mathsf{e}(\gamma_{\tau_2})$ through the functor $R:\mathbf{LP}(S) \to \mathbf{UP}(S)$ is an identity of the form $R\mathsf{e}(\gamma_{\tau_1+\tau_2}) \to R\mathsf{e}(\gamma_{\tau_1}) \times R\mathsf{e}(\gamma_{\tau_2})$ in $\mathbf{UP}(S)$.
\end{proof}

\subsection{Combinatorial GWAS}\label{ssec:combinatorial-GWAS}
In this last section, we define the concept of combinatorial association (see Definition \ref{def:combinatorial-association}) and explain, through an example (see Example \ref{exa:combinatorial-association}) how this concept relates to combinatorial GWAS in the sense of section \ref{ssec:motivations}. Finally, we show how the pedigrad $P_{\gamma}$ of section \ref{ssec:genotypes} can be used to construct combinatorial associations and we use the resulting procedure to design our combinatorial GWAS algorithm.

\begin{definition}[Combinatorial association]\label{def:combinatorial-association}
Let $G$ be a set, $N$ be a non-negative integer, $(S,<)$ be a finite strict linear oder and $\gamma$ be a genotype morphism in $\mathsf{Gen}_N(S,G)$. For every unlabeled partition $b$ in $\mathbf{UP}(S)$, we define a \emph{combinatorial association for $(b,\gamma)$} as a pair $(e,\tau)$ where $e \in \mathbb{V}(P_{\gamma}(\mathbf{1}_N),b)$ and $\tau \in\mathbf{PS}(N)$ such that
\begin{itemize}
\item[1)] (\emph{external factors}) $e$ is a minimal solution of $\mathbb{V}(P_{\gamma}(\mathbf{1}_N),b)$;
\item[2)] (\emph{localization}) $e$ is a solution of $\mathbb{V}(P_{\gamma}(\tau),b)$;
\item[3)] (\emph{minimality}) if there exists an arrow $f:\tau \to \tau'$ in $\mathbf{UP}(S)$ such that $e \in \mathbb{V}(P_{\gamma}(\tau'),b)$, then $f$ is an identity.
\end{itemize}
\end{definition}

\begin{example}[Combinatorial associations]\label{exa:combinatorial-association}
We illustrate the concept of combinatorial association by continuing Example \ref{exa:genotype-partitioning} -- we will use the notations therein considered. 

Suppose that we can associate each individual of $S$ with a phenotype value such that we can partition the set $S$ with respect to these values in terms of a partition $b$. The idea of a combinatorial association $(e,\tau)$ for the partition $b$ and the genotype morphism $\gamma \in\mathsf{Gen}_5(S,G)$ -- see Example \ref{exa:genotype-partitioning} -- is to compare the phenotypic clustering induced by $b$ to the genotypic clustering given by the images of the functor $P_{\gamma}$ to find minimal combinations of nucleotides that explain the phenotypic clustering of $S$. From this point of view, the partition $e$ represents the influence of non-genetic factors potentially causing the phenotypes and the plain segments $\tau$ localizes those SNPs that are likely to cause the phenotypic variation.

Let us give an example. Suppose that our set $S$ of individuals contains $\mathtt{healthy}$ and $\mathtt{diseased}$ individuals. This can be specified through a surjection as follows:
\[
y:
\left(
\begin{array}{lll}
S&\to&\{\mathtt{healthy},\mathtt{diseased}\}\\
\texttt{Angie}&\mapsto&\mathtt{healthy}\\
\texttt{Bob}&\mapsto&\mathtt{diseased}\\
\texttt{Charles}&\mapsto&\mathtt{healthy}\\
\texttt{Doug}&\mapsto&\mathtt{diseased}\\
\texttt{Eric}&\mapsto&\mathtt{diseased}.\\
\end{array}
\right)
\]
Sending the corresponding labeled partition $y:S \to \{\mathtt{healthy},\mathtt{diseased}\}$ through the functor $R:\mathbf{LP}(S) \to \mathbf{UP}(S)$ gives us the following unlabeled partition:
\[
b = \{\texttt{Angie},\texttt{Charles}\},\{\texttt{Bob},\texttt{Doug},\texttt{Eric}\}
\]
According to Example \ref{exa:genotype-partitioning}, the image $P_{\gamma}(\mathbf{1}_5)$ is equal to the partition  
\[
\{\texttt{Angie},\texttt{Eric}\},\{\texttt{Bob}\},\{\texttt{Charles}\},\{\texttt{Doug}\},
\]
which implies that the minimal solutions $e \in \mathbb{V}(P_{\gamma}(\mathbf{1}_5),b)$ can be as follows:
\[
\begin{array}{|l|c|}
\hline
\multicolumn{2}{|c|}{\cellcolor[gray]{.8}\{\texttt{Angie},\texttt{Eric}\},\{\texttt{Bob}\},\{\texttt{Charles}\},\{\texttt{Doug}\} \times e \to \{\texttt{Angie},\texttt{Charles}\},\{\texttt{Bob},\texttt{Doug},\texttt{Eric}\}}\\
\hline
e_1 & \{\texttt{Angie},\texttt{Bob},\texttt{Charles},\texttt{Doug}\},\{\texttt{Eric}\}\\
\hline
e_2 & \{\texttt{Angie}\},\{\texttt{Bob},\texttt{Charles},\texttt{Doug},\texttt{Eric}\}\\
\hline
\end{array}
\]
Hence, we can construct combinatorial associations of the form $(e_1,\tau)$ and $(e_2,\tau)$. Below, we give possible values for $\tau$, which turn out to be the same for both partitions $e_1$ and $e_2$. We can check that each plain segment $\tau$ satisfies the conditions of Definition \ref{def:combinatorial-association}.
\[
\begin{array}{|l|c|}
\hline
\multicolumn{2}{|c|}{\cellcolor[gray]{.8}\textrm{finding $\tau$ for the solution }e_1}\\
\hline
\multicolumn{2}{|c|}{\cellcolor[gray]{.8}P_{\gamma}(\tau) \times \{\texttt{Angie},\texttt{Bob},\texttt{Charles},\texttt{Doug}\},\{\texttt{Eric}\} \to \{\texttt{Angie},\texttt{Charles}\},\{\texttt{Bob},\texttt{Doug},\texttt{Eric}\}}\\
\hline
P_{\gamma}(1,0,1,0,0) &  \{\texttt{Angie},\texttt{Charles},\texttt{Eric}\},\{\texttt{Bob}\},\{\texttt{Doug}\}\\
\hline
P_{\gamma}(0,0,0,0,1) & \{\texttt{Angie},\texttt{Charles},\texttt{Eric}\},\{\texttt{Bob},\texttt{Doug}\}\\
\hline
\multicolumn{2}{|c|}{\cellcolor[gray]{.8}\textrm{finding $\tau$ for the solution }e_2}\\
\hline
\multicolumn{2}{|c|}{\cellcolor[gray]{.8}P_{\gamma}(\tau) \times \{\texttt{Angie}\},\{\texttt{Bob},\texttt{Charles},\texttt{Doug},\texttt{Eric}\} \to \{\texttt{Angie},\texttt{Charles}\},\{\texttt{Bob},\texttt{Doug},\texttt{Eric}\}}\\
\hline
P_{\gamma}(1,0,1,0,0) &  \{\texttt{Angie},\texttt{Charles},\texttt{Eric}\},\{\texttt{Bob}\},\{\texttt{Doug}\}\\
\hline
P_{\gamma}(0,0,0,0,1) & \{\texttt{Angie},\texttt{Charles},\texttt{Eric}\},\{\texttt{Bob},\texttt{Doug}\}\\
\hline
\end{array}
\]
To conclude, if we denote as $\tau_i$ the segment $(0,\dots,0,1,0,\dots,0)$ for which $\tau_i(i) = 1$, then the pairs $(e_1,\tau_1+\tau_3)$, $(e_1,\tau_5)$, $(e_2,\tau_1+\tau_3)$, and $(e_2,\tau_5)$ define combinatorial associations for $(b,\gamma)$.
\end{example}

Let $(S,<)$ be a finite strict alinear order. In the rest of this section, we give a heuristic to construct combinatorial associations for a given genotype morphism $\gamma$ in $\mathsf{Gen}_N(S,G)$ and an unlabeled partition $b:S \to [n]$ for which 
\begin{itemize}
\item[-] there is a function $y:S \to \mathbb{R}$ (associating every individual $i \in S$ with a phenotype);
\item[-] and a surjection $c:\mathsf{Im}(y) \to [n]$ (that categorizes the images of $y$ into groups);
\end{itemize}
such that $b$ is equal to the image of the labeled partition $c \circ \mathsf{e}(y)$ (see Proposition \ref{prop:factorization_system}) through the functor $R:\mathbf{LP}(S) \to \mathbf{UP}(S)$ (see Proposition \ref{prop:reflection_functor_partitions}), meaning that $b = R(c \circ \mathsf{e}(y))$.

Our strategy for designing our heuristic is to take advantage of the ``pedigrad property'' satisfied by the genotype partitioning functor $P_{\gamma}$ (see Proposition \ref{prop:pedigrad-property:genotype-partition}). More specifically, for a given minimal solution $e \in \mathbb{V}(P_{\gamma}(\mathbf{1}_N),b)$, we propose to find the plain segment $\tau$ of a combinatorial association $(e,\tau)$ by decomposing the plain segment $\mathbf{1}_N$ into an exact sum $\tau_1+\dots+\tau_k$ of plain segments in $\mathbf{PS}(N)$ and to find its smallest sub-sum $\tau_{i_1}+\dots+\tau_{i_h}$ for which the following arrow exists:
\begin{equation}\label{eq:pedigrad-causal-sum}
P_{\gamma}(\tau_{i_1}+\dots+\tau_{i_h}) \times e \to b.
\end{equation}
We find each segment $\tau_{i_j}$ by using the pedigrad property of Proposition \ref{prop:pedigrad-property:genotype-partition} on the exact sum $\mathbf{1}_N = \tau_1+\dots+\tau_k$. Specifically, for every $i \in [k]$, the morphism $P_{\gamma}(\mathbf{1}_N) \times e \to y$ is equal to the following morphism:
\[
P_{\gamma}(\tau_{i}) \times P_{\gamma}(\sum_{j \in [k]\backslash\{i\}}\tau_j) \times e \to b.
\]
Since we can expect the solutions of $\mathbb{V}(P_{\gamma}(\tau_{i}),b)$ to all share a common basis, it is reasonable to expect any solution $x_i \in \mathbb{V}(P_{\gamma}(\tau_{i}),b)$ to give us some information regarding what the solution 
\begin{equation}\label{eq:CGWAS:complement-tau-i}
e_{i} = P_{\gamma}(\sum_{j \in [k]\backslash\{i\}}\tau_j)
\end{equation}
looks like. To make sure that the solution $x_i$ does not overfit the partitioning structure of $b$ (see Proposition \ref{prop:existence:embedding-solutions}), so that it only assesses the essential contribution of partition (\ref{eq:CGWAS:complement-tau-i}) to the existence of morphism (\ref{eq:pedigrad-causal-sum}), we can take $x_i$ minimal in $\mathbb{V}(P_{\gamma}(\tau_{i}),b)$. Since $\mathbb{V}(P_{\gamma}(\tau_{i}),b)$ may contain several minimal solutions, we can take $x_i$ such that its associated $\mathsf{F}$-ratio $\mathsf{F}_{y}(x_i)$ is maximal\footnote{It is well known that this value is greater than or equal to 1 -- see \cite{Oehlert}}, meaning that the solution $x_i$ accounts for the most significant statistical structure needed to generate morphism (\ref{eq:pedigrad-causal-sum}). As a result, if the $\mathsf{F}$-ratio of the minimal solution $x_i$ turns out to be low compared to the $\mathsf{F}$-ratios of the other solutions $x_{j}$, then our construction ensures that the partition $e_i$, given in (\ref{eq:CGWAS:complement-tau-i}), is not as essential as the other solutions $e_{j}$ for the existence of morphism (\ref{eq:pedigrad-causal-sum}). This means that the partition $P_{\gamma}(\tau_{i})$ is more likely to explain the partitioning structure of $y$ than the other partitions $P_{\gamma}(\tau_{j})$.
\[
\begin{array}{|c|c|c|}
\hline
\multicolumn{1}{|c|}{\cellcolor[gray]{.8}\mathsf{F}_{y}(x_i)}&\multicolumn{1}{c|}{\cellcolor[gray]{.8}x_i}&\multicolumn{1}{c|}{\cellcolor[gray]{.8}P_{\gamma}(\tau_{i})}\\
\hline
\textrm{relatively low} & \textrm{conttribution does not seem to be essential to $y$} &\textrm{more likely to explain $y$}\\
\hline
\textrm{relatively large} & \textrm{conttribution seems essential to $y$} &\textrm{less likely to explain $y$}\\
\hline
\end{array}
\]
Our previous discussion suggests an obvious procedure, which we formalize in terms of pseudo-code. Our algorithm will make use of two intermediate algorithms, described beforehand.

The goal of our first algorithm is to find, for every embedding problem $(a,b)$, a minimal solution $x \in \mathbb{V}(a,b)$ such that the $\mathsf{F}$-ratio $\mathsf{F}_y(x)$ is maximal among all the other solutions of $\mathbb{V}(a,b)$. We propose the heuristic $\mathsf{A}_1(a,b)$, given below, whose steps are justified by the statements of Theorem \ref{theo:minimal-solutions} and Proposition \ref{prop:changes-F-ratio} (see the explanation given afterwards).

\vspace{4pt}
\noindent\textsf{%
\begin{tabular}{p{.35cm}|p{14.92cm}} 
\hline
\multicolumn{2}{c}{\cellcolor[gray]{.8}$\mathsf{A}_1$}\\
\hline
1& \textbf{Input:} $(a,b)$ where $a:S \to [n],b:S \to [m] \in \mathbf{UP}(S)$\\
2& \textbf{Construct} a function $r:[n] \to [m]$ for which $(a_i \subseteq b_j \Rightarrow r(i) = j)$ (see Theorem \ref{theo:minimal-solutions})\\
3& \textbf{Construct} a reduction $u = \{u_1,\dots,u^*\}$ of $\chi(r,a,b)$ such that $\varepsilon_{y}(u_k \to u_{k+1})$ is minimal\\
4& \textbf{Return} $u^*$\\
\hline
\end{tabular}
}
\vspace{4pt}

\noindent
By Proposition \ref{prop:changes-F-ratio}, line \textsf{3} of algorithm $\mathsf{A}_1$ implies that the $\mathsf{F}$-ratio $\mathsf{F}_y(u_{k+1})$ is likely to stay close to the value $\nu(u_{k} \to u_{k+1})\cdot \mathsf{F}_y(u_{k}) \simeq 1 \cdot \mathsf{F}_y(u_{k})$, meaning that the $\mathsf{F}$-ratio $\mathsf{F}_y(u^{*})$ should be rather maximal compared to all the other constructions of reduced elements. In particular, because the arrow $u_{k} \to u_{k+1}$ is a contraction (see Definition \ref{def:contraction}), Proposition \ref{prop:contractions-codegeneracy-morphisms} allows us to use the formula of Proposition \ref{prop:between-group-change:codegeneracy} to directly determine the contraction $u_{k+1}$ of $u_k$ whose associated quantity $\varepsilon_{y}(u_k \to u_{k+1})$ is minimal among the other possible contractions of $u_k$.

The goal of our second algorithm is to return, for every unlabeled partition $b$ in $\mathbf{UP}(S)$, every genotype morphism $\gamma$ in $\mathsf{Gen}_N(S,G)$, every plain segment $\tau \in \mathbf{Seg}(N)$ and every solution $e \in \mathbb{V}(P_{\gamma}(\tau),b)$, a segment $\tau'$ for which there is an arrow $\tau \to \tau'$ in $\mathbf{Seg}(N)$ and such that $e \in \mathbb{V}(P_{\gamma}(\tau'),b)$. We propose the heuristic $\mathsf{A}_2(N,b,\gamma,\tau,e)$ given below.

\vspace{4pt}
\noindent\textsf{%
\begin{tabular}{p{.35cm}|p{14.92cm}} 
\hline
\multicolumn{2}{c}{\cellcolor[gray]{.8}$\mathsf{A}_2$}\\
\hline
1& \textbf{Input:} $(N,b,\gamma,\tau,e)$ where $b \in \mathbf{UP}(S)$, $\gamma\in \mathsf{Gen}_N(S,G)$, $\tau \in \mathbf{Seg}(N)$, and $e \in \mathbb{V}(P_{\gamma}(\tau),b)$\\
2& \textbf{Decompose} $\tau$ as an exact sum $\tau_1+\dots+\tau_k$\\
3& \textbf{For} every $i \in \{1,\dots,k\}$ \textbf{do:}\\
4& \quad\textbf{Take} $x_i = \mathsf{A}_1(P_{\gamma}(\tau_i)\times e,b)$\\
5& \textbf{Order} $\{1,\dots,k\}$ as $\{i_1\dots,i_k\}$ such that $\mathsf{F}_{y}(x_{i_j}) \leq \mathsf{F}_{y}(x_{i_{j+1}})$ for every $j \in [k]$ \\
6& \textbf{Take} $h = \mathsf{min}\{j~|~ e \in \mathbb{V}(P_{\gamma}(\tau_{i_1}+\dots+\tau_{i_j}),b)\}$\\
7& \textbf{Return} $\tau_{i_1}+\dots+\tau_{i_h}$\\
\hline
\end{tabular}
}
\vspace{4pt}

Finally, for every unlabeled partition $b$ in $\mathbf{UP}(S)$, every genotype morphism $\gamma$ in $\mathsf{Gen}_N(S,G)$, and every positive integer $n$, we propose the heuristic $\mathsf{A}_3(N,b,\gamma,n)$, which essentially consists in applying $n$ times the mapping $(\tau,e) \mapsto (\mathsf{A}_2(N,b,\gamma,\tau,e),e)$ on itself.

\vspace{4pt}
\noindent\textsf{%
\begin{tabular}{p{.35cm}|p{14.92cm}} 
\hline
\multicolumn{2}{c}{\cellcolor[gray]{.8}$\mathsf{A}_3$}\\
\hline
1& \textbf{Input:} $(N,b,\gamma,n)$ where $b \in \mathbf{UP}(S)$, $\gamma\in \mathsf{Gen}_N(S,G)$, and $n \geq 1$\\
2& \textbf{If} $n = 1$ \textbf{then:} \\
3& \quad\textbf{Take} $e =\mathsf{A}_1(P_{\gamma}(\mathbf{1}_N),b)$ \\
4& \quad\textbf{Return} $(\mathsf{A}_2(N,b,\gamma,\mathbf{1}_N,e),e)$\\
5& \textbf{Else:} \\
6&\quad\textbf{Take} $(\tau,e) = \mathsf{A}_3(N,b,\gamma,n-1)$ \\
7&\quad\textbf{Return} $(\mathsf{A}_2(N,b,\gamma,\tau,e),e)$\\
\hline
\end{tabular}
}



\bibliographystyle{plain}

\end{document}